\documentclass[a4paper]{article}
\usepackage[english]{babel}
\usepackage[utf8]{inputenc}
\usepackage{amsmath,amsthm}
\usepackage{amsfonts,amssymb}

\usepackage{mathtools}
\usepackage{enumitem}

\usepackage{graphicx} 

\usepackage[colorlinks=true,linktocpage=true,linkcolor=black,citecolor=black]{hyperref}
\usepackage[capitalise]{cleveref}

\usepackage{tikz,tikz-cd} 
 \usetikzlibrary{decorations.pathmorphing} 

\usepackage{xifthen} 

\swapnumbers

\theoremstyle{plain}
\newtheorem{theorem}{Theorem}[subsection]
\newtheorem{lemma}[theorem]{Lemma}
\newtheorem{prop}[theorem]{Proposition}
\newtheorem{cor}[theorem]{Corollary}

\newtheorem{retract_lemma}[theorem]{Retract lemma}

\theoremstyle{remark}

\theoremstyle{definition}
\newtheorem{definition}[theorem]{Definition}
\newtheorem{remark}[theorem]{Remark}

\newtheorem{notation}[theorem]{Notation}

\newtheorem{example}[theorem]{Example}
\newtheorem{construction}[theorem]{Construction}
\newtheorem{assumption}[theorem]{Assumption}

\newtheorem{QuillenSOA}[theorem]{Quillen's small object argument}
\newtheorem{GarnerSOA}[theorem]{Garner's small object argument}
\newtheorem{EasySOA}[theorem]{The good case of the small object argument}
\newtheorem{proof1}[theorem]{Proof of \cref{th:KanQuillen_WMS}}

\newcommand{\G}{\mathbb{G}}

\newcommand{\Ib}{\mathbb{I}}

\newcommand{\Acal}{\mathcal{A}} 
 
\newcommand{\Ecal}{\mathcal{E}}

\newcommand{\Ical}{\mathcal{I}}

\newcommand{\Scal}{\mathcal{S}} 
\newcommand{\Dcal}{\mathcal{D}} 
\newcommand{\Fcal}{\mathcal{F}}

\newcommand{\Jcal}{\mathcal{J}} 
\newcommand{\Kcal}{\mathcal{K}}

\newcommand{\Wcal}{\mathcal{W}} 
\newcommand{\Xcal}{\mathcal{X}} 
\newcommand{\Ccal}{\mathcal{C}}

\newcommand{\cof}{\textsc{Cof}}
\newcommand{\fib}{\textsc{Fib}}

\newcommand{\cofe}{\text{cof}}
\newcommand{\fibe}{\text{fib}}
\newcommand{\bfe}{\text{bf}}
\newcommand{\bothe}{c \cup f}
\newcommand{\scomp}[1]{\overline{#1}}

\DeclareMathOperator*{\fprod}{\scalebox{1.4}{$\times$}}

\newcommand{\corner}[1]{\,\tikz[baseline=(todotted.base)]{
        \node[inner sep=1pt,outer sep=0pt] (todotted) {$#1$};
        \draw (todotted.north west) -- (todotted.north east);
        \draw (todotted.south west) -- (todotted.north west);
    }}

\newcommand{\cornerl}[1]{\corner{#1}}

\newcommand{\refl}[1]{\text{refl}_{#1}}
\newcommand{\inv}[1]{\text{inv}(#1)}
\newcommand{\comp}[2]{#1 \circ #2}

\begin{document}


\pagestyle{plain}
\title{Weak model categories in classical and constructive mathematics}
\date{}

\author{Simon Henry}

\maketitle

\begin{abstract}
We introduce a notion of ``weak model category'' which is a weakening of the notion of Quillen model category, still sufficient to define a homotopy category, Quillen adjunctions, Quillen equivalences, and most of the usual construction of categorical homotopy theory. Both left and right semi-model categories are weak model categories, and the opposite of a weak model category is again a weak model category. 

The main advantages of weak model categories is that they are easier to construct than Quillen model categories. In particular we give some simple criteria on two weak factorization systems for them to form a weak model category. The theory is developed in a very weak constructive, even predicative, framework and we use it to give constructive proofs of the existence of  weak versions of various standard model categories, including the Kan-Quillen model structure, Lurie's variant of the Joyal model structure on marked simplicial sets, and the Verity model structure for weak complicial sets. We also construct semi-simplicial versions of all these.
\end{abstract}

\renewcommand{\thefootnote}{\fnsymbol{footnote}} 
\footnotetext{\emph{Keywords.} Model categories, constructive mathematics, simplicial sets, semi-simplicial sets, complicial sets.}
\footnotetext{\emph{2010 Mathematics Subject Classification.} 55U35,55U40,18G30,  }
\renewcommand{\thefootnote}{\arabic{footnote}}



\tableofcontents

\section{Introduction and preliminaries}

\subsection{Introduction}

Quillen model category constitute, since their introduction by D.~Quillen in \cite{quillen1967homotopical}, one of the main frameworks for categorical homotopy theory. Let us recall their definition (in a simplified form which can be found in \cite{riehl2009concise}):

\begin{definition} A \emph{Quillen model category} is a complete and cocomplete category with three classes of morphisms $\Wcal$ (equivalences), $\fib$ (fibrations), and $\cof$ (cofibrations), such that:

\begin{enumerate}[label=(\roman*)]
\item $\Wcal$ satisfies 2-out-of-3 and contains the isomorphisms.

\item $(\Wcal \cap \cof,\fib)$ is a weak factorization system.

\item $(\cof,\fib \cap \Wcal)$ is a weak factorization system.
\end{enumerate}

\end{definition}

It has been realized more recently that for many examples, some parts of this structure, which are actually not useful in practice, are difficult, or even impossible to obtain. This has motivated the introduction of several weakenings of this definition:

The notion of \emph{left semi-model category}\footnote{First introduced in \cite{spitzweck2001operads} under the name $J$-semi model structure.} weakens axiom (ii) by:
\begin{itemize}[noitemsep,topsep=2pt]
\item Only requiring that arrows with a cofibrant domain (instead of all arrows) can be factored as an acyclic cofibration followed by a fibration. 
\item Only requiring that acyclic cofibrations with cofibrant domain (instead of all acyclic cofibrations) have the left lifting property against fibrations.
\item The stability under retract of fibrations and acyclic cofibrations, no longer automatic, is also required.
\end{itemize}
The dual notion of \emph{right semi-model category}\footnote{Introduced in \cite{barwick2010left}.} is obtained by instead weakening axiom (iii), restricting the existence of factorizations to arrows with fibrant target, and the lifting property to acyclic fibrations with fibrant target. We will not give many examples of such structures, for which we refer to \cite{spitzweck2001operads}, \cite{barwick2010left} or \cite{batanin2020Bousfield}. The introduction of \cite{batanin2020Bousfield} contains an up-to-date bibliography on the topic.

\bigskip

In this paper we introduce a new notion called ``weak model category'' where we instead restrict the two weak factorization systems ``on both sides'', i.e. we only require that arrows with both a cofibrant domain and a fibrant target can be factored, and only ask for the lifting property between (acyclic) cofibrations with cofibrant domain and (acyclic) fibration with fibrant targets. The precise definition will be given in Section~\ref{Subsection_DEFofWMS} (Definition~\ref{Def::weak_model_categories}). This is a generalization of the notion of Quillen model category which encompasses both left and right semi-model categories. The notion is self-dual: the opposite of a weak model category is also a weak model category, and it is still sufficient to study and compare different homotopy theories: we will define the homotopy category of a weak model category similarly to that of a Quillen model category, a notion of Quillen adjunction and Quillen equivalence between weak model categories, and so on.  Our guiding principle will be that only the notion of cofibration with cofibrant domain and fibration with fibrant target should be considered meaningful. As a consequence, if $X$ is an object which is neither fibrant nor cofibrant, it is not possible to construct fibrant or cofibrant replacements for $X$, and hence such objects should not be considered when talking about homotopy theoretic properties. For this reason the class $\Wcal$ of equivalences will be a class of arrows in the full subcategory $\Ccal^{\bothe} = \Ccal^{\cofe} \cup \Ccal^{\fibe}$ of objects of $\Ccal$ that are either fibrant or cofibrant.

Initially, the main reason for developing this theory was the appearence of some examples of such structures in my work: The structure defined on the category of pre-cylinder categories in \cite{henry2016algebraic} is a weak model structure, and the main result of \cite{henry2018regular} (published simultaneously with the first version of the present paper) involves constructing and comparing several weak model structures, and relies on results of the present paper.

But while developing the theory of weak model structures it appeared that there are many examples of categories which, even if they do admit a full Quillen model structures, it is considerably easier to construct only a weak model structure and, most of time, this is enough for all the practical applications. This has allowed to tackle another important problem: giving constructive proofs that some classical examples of model structure exist. In the present paper we will show that the projective model structure on chain complexes (Subsection~\ref{subsec:ChainCplex}), the Kan-Quillen model structure on simplicial sets (Subsection~\ref{subsec:KanQuillenMS}), the Joyal-Lurie model structure on marked simplicial sets (Subsection~\ref{subsec:LurieMS}), and the Verity model structure on stratified simplicial set (Subsection~\ref{subsec:complicialsets}) can all be proved to exist constructively as weak model structures, hence opening the door to a constructive theory of higher categories. This has, in the subsequent work \cite{Gambino2019simplicial}, allowed, up to some coherence issue that still needs to be taken care of, to give a constructive version of Voevodsky's simplicial model of homotopy type theory. It was also shown (in \cite{henry2019Kan} and \cite{gambino2019constructive}) that the Kan-Quillen model structure on simplicial sets can be constructively proved to be a proper Quillen model category, but this requires considerably more work, and uses properties very specific to the Kan-Quillen model structure.\footnote{In fact, the author believes a similar result for the Joyal model structures is out of reach at the moment.}

\bigskip

The present paper is mostly focused on the aspect of the theory of weak model categories that can be developed in constructive mathematics. In particular, many classical topics of the theory of Quillen model categories, for example the notion of combinatorial model structure, or the theory of Bousfield localizations, will be ignored in the present paper as they require stronger logical assumptions. A subsequent paper \cite{henry2019CombWMS} studies these non-constructive aspects, and the theory of combinatorial and accessible weak model categories. The precise relation between weak model structures and left and right-semi model structure will also be studied in great details in \cite{henry2019CombWMS}. A reader not interested in constructive aspects can find a more concise introduction to weak model categories from \cite{henry2019CombWMS}.

\bigskip

Another important objective of the present paper is to give several easy criteria for constructing a weak model structure on a category, especially in the case where we start from the cofibrations and the fibrations, but we do not have a good description of the weak equivalences, as this is generally a hard task for Quillen model categories. This is a key point for the constructive examples of weak model structure because, before the present works, it was not known how to define a notion of ``weak equivalences'', for example between simplicial sets, that would allow to construct a model structure. 

\bigskip

\textbf{Acknowledgment:} I would like to thank Nicola Gambino and John Bourke for their many comments and suggestion while I was preparing the second version of this paper. I would also want to thank Harry Gindy and Viktoriya Ozornova who independently pointed out a mistake in the first version of the paper regarding my attempt to give a simpler proof of the corner-product condition for complicial sets.

This work was supported by the Grant agency of the Czech republic under the grant P201/12/G028.

\subsection{Overview of the paper}

The paper is relatively long, but it does not need to be read from the first page to the last in order. The core of the paper consists of Subsections~\ref{Subsection_DEFofWMS} and \ref{subsec:homotopycategory} which contains the basic theory of weak model structures:  their definition and the construction of their homotopy categories. As such they are the only sections that are necessary to read in order to follow the rest of paper.

The rest of Section~\ref{section_general} contains other aspects of the general theory of weak model structures: Subsection~\ref{sub_section_equivalentDefinitions} gives additional criteria to identify weak model categories, Subsection~\ref{subsec:Quillenfunctors} introduces Quillen adjunctions and Quillen equivalences between weak model categories.

Section~\ref{sec_monoidalcase} gives a couple of theorems (\ref{Th_tensorWMS} and \ref{Th::WMS_from_cylinder_path}) allowing to easily construct weak model categories in the presence of a monoidal structure, an enrichment or a well behaved (left adjoint) cylinder functor. These theorems will be our main tools to construct examples of weak model structures. They can be thought of as a version of Cisinski-Olschok theory (as in \cite{olschok2011left}) for weak model categories.

Section~\ref{Section_Examples} deals with two very simple examples of weak model categories (setoids and chain complexes) which might be enlightening for readers unfamiliar with model categories in general.

Section~\ref{sec:simplicialEx} deals with well-known simplicially based examples, but treats them in a completely constructive way (which is mostly new). One of the main differences with classical mathematics is that not all monomorphisms are cofibrations, and not all simplicial sets are cofibrant if one does not assume the law of excluded middle. We will start with the usual Kan-Quillen model structure on simplicial sets, then we treat a variant of the Joyal model structure constructed by Lurie on the category of marked simplicial sets, which we will refer to as the Joyal-Lurie model structure, and more generally we treat the case of the Verity model structure for weak complicial sets (constructed by D.~Verity in \cite{verity2008weak}), which is supposed to give a model for $(\infty,n)$-categories and even $(\infty,\infty)$-categories.

Finally, in Section~\ref{Sec_ex_semi-simplicial sets} we develop ``semi-simplicial'' versions of all these model categories which as far as I know are new even classically. These examples cannot be Quillen model categories and are only right semi-model categories.

Appendix~\ref{section_prelim_setoids} briefly introduces the notion of ``setoids'' and ``setoid-categories'', which most readers will be happy to just replace by ``sets'' and ``categories''.  Appendix~\ref{section_pisetoids} uses these setoids to give tools to obtain a constructive version of the usual characterization of equivalences between fibrant objects. These tools are used in Proposition~\ref{prop:Charac_of_simp_equiv} to show that the equivalences of the Kan-Quillen model structure can be characterized as the maps inducing isomorphisms of $\pi_n$-setoids (which, assuming choice, is equivalent to bijections of $\pi_n$-sets). These setoids are useful in two situations:

\begin{itemize}[noitemsep,topsep=2pt]

\item One wants to work in an extremely weak logical framework, where quotients of sets by equivalence relations cannot always be constructed.

\item One wants to work without the axiom of choice and read Appendix~\ref{section_pisetoids} about the $\pi$-setoids characterization of equivalences.

\end{itemize}

Appendix~\ref{Subsection_Joyal_tierney_calculus} reviews Joyal-Tierney calculus, which plays a key role in Section~\ref{sec_monoidalcase} and is useful for the treatment of examples in Sections~\ref{Section_Examples} and \ref{sec:simplicialEx}.  Finally, Appendix~\ref{section_the_small_object_arguments} discusses the small object argument in constructive mathematics.

\subsection{Logical framework}
\label{sec:Logical_framework}

Everything we do here can be formalized in P.~Aczel's constructive set theory (CZF) \cite{Aczel2010CST}. It can also be formalized in the internal logic of an elementary topos with a natural number object. One minor exception is the most general form of the small object argument (as presented in Appendix~\ref{section_the_small_object_arguments}) applied to a large category, which relies on constructing objects by induction on a natural number, but all concrete applications of the small object argument we will use can be formalized both in an elementary topos with a natural number object or in (CZF). However, both these options are far stronger than what we need, and we will not impose any specific framework.

Indeed, while it was not our goal to look for the absolute minimal logical framework in which to do homotopy theory, it appeared that the natural framework for developing this theory is in fact far lighter than we would have thought. In the end, most of the general theory of weak model categories (i.e. Section~\ref{section_general}) is developed in the internal logic\footnote{It is not clear if the word ``logic'' is still suitable for such a low level.} of a mere \emph{category with finite limits}. Note that this is only for the general theory of weak model structures. Most examples will require a slightly stronger logical framework, mostly in order to implement the small object argument (this is discussed in Appendix~\ref{section_the_small_object_arguments}). Also technically speaking the definition of $\Ccal^{\bothe}$ in \ref{notation:cof_fib} involves taking a disjoint union, so we actually need the internal logic of an \emph{extensive\footnote{A category with finite limits and disjoint and universal finite coproducts.} category}, but this is only for convenience and could be avoided. 

There is a reason for this: the only way to make the theory work without axiom of choice is to require that everything that should exists (like diagonal fillers for lifting problems, factorization of maps, the limits and colimits that we need, and so on) is chosen. In particular, the correct way to interpret any quantification like ``$\forall x, \exists y$'' is as the existence of an application which given an $x$ produces a $y$. This has the effect of removing all need for any kind of quantification or logic from the theory. Hence by asserting that we work in the internal logic of a category with finite limits we avoid any possible doubt of how a statement like this should be interpreted.

This being said, we will sometimes, to keep the exposition readable, (especially for readers not interested in constructive aspects) still use quantifiers and say things like ``for all $x$ there exists a $y$ such that''. A statement like this should always be interpreted as a function. We will leave to the reader interested in the constructive aspects to make the appropriate translation. No confusion is possible here as our framework does not allow for any other interpretation of such sentences.

Another requirement that we could have for our logical framework is the existence of quotient sets. For example, morphisms in the homotopy category are defined as equivalence classes of maps for the homotopy relation. As far as we know there are two ways to deal with this:

\begin{itemize}[noitemsep,topsep=2pt]

\item Require the existence of quotients in our logical framework. This would mean working internally in an exact category.

\item Avoid the use of quotient by using ``setoids'' instead. This essentially amount to working internally in the exact completion of our category with finite limits.

\end{itemize}

For most of the paper the two options are equally valid, but for Appendix~\ref{section_pisetoids} the use of setoids is crucial in order to avoid the use of the axiom of choice, and it makes the exposition smoother if the homotopy category has been previously introduced in terms of setoids instead of quotient sets. For this reason we will use the setoid approach everywhere.

As mentioned before, Sections~\ref{Section_Examples} and \ref{sec:simplicialEx}, being focused on examples, will require a stronger logical framework in order to implement the small object argument. The precise nature of the required framework is a complicated matter that is discussed more in Appendix~\ref{section_the_small_object_arguments}.

\section{Weak model structures}
\label{section_general}

\subsection{Definition of weak model categories and homotopies}
\label{Subsection_DEFofWMS}

Weak model categories will be categories endowed with two classes of maps, ``cofibrations'' and ``fibrations'', satisfying some axioms. These axioms are considerably weaker than those of a Quillen model category, but are still enough to define a homotopy category and to introduce classicals notions like (weak) equivalence, homotopy limits and colimits, Quillen adjunctions, Quillen equivalences, etc.

\begin{notation}\label{notation:cof_fib}
A cofibration will always be denoted by a ``hooked'' arrow: $A \hookrightarrow B$, and a fibration by a double headed arrow: $X \twoheadrightarrow Y$.

In a category $\Ccal$ which has an initial object $0$ and a notion of cofibration, we say that an object $X$ is \emph{cofibrant} if the unique map $ 0 \rightarrow X$ is a cofibration. The full subcategory of cofibrant objects is denoted $\Ccal^{\cofe}$.

Similarly, if $\Ccal$ has a terminal object and a notion of fibration, we say that an object $X$ is \emph{fibrant} if the unique map $X \rightarrow 1$ is a fibration. The full subcategory of fibrant objects is denoted $\Ccal^{\fibe}$.

If $\Ccal$ has all these structures, an object will be called \emph{bifibrant} if it is both fibrant and cofibrant. The full subcategory of bifibrant objects is denoted $\Ccal^{\bfe}$.

The full subcategory of $\Ccal$ of objects that are either fibrant or cofibrant will be denoted by $\Ccal^{\bothe}$. More precisely, in the constructive setting, $\Ccal^{\bothe}$ is defined as the category whose set of objects is $\Ccal^{\fibe} \coprod \Ccal^{\cofe}$ and whose morphisms are the morphisms between their images in $\Ccal$.

\end{notation}

\begin{definition}\label{def:class_of_cof}
A \emph{class of cofibrations} on a category $\Ccal$ is a set of maps called \emph{cofibrations} which satisfies the following properties:

\begin{itemize}

\item $\Ccal$ has an initial object $0$ and it is cofibrant.

\item Any isomorphism with a cofibrant domain is a cofibration.

\item The composite of two cofibrations is a cofibration.

\item Given a diagram:

\[
\begin{tikzcd}
A \arrow[hook]{d}{i} \ar[dr,phantom,"\ulcorner"very near end] \arrow{r}{f} & C \ar[d,dotted] \\
B \ar[r,dotted] & \displaystyle C \coprod_A B \\
\end{tikzcd}
\]

\noindent with $A$ and $C$ two cofibrant objects and $i$ a cofibration, then the pushout $C \coprod_A B$ exists and the map $C \rightarrow C \coprod_A B$ is a cofibration.

\end{itemize}

Dually, a \emph{class of fibrations} on a category $\Ccal$ is a set of maps, called \emph{fibrations} in $\Ccal$ which form a class of cofibrations in $\Ccal^{op}$.

\end{definition}

\begin{remark}
A weak model category will be a category $\Ccal$ endowed with both a class of fibrations and a class of cofibrations satisfying some additional compatibility axioms, see Definition~\ref{Def::weak_model_categories}. 

Here again, in a weak logical framework, everything should be interpreted following the ideas of Appendix~\ref{section_setoids}: The fibrations and cofibrations are not necessarily subsets of morphisms, but sets $\fib(\Ccal)$ and $\cof(\Ccal)$ endowed with a map to the set of all arrows of $\Ccal$, and all the axioms of the definition are interpreted as operations. 

In particular we assume that we have chosen pushouts along cofibrations, but this choice can depend on the ``cofibration structure'' of the map: If $i,j \in \cof(\Ccal)$ have the same underlying arrow in $\Ccal$ pushouts along them can be different. 
\end{remark}

\begin{remark}

Given a class of cofibrations on a category $\Ccal$, the class of ``cofibrations between cofibrant objects'' is again a class of cofibrations. Moreover, the definition of weak model category, and all the relevant notion related to it will only involve the cofibrations between cofibrant objects and fibrations between fibrant objects. Hence we could freely add the assumptions that:

\begin{itemize}[noitemsep,topsep=2pt]

\item The domain of every cofibration is cofibrant,

\item The target of every fibration is fibrant,

\end{itemize}

\noindent without changing the content of any of the results we will give here. Even if we do not explicitely make these assumptions to not restrict ourselves, we emphasize that:

\begin{center} \emph{We only ever consider cofibrations with cofibrant domain and fibrations with fibrant target.} \end{center}

\end{remark}

\begin{remark} The reader may be surprised by the fact that we do not include closure under retract in the definition of a class of cofibrations. The reason for this is simply that this property is not used anywhere in the paper.
\end{remark}

\begin{notation}
\label{Def_lifting_property_and_weak_factorization_system}As usual if $f$ and $g$ are two morphisms in a category $\Ccal$ we say that $f$ has the left lifting property against $g$ (or that $g$ has the right lifting property against $f$) and we write $f \pitchfork g$ if for each solid square:

\[
\begin{tikzcd}[ampersand replacement=\&]
A \arrow{d}{f} \arrow{r} \& X \arrow{d}{g} \\
B  \arrow[dotted]{ur}{\exists} \arrow{r} \& Y \\
\end{tikzcd}
\]

\noindent there is a (chosen) dotted diagonal filling.

\end{notation}

\begin{definition}\label{Def_trivcofib}
Let $\Ccal$ be a category endowed with a class of fibrations and a class of cofibrations. An arrow is said to be:

\begin{itemize}

\item An \emph{acyclic fibration} if it is a fibration and it has the right lifting property against all cofibrations between cofibrant objects.

\item An \emph{acyclic cofibration} if it is a cofibration and it has the left lifting property against all fibrations between fibrant objects.

\end{itemize}

In diagrams, acyclic cofibrations are represented by $\overset{\sim}{\hookrightarrow}$ and acyclic fibrations by $\overset{\sim}{\twoheadrightarrow}$.

\end{definition}

Technically speaking, our logical framework does not allow us to form the ``set of acyclic fibrations'', but we can still say that a map ``is an acyclic fibration'' to mean that there is a function producing the desired lifts.

Of course, acyclic fibrations and cofibrations will end up being ``equivalences'' as soon as we will have defined the notion (Proposition~\ref{Prop_trivCofFibAreinvertible}). In fact we will prove in \ref{Prop_triv=equiv} that in a weak model category, a (co)fibration is acyclic if and only if it an equivalence. It should also be noted (see for example Lemma 7.14 of \cite{joyal2006quasi}) that in a Quillen model category, a cofibration is acyclic if and only if it has the left lifting property with respect to all fibrations between fibrant objects. Hence the terminology introduced here is compatible with the theory of Quillen model categories.

\begin{lemma}
\label{lem:acyclicCof_basics}Acyclic cofibrations are stable under composition and pushout (amongst cofibrant objects) . A cofibration $i$ which is a retract of an acyclic cofibration $j$ is again an acyclic cofibration. If $i$ and $j$ are composable cofibrations and if $i \circ j$ and $i$ are acyclic, then $j$ is acyclic.  All the dual conditions holds for acyclic fibrations.
\end{lemma}

\begin{proof}
This is just the classical fact that the class of maps $f$ such that $f \pitchfork g$ is stable under pushout, composition and retract. The ``$2$-out-of-$3$'' claim follows from the fact being an acyclic cofibration is tested against fibrations between \emph{fibrant} objects:

\[\begin{tikzcd}[ampersand replacement=\&]
U \arrow[hook]{d}{j} \arrow{r} \& X \arrow[two heads]{d}\\
V \arrow[hook]{d}{i} \arrow{r} \& Y \\
W \arrow[dotted]{ur} \arrow[dotted]{uur} \& \\
\end{tikzcd}\]

The lower dotted arrow is constructed using that $Y$ is fibrant and $i$ is acyclic, and the upper one using that $i \circ j$ is acyclic and $X \rightarrow Y$ is a fibration between fibrant objects. The composite $V \rightarrow X$ gives the diagonal filling we are after.\end{proof}

In a category with classes of fibrations and cofibrations as above, if $X$ is a cofibrant object, a fibrant replacement of $X$ or bifibrant replacement of $X$ is a fibrant object $X^{\fibe}$ endowed with an acyclic cofibration $X \overset{\sim}{\hookrightarrow} X^{\fibe}$. Dually a cofibrant replacement (or bifibrant replacement) of a fibrant object $X$ is a cofibrant object $X^{\cofe}$ endowed with an acyclic fibration $X^{\cofe} \overset{\sim}{\twoheadrightarrow} X$. 

\begin{definition}
\begin{itemize}\item[]

\item A \emph{relative strong cylinder object} for a cofibration $A \hookrightarrow B$ is a factorization of the relative co-diagonal map $B \coprod_A B \rightarrow B$ into:

\[ B \coprod_A B \hookrightarrow I_A B \rightarrow B \]

\noindent where the first map is a cofibration and its pre-composite with the first co-product inclusion $B \hookrightarrow B \coprod_A B \hookrightarrow I_A B$ is an acyclic cofibration.

\item A \emph{relative strong path object} for a fibration $Y \twoheadrightarrow X$ is a factorization of the relative diagonal map into:

\[ Y \rightarrow P_X Y \twoheadrightarrow Y \times_X Y \]

\noindent where the second map is a fibration and its composite $P_X Y \twoheadrightarrow Y \times_X Y \twoheadrightarrow Y$ is an acyclic fibration.

\end{itemize}

\end{definition}

A (strong) cylinder object $IX$ for a cofibrant object $X$ is a relative cylinder for the cofibration $\emptyset \hookrightarrow X$. A (strong) path object $PY$ for a fibrant object $Y$ is a relative path object for the fibration $Y \twoheadrightarrow 1$.


The apparent asymmetry of the definition (only one of the two ``legs'' is asked to be acyclic) is artificial: in a weak model categories we will define a notion of equivalences (\ref{def:WeakEquiv}) satisfying the $2$-out-of-$3$ condition, and we will show in \cref{Prop_triv=equiv}.\ref{Prop_triv=equiv:isWMS} that (co)fibrations between (co)fibrant objects are acyclic if and only if they are equivalences, so the second leg will automatically be acyclic as well.

We can now give the main definition:

\begin{definition}\label{Def::weak_model_categories}
A \emph{weak model category}, is a category $\Ccal$ endowed with both a class of cofibrations and a class of fibrations which satisfies the following:

\begin{itemize}

\item \emph{Factorization axiom:} Any map from a cofibrant object to a fibrant object can be factored both as a cofibration followed by an acyclic fibration and as an acyclic cofibration followed by a fibration.

\item \emph{Cylinder axiom:} Every cofibration from a cofibrant object to a fibrant object admits a relative strong cylinder object.

\item \emph{Path object axiom:} Every fibration from a cofibrant object to a fibrant object admits a relative strong path object.

\end{itemize}

\end{definition}

Weak model categories have the following elementary stability properties:

\begin{prop}
Let $\Ccal$ be a weak model category then:

\begin{itemize}

\item $\Ccal^{op}$ is a weak model category whose (acyclic) fibrations and (acyclic) cofibrations are respectively the (acyclic) cofibrations and (acyclic) fibrations of $\Ccal$.

\item For any cofibrant object $A$ of $\Ccal$, the coslice category $A/\Ccal$ of arrows $A \rightarrow X$ is a weak model category, whose cofibrations, acyclic cofibrations, fibrations and acyclic fibrations are the maps whose image by the forgetful functor to $\Ccal$ has the same property.

\item Dually, for any fibrant object $X$ of $\Ccal$ the slice category $\Ccal/X$ of arrows $B \rightarrow X$ is a weak model category, whose cofibrations, acyclic cofibrations, fibrations and acyclic fibrations are the maps whose image by the forgetful functor to $\Ccal$ has the same property.
\end{itemize}
\end{prop}

In a weak model category, cofibrations between cofibrant objects and fibrations between fibrant objects still admit a kind of ``relative cylinder object'' and ``relative path object'' which we call ``weak cylinder objects'' and ``weak path objects'':

\begin{definition}
\begin{itemize}

\item[]

\item A \emph{relative weak cylinder object} for a cofibration $A \hookrightarrow B$ is a diagram of the form: 

\[
\begin{tikzcd}[ampersand replacement=\&]
B \coprod_A B \arrow[hook]{d} \arrow{r} \& B \arrow[hook]{d}{\sim}  \\
I_AB \arrow{r} \&  D_A B \\
\end{tikzcd}
\]

\noindent where furthermore the first map $\iota_0\colon B \overset{\sim}{\hookrightarrow} I_A B$ is an acyclic cofibration.

\item A \emph{relative weak path object} for a fibration $Y \twoheadrightarrow X$ is a diagram of the form:

\[
\begin{tikzcd}[ampersand replacement=\&]
T_X Y \arrow{r} \arrow[two heads]{d}{\sim} \& P_X Y \arrow[two heads]{d} \\
Y \arrow{r}{\Delta} \& Y \times_X Y
\end{tikzcd}
\]

\noindent where furthermore the first projection $\pi_0\colon  P_X Y \overset{\sim}{\twoheadrightarrow} Y$ is an acyclic fibration.

\end{itemize}

\end{definition}

The idea is simple: in general we do not have a map $I_A B \rightarrow B$, which would be used for example to define self homotopies of a map $B \rightarrow X$. Instead we have a copsan $I_A B \rightarrow D_A B \overset{\sim}{\hookleftarrow} B$ which defines such a map at least at the level of the homotopy category. The object $D_AB$ (as well as $T_X Y$) is called the \emph{reflexivity witness}.

\begin{remark}\label{rk:Exist_strong_Cyl=exists_weak_cyl}
\begin{itemize}

\item[]

\item Any relative strong cylinder object can be seen as a relative weak cylinder object by taking $D_A B = B$.

\item If a cofibration $A \hookrightarrow B$ has a relative weak cylinder object and $B$ is furthermore fibrant, then, using the lifting property of $B$, we can construct a retraction:

\[\begin{tikzcd}[ampersand replacement=\&]
B \arrow[hook]{d}{\sim} \arrow{r}{Id_B} \& B \\
D_A B \ar[dotted,ur,"r"{description}]  
\end{tikzcd}\]

\noindent and the composite:

\[ B \coprod_A B \hookrightarrow I_A B \rightarrow D_A B \overset{r}{\rightarrow} B\]

\noindent forms a relative \emph{strong} cylinder object of $A \hookrightarrow B$.

\item If $A \hookrightarrow B $ is a cofibration, and $B \overset{\sim}{\hookrightarrow} \widetilde{B}$ is a fibrant replacement of $B$ then a relative strong cylinder object for the cofibration $A \hookrightarrow \widetilde{B}$ gives us a relative weak cylinder object for $A \hookrightarrow B$ as follows:

\[
\begin{tikzcd}[ampersand replacement=\&]
B \coprod_A B \arrow[hook]{d} \arrow{r} \& B \arrow[hook]{dd}{\sim}  \\
\widetilde{B} \coprod_A \widetilde{B} \arrow[hook]{d}\& \\
I_A \widetilde{B} \arrow{r} \&  \widetilde{B} \\
\end{tikzcd}
\]

\item Hence, in the presence of the factorization axiom, the cylinder axiom is equivalent to the requirement that every cofibration between cofibrant objects has a relative weak cylinder object.

\item All the remarks above can be dualized to path objects and fibrations.

\end{itemize}

\end{remark}

\begin{definition}
Let $f,g \colon  X \rightrightarrows Y$ be two maps from a cofibrant object to a fibrant object in a category with fibrations and cofibrations.

\begin{itemize}

\item We say that $f$ and $g$ are homotopic relative to a (weak or strong) cylinder object $IX$ for $X$ if the map $(f,g)\colon  X \coprod X \rightarrow Y$ factors through $X \coprod X \hookrightarrow IX$.

\item We say that $f$ and $g$ are homotopic relative to a (weak or strong) path object $PY$ for $Y$ if the map $(f,g)\colon X \rightarrow Y \times Y$ factors through $PY \twoheadrightarrow Y \times Y $.

\end{itemize}

\end{definition}

Note that if $i\colon A \hookrightarrow B$ is a cofibration (with $A$ and $B$ cofibrant) and $f,g$ are two maps $f,g\colon B \rightrightarrows Y$ (with $Y$ fibrant) such that $f \circ i = g \circ i$ we can also talk about ``homotopy relative to $A$'', that will be for example parametrized by a relative cylinder object for $A \hookrightarrow B$. This relative version will be very useful. We do not discuss this further simply because it is the homotopy relation in the coslice category $A/\Ccal$, so it can be seen as a special case of the non-relative version.

\begin{lemma}
\label{Lem_reflexivity}Let $f \colon X \rightarrow Y$ be a map from a cofibrant object $X$ to a fibrant object $Y$, then there is a homotopy $r_f$ from $f$ to $f$ relative to any cylinder object for $X$ or path object for $Y$.
\end{lemma}

\begin{proof}
 For a weak cylinder object $(IX,DX)$ for $X$ we obtain the reflexivity homotopy $r_f$ as follows: 

\[\begin{tikzcd}[ampersand replacement=\&]
X \coprod X \arrow[hook]{d} \arrow{r}{\nabla} \&  X \arrow{r}{f} \arrow[hook]{d}{\sim} \& Y  \\
IX \arrow{r} \arrow[bend right = 70]{rru}[swap]{r_f} \&  DX \arrow[dotted]{ur} \& \\
\end{tikzcd}\]

\noindent and dually for the case of a path object for $Y$.\end{proof}

\begin{prop}\label{cylinderEqPath}
Consider two maps $f,g \colon X \rightrightarrows Y$ with $X$ cofibrant and $Y$ fibrant, such that $X$ admits at least one cylinder object and $Y$ admits at least one path object. Then the homotopy relations defined by any cylinder object for $X$ or path object for $Y$ are equivalent.
\end{prop}

We will hence just say that $f$ and $g$ are homotopic without specifying if it is with respect to a cylinder object or to a path object nor with respect to which path object or cylinder object, at least as long as we do not need to specify the homotopy itself. Of course this proposition really means that we have an explicit construction which associate to any homotopy relative to some path objects a homotopy relative to any other path object or cylinder object.

\begin{proof}
Let $f,g\colon X \rightrightarrows Y$ be two arrows as in the proposition, with a homotopy $h\colon  IX \rightarrow Y$ between $f$ and $g$ relative to a weak cylinder object $IX$. Let $PY$ be any weak path object for $Y$. There is a commutative square:

\[
\begin{tikzcd}[ampersand replacement=\&]
X \arrow[hookrightarrow]{d}{\sim} \arrow{r}{r'_f} \&  PY \arrow[twoheadrightarrow]{d}  \\
IX \arrow{r}{(r_f,h)} \& Y \times Y \\
\end{tikzcd}
\]

\noindent where the left vertical map is the ``first inclusion'', and $r_f$ and $r'_f$ denotes the homotopy from $f$ to $f$ produced by Lemma~\ref{Lem_reflexivity}.

We obtain a diagonal filling $w \colon  IX \rightarrow PY$, and pre-composing it with the second ``inclusion'' $i_2\colon  X \rightarrow IX$ gives a map $X \rightarrow PY$ whose projections to $Y$ are $f$ and $g$, i.e. a homotopy $h'$ between $f$ and $g$ relative to $PY$:

\[\begin{tikzcd}[ampersand replacement=\&]
\& X \arrow[hookrightarrow]{d}{i_1} \arrow{r}{r'_f} \&  PY \arrow[twoheadrightarrow]{d}  \\
X \arrow[hook]{r}{i_2} \arrow[bend left=60]{rru}{h'} \arrow[bend right=30]{rr}[swap]{(f,g)} \& IX \arrow{ur}{w} \arrow{r}[swap]{(r_f,h)} \& Y \times Y \\
\end{tikzcd}
\]

Dually, a homotopy indexed by any path object will induce a homotopy between any other cylinder object, which concludes the proof. \end{proof}

\begin{theorem}\label{Th_homotopyEqRel}
Let $\Ccal$ be a category with fibrations and cofibrations, let $X$ be a cofibrant object admitting at least one cylinder object and $Y$ a fibrant object admitting at least one path object. Then the homotopy relation for maps from $X$ to $Y$ is an equivalence relation.
\end{theorem}

We mean by that we have a setoid structure on the set of maps from $X$ to $Y$ and the set of homotopies between them, this holds for whatever choice of cylinder and or path objects we are using (and using several choices of path and cylinder object simultaneously is also an option).

\begin{proof}
 Reflexivity has been proved as Lemma~\ref{Lem_reflexivity}. Let $\alpha,\beta,\gamma$ be three arrows $X \rightarrow Y$ with homotopies $h$ from $\alpha$ to $\beta$ and $h'$ from $\beta$ to $\gamma$.

There is a ``homotopy'' between $\alpha$ and $\gamma$ relative to the object:

\[ X \coprod X \hookrightarrow IX \coprod_X IX  \]

\noindent this cofibration fits in a diagram:

\[
\begin{tikzcd}[ampersand replacement=\&]
X \coprod X \arrow[hook]{d} \arrow{r}{\nabla} \& X \arrow[hook]{d}{\sim}  \\
IX \coprod_X IX \arrow{r} \&  DX \coprod_X D X \\
\end{tikzcd}
\]

\noindent and the stability of acyclic cofibrations under pushout and compositions gives all the conditions that we need for this to be a weak cylinder object for $X$ hence proves that $\alpha$ is homotopic to $\gamma$.

Symmetry needs a little more work because of our asymetrical definition of cylinder objects: Let $h\colon IX \rightarrow Y$ be a homotopy between $f,g\colon X \rightrightarrows Y$. Let $PY$ be any path object for $Y$, and let $P'Y$ be $PY$ composed with the exchange map $\tau\colon Y \times Y \rightarrow Y \times Y$. As the proof of Proposition~\ref{cylinderEqPath} did not use the assumption that the projections of the path object are acyclic fibrations it also applies to $P'Y$ and hence the homotopy given by $IX$ produces a $P'Y$-homotopy between $f$ and $g$, but this is exactly a $PY$-homotopy between $g$ and $f$ and this proves the symmetry of the homotopy relation.\end{proof}

\subsection{Equivalences and the homotopy category}
\label{subsec:homotopycategory}

\begin{assumption}\label{Assumption:PreWMS} The results in this section apply to slightly more general structure than weak model categories: We do not need the relative version of path objects and cylinder objects. Instead we consider $\Ccal$ a category with fibrations and cofibrations, which satisfies the factorization axiom of Definition~\ref{Def::weak_model_categories}, and in which every bifibrant object has both a cylinder object and a path object, or (equivalently, by Remark~\ref{rk:Exist_strong_Cyl=exists_weak_cyl}) in which every cofibrant object has a weak cylinder object and every fibrant object has a weak path object.
\end{assumption}

\begin{definition}
We denote by $Ho(\Ccal^{bf})$ the (setoid\footnote{See Appendix~\ref{section_setoids} for the notion of setoid-category. Though one can be ignore this for most of the paper and consider $Ho(\Ccal^{bf})$ as an ordinary category.}) category whose objects are the bifibrant objects of $\Ccal$ and whose arrows are morphisms in $\Ccal$ up to the homotopy relation.

\end{definition}

We proved in Theorem~\ref{Th_homotopyEqRel} that the homotopy relation is an equivalence relation, and, as Proposition~\ref{cylinderEqPath} shows that it can be defined equivalently using a cylinder object or a path object, it is clearly preserved both by pre-composition and post-composition, hence the ``quotient'' of $\Ccal$ by this equivalence relation is indeed a setoid-category.

\begin{prop}\label{Prop_trivCofFibAreinvertible}
Acyclic cofibrations and acyclic fibrations between bifibrant objects are invertible arrow in $Ho(\Ccal^{bf})$.
\end{prop}

\begin{proof}
It is enough to show it for acyclic cofibrations. Let $j\colon X \overset{\sim}{\hookrightarrow} Y$ be an acyclic cofibrations between two bifibrant objects. 

A diagonal filling in the following square:

\[
\begin{tikzcd}[ampersand replacement=\&]
X \arrow[hookrightarrow]{d}{\sim} \arrow{r} \&  X \arrow{d}  \\
Y \arrow{r} \arrow[dotted]{ur}{r} \& 1 \\
\end{tikzcd}
\]

\noindent gives us a retraction of $j$. And $j$ is an epimorphism in $Ho(\Ccal^{bf})$: if two maps $u,v\colon Y \rightrightarrows Z$ are such that $u \circ j$ and $v\circ j$ are homotopic then a diagonal filling in the square:

\[
\begin{tikzcd}[ampersand replacement=\&]
X \arrow[hookrightarrow]{d}{\sim} \arrow{r}{h} \&  PZ \arrow[two heads]{d}  \\
Y \arrow{r}{(u,v)} \arrow[dotted]{ur} \& Z \times Z \\
\end{tikzcd}
\]

\noindent gives a homotopy between $u$ and $v$. Applying this to $Z=Y$, $u=Id_Y$ and $v =j \circ r$ gives that $j \circ r$ is homotopic to $Id_Y$ and concludes the proof.\end{proof}

\begin{prop}\label{Prop_HoCcf}
The quotient functor $\Ccal^{bf} \rightarrow Ho(\Ccal^{bf})$ identifies $Ho(\Ccal^{bf})$ with the localization of $\Ccal^{bf}$ at all acyclic cofibrations (dually at all acyclic fibrations).
\end{prop}

What we mean here is that $Ho(\Ccal^{bf})$ has the universal property of a localization, in the sense that for any functor $F\colon  \Ccal^{bf} \rightarrow D$ which send acyclic cofibrations (or acyclic fibrations) to isomorphisms factors uniquely as $\Ccal^{bf} \rightarrow Ho(\Ccal^{bf}) \rightarrow D$.

Moreover $D$ can be taken to be a setoid-category in this statement.

In particular if the logical framework is strong enough to construct the formal (Gabriel-Zisman) localization of $\Ccal^{bf}$ (for example if $\Ccal^{bf}$ is small and if we have list object and quotient by equivalence relation) then this formal localization will be equivalent to $Ho(\Ccal^{bf})$.

\begin{proof}
First, we observed in Proposition~\ref{Prop_trivCofFibAreinvertible} that acyclic cofibrations (and acyclic fibrations) are invertible in $Ho(\Ccal^{bf})$. Let $F \colon \Ccal^{bf} \rightarrow \Dcal$ be a functor which invert all acyclic cofibrations, in particular, it inverts the map $i_1\colon X \hookrightarrow IX$ and hence also the map $u\colon IX \rightarrow X$ as it is a retraction of the previous one. As $i_2\colon X \hookrightarrow IX$ is another section of $u$ we have $F(i_1)=F(i_2)$ in $\Dcal$.

Any two homotopic maps in $\Ccal$ are written as $h\circ i_1$ and $h \circ i_2$ and hence have equals image in $\Dcal$. This shows that $F$ factors uniquely into $ Ho(\Ccal^{bf})$ and hence proves that the quotient functor $\Ccal^{bf} \rightarrow Ho(\Ccal^{bf})$ is the localization of $\Ccal^{bf}$ at acyclic cofibrations. By duality, $Ho(\Ccal^{bf})$ is also the localization of $\Ccal^{bf}$ at acyclic fibrations. \end{proof}

It is well known that in a Quillen model category the homotopy category of bifibrant objects is in fact equivalent to the localization of the whole category at equivalences. To obtain a similar result for weak model categories, we will gradually push this equivalence between this homotopy category of bifibrant objects and localization of various larger full subcategory of $\Ccal$ using the following lemma:

\begin{lemma}\label{Lem_localizationExt}
Let $\Ccal$ be a category, $\Dcal \subset \Ccal$ a full subcategory, $\Wcal$ a class of maps in $\Ccal$ and $\Wcal'$ a class of maps in $\Dcal$.

We assume that:
\begin{enumerate}

\item \label{Lem_localizationExt:Cond_Locexists} The localization $\Dcal[\Wcal'^{-1}]$ exists.

\item \label{Lem_localizationExt:Cond_ReplacementExists} For each object $c \in \Ccal$ there is an arrow $c \rightarrow d$ with $w \in \Wcal$ and $d \in \Dcal$.

\item \label{Lem_localizationExt:Cond_LiftExists} For each solid diagram: 

\[
\begin{tikzcd}[ampersand replacement=\&]
c \arrow{d} \arrow{r}{w} \& d \arrow[dotted]{dl} \\
d' 
\end{tikzcd}
\]

\noindent with $c\in \Ccal, w \in \Wcal$ and  $d,d' \in \Dcal$ there is a dotted arrow that makes the triangle commutes.

\item \label{Lem_localizationExt:Cond_LiftUnique} Each pair of arrows fitting in place of the dotted arrow in the diagram above have the same image in $\Dcal[\Wcal'^{-1}]$.

\item \label{Lem_localizationExt:Cond_Wcompose} $\Wcal$ is stable under composition.

\end{enumerate}

Then the localization $ \Ccal[(\Wcal \cup \Wcal')^{-1}]$ exists and is equivalent to $\Dcal[\Wcal'^{-1}]$ by the functor induced by the inclusion $\Dcal \subset \Ccal$.

\end{lemma}

As previously mentioned, all the ``there is'' in the assumption are interpreted as ``we have operations giving us these objects''. The correct interpretation of assumption $4.$ in setoid language is that given two arrows that makes the triangle commutes there is a (chosen) relation between them in the localization.

\begin{proof}
We assume that $\Dcal[\Wcal'^{-1}]$ exists, we will construct a functor $F\colon \Ccal \rightarrow \Dcal[\Wcal'^{-1}]$:

Any object $c \in \Ccal$ is sent to the chosen object $F(c)=d$ such that there is an arrow $w\colon c \rightarrow d$ with $w \in \Wcal$ and $d \in \Dcal$.
If $f\colon  c \rightarrow c'$ is an arrow and $d$ and $d'$ are the image of $c$ and $c'$, we construct the image of $f$ by taking a lift:

\[\begin{tikzcd}[ampersand replacement=\&]
c \arrow{d}{f} \arrow{r}{w \in \Wcal} \& d \arrow[dotted]{d}{F(f)} \\
c' \arrow{r}{w'} \& d'
\end{tikzcd}\]

Such an arrow exists because of the third assumption and is unique because of the fourth assumption, hence the functions exists. To be more precise, in the setoids language, ``unique'' means that any two such arrows can be connected by a relation, and the functions ``exists'' means that it can be made into a morphism of setoids, i.e. that it acts on relations as well.

Functoriality (in the setoid-categories sense) is immediate because of this uniqueness result. It is easy to show that any arrow in $\Wcal$ or $\Wcal'$ is sent to an isomorphism by this functor. Also the restriction of this functor to $\Dcal \subset \Ccal$ is naturally isomorphic to the universal functor $\Dcal \rightarrow \Dcal[\Wcal'^{-1}]$.

We can now show that any functor $G \colon \Ccal \rightarrow \Kcal$ which inverts all maps in $\Wcal$ and $\Wcal'$ factors through $F$ up to equivalence:

First $G$ restricted to $\Dcal$ induces a functor $G_{\Dcal} \colon  \Dcal[\Wcal'^{-1}] \rightarrow \Kcal$, for any object $c \in \Ccal$, there is an arrow $w\colon c \rightarrow F(c)$ with $w \in \Wcal$, applying $G$ on both sides gives $G(w) \colon  G(c) \rightarrow G_{\Dcal}(F(c))$. By assumption $G(w)$ is an isomorphism, and it is immediate to check that it is functorial in $w$. Hence this produces an isomorphism of functors $G \sim G_{\Dcal} \circ F$ hence proving that $G$ factors into $F$ up to isomorphisms which show that $\Dcal[\Wcal^{-1}]$ has the universal property of the localization $\Ccal[\Wcal^{-1},\Wcal'^{-1}]$. (The uniqueness of the factorization up to unique equivalence is clear).\end{proof}

\begin{theorem}\label{Th_LocCatEq}
Let $\Ccal$ be a weak model category. The following categories (see \ref{notation:cof_fib}) all exists and are equivalent:

\begin{enumerate}

\item $Ho(\Ccal^{bf})$

\item The localization of $\Ccal^{bf}$ at acyclic fibrations.

\item The localization of $\Ccal^{bf}$ at acyclic cofibrations.

\item The localization of $\Ccal^{\cofe}$ at acyclic cofibrations.

\item The localization of $\Ccal^{\fibe}$ at acyclic fibrations.

\item The localization of $\Ccal^{\bothe}$ at all acyclic cofibrations with cofibrant domain and all acyclic fibrations with fibrant target.

\end{enumerate}

The equivalence being induced by the natural quotient functor from $\Ccal^{bf}$ to $Ho(\Ccal^{bf})$ and the square of inclusion:

\[
\begin{tikzcd}[ampersand replacement=\&]
\Ccal^{bf} \arrow{d} \arrow{r} \& \Ccal^{\cofe} \arrow{d} \\
\Ccal^{\fibe} \arrow{r} \& \Ccal^{\bothe} \\
\end{tikzcd}
\]

\end{theorem}

\begin{proof}
The equivalence of first three categories have already been proved. We then prove that $\Ccal^{bf} \rightarrow \Ccal^{\cofe}$ induces an equivalence after localizing at acyclic cofibrations using Lemma~\ref{Lem_localizationExt} with $\Wcal$ and $ \Wcal'$ both being the acyclic cofibrations. Condition~\ref{Lem_localizationExt:Cond_Locexists} follows from Proposition~\ref{Prop_HoCcf}.
Condition~\ref{Lem_localizationExt:Cond_ReplacementExists} is just the existence of factorization as an acyclic cofibration followed by a fibration of $X \rightarrow 1$.
Condition~\ref{Lem_localizationExt:Cond_LiftExists} is the lifting property of acyclic cofibrations with respect to the fibration $d \twoheadrightarrow 1$.
Condition~\ref{Lem_localizationExt:Cond_LiftUnique}: Given $c \overset{\sim}{\hookrightarrow} d \rightrightarrows d'$ with $d'$ fibrant. We construct a homotopy between the two maps as a diagonal filling in:

\[
\begin{tikzcd}[ampersand replacement=\&]
c \arrow[hook]{d}{\sim} \arrow{r} \& Pd' \arrow[twoheadrightarrow]{d} \\
d \arrow{r} \& d' \times d'
\end{tikzcd}
\]

Finally acyclic cofibrations are stable under composition (Condition~\ref{Lem_localizationExt:Cond_Wcompose}).

This proves that the localization of $\Ccal^{\cofe}$ at acyclic cofibrations is equivalent to the localization of $\Ccal^{bf}$ at acyclic cofibrations, i.e. is equivalent to $Ho(\Ccal^{bf})$.

\bigskip

Dually, the localization of $\Ccal^{\fibe}$ at acyclic fibrations is equivalent to $Ho(\Ccal^{bf})$.

\bigskip

We now move to the localization of $\Ccal^{\bothe}$ at all maps that are either acyclic cofibration with cofibrant domain or acyclic fibration with fibrant target. We apply Lemma~\ref{Lem_localizationExt} to the inclusion $\Ccal^{\fibe} \subset \Ccal^{\bothe}$ with $\Wcal$ being all acyclic cofibration with cofibrant domains as well as identity maps, and $\Wcal'$ being the class of acyclic fibration with fibrant domain. All the conditions are checked exactly in the same way as in the previous case, except Condition~\ref{Lem_localizationExt:Cond_LiftUnique}: we construct instead a homotopy in the sense of a weak path object $Pd'$ but this is enough to show that the two maps are equal in the localization at acyclic fibrations.\end{proof}

\begin{definition}\label{def:WeakEquiv}
An arrow in $\Ccal^{\bothe}$ is said to be an equivalence if it is invertible in the homotopy category, i.e. in the equivalent localization of Theorem~\ref{Th_LocCatEq}.
\end{definition}

Equivalences automatically satisfies the $2$-out-of-$3$ condition, and even the stronger $2$-out-of-$6$ condition: If $f,g$ and $h$ are composable and both $f \circ g$ and $g\circ h$ are equivalences then $f,g,h$ and $f \circ g \circ h$ are equivalences. They are also stable under retracts. Acyclic cofibrations with cofibrant domain as well as acyclic fibrations with fibrant target are equivalences. But as mentioned in the introduction, we do not have in general a good notion of equivalences for objects which are neither fibrant nor cofibrant.

\begin{lemma}\label{Lem_weaklifting1}
Let $\Ccal$ be as in \ref{Assumption:PreWMS}, given a diagram of the form:

\[
\begin{tikzcd}[ampersand replacement=\&]
A \arrow[hook]{d}{i} \arrow{r}{w} \& X \arrow{d}{\in \Wcal}[swap]{f}  \\
B \arrow{r} \& Y  \\
\end{tikzcd}
\]

\noindent with $A$ and $B$ cofibrant, $i$ a cofibration, $X$ and $Y$ fibrant and $f$ an equivalence, then there is a diagonal filler which makes the upper triangle commutes.

\end{lemma}

\begin{proof}

We first show the lemma when all the objects involved are bifibrant. In this situation, as $f$ is an isomorphism in $Ho(\Ccal^{bf})$ there must exist a diagonal filler in the category $Ho(\Ccal^{bf})$, in particular there is a map $v\colon  B \rightarrow X$ and a homotopy $h\colon  A \rightarrow PX$ from $v \circ i$ to $w$, we can then form the solid diagram below, which admit a dotted diagonal filler:

\[
\begin{tikzcd}[ampersand replacement=\&]
A \arrow[bend left=30]{rr}{w} \arrow[hook]{d}{i} \arrow{r}{h} \& PX \arrow[two heads]{d}{\sim}[swap]{\pi_1} \arrow[two heads]{r}{\pi_2}[swap]{\sim} \& X  \\
B \arrow[dotted]{ur}{t} \arrow{r}{v} \& X  \\
\end{tikzcd}
\]

The composite $\pi_2 \circ t$ gives us a map from $B$ to $X$ such that $\pi_2 \circ t \circ i = w$ hence this concludes the proof.

We then shows that given a square as in the proposition there is an ``inner square'' as below: 

\[
\begin{tikzcd}[ampersand replacement=\&]
A \arrow[hook]{ddd}{i} \arrow{rrr}{w} \arrow[hook]{dr}{\sim} \& \& \& X \arrow{ddd}{\in \Wcal}[swap]{f}  \\
\& A_1 \arrow[hook]{d}{i_1} \arrow{r}{w_1} \& X_1 \arrow{d}{f_1} \arrow[two heads]{ur}{\sim} \& \\
\& B_1  \arrow{r} \& Y_1 \arrow[two heads]{dr}{\sim}  \& \\
B \arrow[hook]{ur}{\sim} \arrow{rrr} \& \& \& Y  \\
\end{tikzcd}
\]

\noindent with all the objects of the inner square being bifibrant. The map $f_1$ is still an equivalence because the acyclic fibration with fibrant target are equivalences (by Theorem~\ref{Th_LocCatEq}) and equivalences satisfies the $2$-out-of-$3$ property. A filler as above in the inner square produce a filler in the outer square.

 Indeed, we first factor the map $A \rightarrow X$ as an acyclic cofibration followed by a fibration $A \overset{\sim}{\hookrightarrow} A_1 \twoheadrightarrow X$ to get a bifibrant object $A_1$, and the map $A_1 \rightarrow X$ as a cofibration followed by an acyclic fibration $A_1 \hookrightarrow X_1 \overset{\sim}{\twoheadrightarrow} X$ to get a bifibrant object $X_1$. We get a diagram:

\[
\begin{tikzcd}[ampersand replacement=\&]
A \arrow[hook]{dd}{i} \arrow{rrr}{w} \arrow[hook]{dr}{\sim} \& \& \& X \arrow{dd}{\in \Wcal}[swap]{f}  \\
\& A_1  \ar[r,hook,"w_1"] \& X_1 \arrow[two heads]{ur}{\sim} \& \\
B \arrow{rrr} \& \& \& Y \, .  \\
\end{tikzcd}
\]

We form the pushout $B' = B \coprod_A A_1$ and $Y'=B \coprod_A X_1$ to get a diagram:

\[
\begin{tikzcd}[ampersand replacement=\&]
A \arrow[hook]{ddd}{i} \arrow{rrr}{w} \arrow[hook]{dr}{\sim} \& \& \& X \arrow{ddd}{\in \Wcal}[swap]{f}  \\
\ar[dr,phantom,"\ulcorner"{very near end}] \& A_1 \arrow[hook]{d} \arrow{r}{w_1} \ar[dr,phantom,"\ulcorner"{very near end}] \& X_1 \ar[d,hook]\arrow[two heads]{ur}{\sim} \& \\
\& B'  \arrow{r} \&  Y' \ar[dr] \& \\
B \arrow[hook]{ur}{\sim} \arrow{rrr} \& \& \& Y \, . \\
\end{tikzcd}
\]

Finally we factor $Y' \rightarrow Y$ as a cofibration followed by an acyclic fibration $Y' \hookrightarrow Y_1 \overset{\sim}{\twoheadrightarrow} Y$ to get a bifibrant object $Y_1$, then we factor the map $B' \rightarrow Y_1$ as an acyclic cofibration followed by a fibration $B' \overset{\sim}{\hookrightarrow} B_1 \twoheadrightarrow Y_1$ to get a bifibrant object $B_1$ and we obtain a diagram with an inner square as claimed above.\end{proof}

Before going further we briefly recall the well-known:

\begin{retract_lemma}\label{retract_lemma}
In any category, if there is a factorization $f=pi$ and $f$ has the right lifting property against $i$, then $f$ is a retract of $p$. Dually if $f$ has the left lifting property against $p$ then $f$ is a retract of $i$.
\end{retract_lemma}
\begin{proof} We prove the first claim. The lift in the square left below:

\[ \begin{tikzcd}
 X \ar[d,"i"] \ar[r,equal,] & X \ar[d,"f"] & X \ar[d,"f"] \ar[r,"i"] & Y \ar[d,"p"]  \ar[r,"w"] & X \ar[d,"f"]  \\
Y \ar[r,"p"] \ar[ur,dotted,"w"] & Z & Z \ar[r,equal]  &   Z \ar[r,equal] & Z \\
\end{tikzcd}\]

 produces the map to complete the retract diagram right above.\end{proof}

The next proposition, and more specifically the fact that any weak model structure satisfies Condition~\ref{Prop_triv=equiv:acycCof=equiv} and \ref{Prop_triv=equiv:acycFib=equiv} is of the highest importance for the theory:

\begin{prop}\label{Prop_triv=equiv} Let $\Ccal$ be as in \ref{Assumption:PreWMS}, i.e. it satisfies the factorization axiom of Definition~\ref{Def::weak_model_categories} and every bifibrant object has both a path object and a cylinder object. Then the following conditions are equivalent:
\begin{enumerate}[label=(\roman*)]

\item\label{Prop_triv=equiv:isWMS} $\Ccal$ is a weak model category.

\item\label{Prop_triv=equiv:relativeCyl} Every cofibration between bifibrant objects has a relative cylinder object.

\item\label{Prop_triv=equiv:relativePath} Every fibration between bifibrant objects has a relative path object.

\item\label{Prop_triv=equiv:acycCof=equiv} A cofibration between cofibrant objects is an acyclic cofibration if and only if it is an equivalence.

\item\label{Prop_triv=equiv:acycFib=equiv} A fibration between fibrant objects is an acyclic fibration if and only if it is an equivalence.

\end{enumerate} 
\end{prop}

\begin{proof}
It is clear that \ref{Prop_triv=equiv:isWMS} $\Rightarrow $ \ref{Prop_triv=equiv:relativeCyl} and \ref{Prop_triv=equiv:relativePath}.

We prove \ref{Prop_triv=equiv:relativeCyl} $\Rightarrow$ \ref{Prop_triv=equiv:acycCof=equiv}: Acyclic cofibrations between cofibrant objects are equivalence almost by definition of equivalences. Conversely let $i\colon A \hookrightarrow B$ be a cofibration between cofibrant objects which is an equivalence, we will prove it is acyclic. Using the same replacement as in the end of the proof of Lemma~\ref{Lem_weaklifting1} it is enough to show it when $A$ and $B$ are bifibrant. Using Lemma~\ref{Lem_weaklifting1} in the square:

\[
\begin{tikzcd}[ampersand replacement=\&]
A \arrow[hook]{d}{i}[swap]{\in \Wcal} \arrow{r}{Id_A} \& A \arrow[hook]{d}{i}[swap]{\in \Wcal} \\
B \arrow{r}{Id_B} \& B \\
\end{tikzcd}
\]

\noindent gives us a retraction $r\colon B \rightarrow A$ of $i$. As a retract of an equivalence, $r$ is also an equivalence, hence we can use Lemma~\ref{Lem_weaklifting1} in the square:

\[
\begin{tikzcd}[ampersand replacement=\&]
B \coprod_A B \arrow[hook]{d} \arrow{r}{(i \circ r, Id_B)} \& B \arrow{d}{r}[swap]{\in \Wcal} \\
I_A B \arrow{r} \& A \\
\end{tikzcd}
\]

\noindent gives a homotopy $h\colon I_A B \rightarrow B$ between $i \circ r$ and $Id_B$ relative to $A$, we form the commutative diagram:

\[
\begin{tikzcd}[ampersand replacement=\&]
A \ar[rr,equal,bend left = 40] \arrow[hook]{d}{i} \arrow[hook]{r}{i} \& B \arrow[hook]{d}{\iota_0} \arrow{r}{r} \& A \arrow[hook]{d}{i} \\
B \ar[rr,equal,bend right = 40] \arrow[hook]{r}{\iota_1} \& I_A B \arrow{r}{h} \& B \\
\end{tikzcd}
\]

\noindent which shows that $A \hookrightarrow B$ is a retract of $B \hookrightarrow I_A B$, hence it has the same lifting property as the acyclic cofibration  $B \hookrightarrow I_A B$, so as it is a cofibration, it is acyclic.

By duality, we also have \ref{Prop_triv=equiv:relativePath} $\Rightarrow$ \ref{Prop_triv=equiv:acycFib=equiv}.

Then we show that \ref{Prop_triv=equiv:acycCof=equiv} $\Rightarrow$ \ref{Prop_triv=equiv:acycFib=equiv}. Let $f\colon X \rightarrow Y $ be a fibration between fibrant objects which is an equivalence. As above, we can freely assume that $X$ is cofibrant. We then factor $f$ as a cofibration followed by an acyclic fibration. By $2$-out-of-$3$ for equivalences, the cofibrations part is an equivalence and hence is acyclic by assumption, and hence has the left lifting property against $f$. The retract Lemma~\ref{retract_lemma} then implies that $f$ is a retract of the acyclic fibration part of the factorization and this concludes the proof. By duality we can in fact deduces that \ref{Prop_triv=equiv:acycCof=equiv} $\Leftrightarrow$ \ref{Prop_triv=equiv:acycFib=equiv}.

Finally assuming $\Ccal$ satisfies \ref{Prop_triv=equiv:acycCof=equiv}, and given a cofibration $A \hookrightarrow B$ from a cofibrant object to a fibrant object, we consider a cofibration/acyclic fibration factorization of the relative co-diagonal map:

\[ B \coprod_A B \hookrightarrow I_A B \overset{\sim}{\twoheadrightarrow} B \]

The composite $B \hookrightarrow I_A B \overset{\sim}{\twoheadrightarrow} B$ is the identity, hence is an equivalence. By $2$-out-of-$3$ for equivalences and the fact that acyclic fibrations are equivalences we conclude that $B \hookrightarrow I_A B$ is an equivalence and hence is acyclic. Dually, we use \ref{Prop_triv=equiv:acycFib=equiv} to construct relative path objects for fibrations from a cofibrant object to a fibrant object. Hence showing that the two equivalent conditions~\ref{Prop_triv=equiv:acycFib=equiv} and \ref{Prop_triv=equiv:acycCof=equiv} imply \ref{Prop_triv=equiv:isWMS}.\end{proof}

\begin{cor} Let $\Ccal$ be a weak model category.
\begin{enumerate}[label=(\roman*)]
\item Let $X$ be a cofibrant object of $\Ccal$, then a map $f$ between fibrant or cofibrant objects in $X/\Ccal$ is an equivalence if and only if it is an equivalence in $\Ccal$.

\item Let $X$ be a fibrant object of $\Ccal$, then a map between fibrant or cofibrant objects in $\Ccal/X$ is an equivalence if and only if it is an equivalence in $\Ccal$.
\end{enumerate}

\end{cor}

\begin{proof}
Let $f\colon  Z \rightarrow Y$ be a map from a cofibrant object to a fibrant object in $X/\Ccal$. Let $f\colon Z \hookrightarrow Z' \overset{\sim}{\twoheadrightarrow} Y$ be a factorization of $f$ as a cofibration followed by an acyclic fibration. $f$ is an equivalence in $\Ccal$ (resp. in $X/\Ccal$) if and only if the cofibration part is in fact an acylic cofibration in $\Ccal$ (resp. in $X/\Ccal$), but acyclic cofibrations in $\Ccal$ and in  $X/\Ccal$ are the same things and this proves the result in the case where the source of $f$ is cofibrant and the target of $f$ is fibrant. If the domain of $f$ is instead fibrant we pre-compose it with an acyclic fibrations of cofibrant domain and if $f$ has cofibrant target we post-composite with an acyclic cofibration with fibrant domain to go back to the previous case.\end{proof}

\subsection{Equivalent definitions}
\label{sub_section_equivalentDefinitions}

\begin{assumption} In this subsection, $\Ccal$ is a category with cofibrations and fibrations as in Definition~\ref{def:class_of_cof}, which satisfies the factorization axiom of Definition~\ref{Def::weak_model_categories}, i.e. every arrow from a cofibrant object to a fibrant object can be factored both as a cofibration followed by an acyclic fibration and as an acyclic cofibration followed by a fibration. We will give equivalent fomulations for the other axioms of weak model categories.
\end{assumption}

\begin{prop}\label{Prop_alternateCond1}
 The following conditions are equivalent:
\begin{enumerate}[label=(\roman*)]

\item\label{Prop_alternateCond1:(1)factor} For any bifibrant object $A$ and any factorization $Id_A\colon A \hookrightarrow B \overset{\sim}{\twoheadrightarrow} A$ of the identity of $A$ as a cofibration followed by an acyclic fibration, the cofibration is acyclic.

\item\label{Prop_alternateCond1:rest_factor} There is a class $\Fcal$ of acyclic fibrations in $\Ccal$ such that any arrow from a cofibrant object to a fibrant object can be factored as a cofibration followed by an arrow in $\Fcal$, and condition \ref{Prop_alternateCond1:(1)factor} holds when the acyclic fibration part of the factorization is in $\Fcal$.

\item\label{Prop_alternateCond1:(2)gen_cyl} Any cofibration $A \hookrightarrow B$ with $A$ cofibrant and $B$ fibrant admits a relative cylinder object.

\item\label{Prop_alternateCond1:(3)restr_cyl} Any cofibration $A \hookrightarrow B$ between bifibrant objects admits a relative cylinder object.

\end{enumerate}

\end{prop}

And this can be dualized for the existence of path objects.

\begin{proof}
The implication $\ref{Prop_alternateCond1:(1)factor} \Rightarrow \ref{Prop_alternateCond1:rest_factor}$ and $\ref{Prop_alternateCond1:(2)gen_cyl} \Rightarrow \ref{Prop_alternateCond1:(3)restr_cyl}$ are immediate (one takes $\Fcal$ to be the class of all acyclic fibrations).

$\ref{Prop_alternateCond1:rest_factor} \Rightarrow \ref{Prop_alternateCond1:(2)gen_cyl}$: Consider a factorization $B \coprod_A B \hookrightarrow I_A B \overset{\in \Fcal}{\twoheadrightarrow} B$. The composite  $B \hookrightarrow I_A B \overset{\in \Fcal}{\twoheadrightarrow} B$ is a factorization of the identity of $B$, which is bifibrant, as in condition $\ref{Prop_alternateCond1:rest_factor}$, hence $B \hookrightarrow I_A B $ is an acyclic cofibration. This proves $\ref{Prop_alternateCond1:(2)gen_cyl}$

$\ref{Prop_alternateCond1:(3)restr_cyl} \Rightarrow \ref{Prop_alternateCond1:(1)factor}$: Given $A \hookrightarrow B \overset{\sim}{\twoheadrightarrow} A$ a factorization of the identity of a bifibrant object $A$, we denote by $r \colon  B \overset{\sim}{\twoheadrightarrow} A \hookrightarrow B$ the idempotent induced on $B$. We obtain a (dotted) diagonal filling $h$ in the square on the left below:

\[
\begin{tikzcd}
B \coprod_A B \arrow{rr}{(r,Id_B)} \arrow[hook]{d} & & B \arrow[two heads]{d}{\sim}  & & A \ar[rr,bend left = 40,equal] \ar[d,hook] \ar[r,hook] & B \ar[d,hook,"i_1"] \ar[r,->>] & A \ar[d,hook] \\
I_A B \arrow{r} \arrow[dotted]{rru}{h} & B \arrow[two heads]{r}{\sim} & A  & & B \ar[rr,bend right=40,equal] \ar[r,hook,"i_2"] & I_A B \ar[r,"h"] & B 
\end{tikzcd}\]

\noindent which then fits in the retract diagram on the right above, showing that $A \hookrightarrow B$ is a retract of $B \hookrightarrow I_A B$, and hence is an acyclic cofibration by \cref{lem:acyclicCof_basics}.\end{proof}

\begin{prop}\label{Prop_alternateCond2}

Assume furthermore that $\Ccal$ satisfies the cylinder axiom of Definition~\ref{Def::weak_model_categories}; i.e. the equivalent conditions of proposition~\ref{Prop_alternateCond1}. The following conditions are equivalent:

\begin{enumerate}[label=(\roman*)]

\item\label{Prop_alternateCond2:(1)PO} $\Ccal$ is a weak model category, i.e. any fibration from a cofibrant object to a fibrant object admits a relative strong path object.

\item\label{Prop_alternateCond2:(2)Rcancellation} If $A \overset{i}{\hookrightarrow} B \overset{j}{\hookrightarrow} C$ are two cofibrations between bifibrant objects, such that $i$ and $j \circ i$ are acyclic then $j$ is acyclic.

\item\label{Prop_alternateCond2:restricted_Rcancellation} There is a class $\Jcal$ of acyclic cofibrations such that any arrow from a cofibrant object to a fibrant object can be factored as an arrow in $\Jcal$ followed by a fibration, and condition \ref{Prop_alternateCond2:(2)Rcancellation} holds when we further assume that $i \in \Jcal$.

\end{enumerate}

\end{prop}

\begin{proof}
We proved in \ref{Prop_triv=equiv} that in a weak model category acyclic cofibrations between cofibrant objects are exactly the cofibrations that are invertible in the homotopy category, hence $\ref{Prop_alternateCond2:(2)Rcancellation}$ holds in any weak model category. The implication $\ref{Prop_alternateCond2:(2)Rcancellation} \Rightarrow \ref{Prop_alternateCond2:restricted_Rcancellation}$ is immediate. We now assume $\ref{Prop_alternateCond2:restricted_Rcancellation}$, and consider a factorization $A \overset{\in \Jcal}{\hookrightarrow} B \overset{p}{\twoheadrightarrow} A$ of the identity of a bifibrant object $A$, and (following the dual \cref{Prop_alternateCond1}) we will prove that $p$ is acyclic. We further factorize $p$ as  $B \hookrightarrow C \overset{\sim}{\twoheadrightarrow} A$, which gives us a factorization of the identity of $A$ as:

\[ A \overset{\sim}{\underset{\in \Jcal}{\hookrightarrow}} B \hookrightarrow C \overset{\sim}{\twoheadrightarrow} A \]

As $\Ccal$ has strong cylinder objects, it also satisfies condition $\ref{Prop_alternateCond1:(1)factor}$ of Proposition~\ref{Prop_alternateCond1}, hence the composite cofibration $A \hookrightarrow C$ is an acyclic cofibration and hence, because of $\ref{Prop_alternateCond2:restricted_Rcancellation}$, the cofibration $B \hookrightarrow C$ is acyclic. The retract lemma (\ref{retract_lemma}) then shows that $p$ is a retract of the acyclic fibration $C \overset{\sim}{\twoheadrightarrow} A$ and hence is also acyclic.\end{proof}

\begin{remark}
Proposition~\ref{Prop_alternateCond2} gives a characterization of weak model categories, which, except for the factorization axiom, only involves the cofibrations and the acyclic cofibrations. As an application of this, if we start from a given weak model structure and one modifies its class of fibrations in a way that do not change the class of acyclic cofibrations and so that the factorization axiom is preserved, we still have a weak model structure.

For example, one can take the closure of the class of fibrations under retract, or take them to be ``all the arrows having the right lifting property against cofibrations between cofibrant objects'', as soon as these are well defined in the logical framework, and still forms a class of fibrations. The same remark applies dually to modification of the class of cofibrations.
\end{remark}

Let us also recall:

\begin{prop}\label{prop:Non-relative_on_one_side_isEnough}
Assume that $\Ccal$ satisfies the cylinder axiom and that every bifibrant object of $\Ccal$ admits a strong path object. Then $\Ccal$ is a weak model category.\end{prop}

\begin{proof}Such a category has both cylinder and path objects for all bifibrant objects, hence we can apply Proposition~\ref{Prop_triv=equiv} and conclude from the fact that it satisfies the cylinder axiom that it is a weak model category. \end{proof}

Finally, we observe that constructiong ``path objects'' \emph{without units} is enough to get actual path objects:

\begin{lemma}[``Self-composed span trick'']\label{lem:Span_trick}
Let $X$ be a fibrant object in $\Ccal$. Assume that there is a fibrant object $X'$, and a fibration:

 \[ C \twoheadrightarrow X \times X' \]
\noindent whose components are acyclic:
\[
\begin{tikzcd}
& C \ar[dl,->>,"\sim"{swap}] \ar[dr,->>,"\sim"]  \\
X & & X' 
\end{tikzcd}
\]

Then $X$ admits a weak path object.
\end{lemma}

This lemma applies as soon as we have a class of fibrations and acyclic fibrations stable under pullbacks and compositions.

\begin{remark} Dually, there is a version for weak cylinder objects constructed out of cospans of acyclic cofibrations, and applying the results in (co)slices of $\Ccal$ automatically gives a version of the statement for relative weak path objects and relative weak cylinder objects. For example, if $A \hookrightarrow B$ is a cofibration that fits into a diagram:

\[
\begin{tikzcd}
  A \ar[r,hook] \ar[d,hook] & B \ar[d,hook,"\sim"] \\
B' \ar[r,hook,"\sim"] & I 
\end{tikzcd}
\]

\noindent such that the map $B \coprod_A B' \rightarrow I$ is a cofibration, then the cofibration $A \hookrightarrow B$ admits a relative weak cylinder object.
\end{remark}

\begin{proof} $P=C \times_{X'} C$ is a weak path object, with $C$ as reflexivity witness object:

\[\begin{tikzcd}
C \arrow[rd,"\Delta"] \arrow[rrd, "Id_C", bend left] \arrow[rdd, "Id_C"', bend right] &                                                                                     &                                                                                     &   \\
                                                                             & P \ar[dr,phantom,"\lrcorner"{very near start}]\arrow[d, "\sim" description, two heads] \arrow[r, "\sim" description, two heads] & C \arrow[d, "\sim" description, two heads] \arrow[r, "\sim" description, two heads] & X \\
                                                                             & C \arrow[r, "\sim" description, two heads] \arrow[d, "\sim" description, two heads] & X'                                                                                  &   \\
                                                                             & X                                                                                   &                                                                                     &  
\end{tikzcd} \qquad
\begin{tikzcd}
  C \ar[d,two heads,"\sim"{description}] \ar[r,"\Delta"] & P \ar[d,two heads] \\
X \ar[r] & X \times X 
\end{tikzcd}\]\end{proof}

\subsection{Weak Quillen functors and Quillen equivalences}
\label{subsec:Quillenfunctors}

In this subsection we introduce ``Quillen pairs'' and ``Quillen equivalences'', which are the natural notion of morphisms and equivalences between weak model categories. For classical Quillen model categories they are defined as pairs of adjoint functors  $F\colon \Ccal \rightleftarrows \Dcal\colon G$ satisfying some conditions, but in the weak context it is natural to only ask for the left adjoint functor $F$ to be defined on cofibrant objects and for the right adjoint functor $G$ to be defined on fibrant objects. More precisely: 

\begin{definition}
A \emph{weak Quillen pair} $F\colon \Ccal \rightleftarrows \Dcal\colon G$ between two weak model categories $\Ccal$ and $\Dcal$ is a pair of functors $F\colon \Ccal^{\cofe} \rightarrow \Dcal^{\cofe}$ and $G\colon \Dcal^{\fibe} \rightarrow \Ccal^{\fibe}$ such that:

\begin{enumerate}

\item $F$ and $G$ are ``adjoint'' in the sense that there is an isomorphism 

\[ Hom_{\Dcal}(F(X),Y) \simeq Hom_{\Ccal}(X,G(Y)) \]

 natural in $X \in \Ccal^{\cofe}$ and $Y \in \Dcal^{\fibe}$.

\item $F$ send cofibrations to cofibrations.
\item $G$ send fibrations to fibrations.
\end{enumerate}

\end{definition}

$F$ is called a left (weak) Quillen functor and $G$ a right (weak) Quillen functor. In the rest of the paper we will omit the ``weak'' and just talk about Quillen pair and Quillen functors.

\begin{example}
\label{pullbackpushoutQuillenpair}Let $\Ccal$ be a weak model category and let $X$ and $Y$ be cofibrant objects and $f\colon  X \rightarrow Y$ a map. There is a Quillen pair:

 \[ P_f\colon X/\Ccal \rightleftarrows  Y/\Ccal\colon  U_f \]

Where $P_f$ is the functor which takes a cofibrant $X \hookrightarrow Z$ to its pushout $Y \hookrightarrow Z \coprod_X Y$ and $U_f$ takes a fibrant $Z$ with a map $Y \rightarrow Z$ to the composite $X \rightarrow Y \rightarrow Z$.

This example is the main reason why we do not ask Quillen pairs to be globally defined adjoint functors: As we do not assume that $\Ccal$ has all finite colimits, but only pushout along cofibrations, the functor $P_f$ is only defined for cofibrant objects of $X/\Ccal$.

There is a dual version: if $X$ and $Y$ are fibrant objects of $\Ccal$ and  $f\colon X \rightarrow Y$ is any map there is a Quillen pair:

\[ U_f \colon  \Ccal/X \rightleftarrows \Ccal/Y \colon  P_f \]

Where $U_f$ send any cofibrant object $Z \rightarrow X$ to the composite $Z \rightarrow X \rightarrow Y$ and $P_f$ is the pullback functor taking a fibration $Z \twoheadrightarrow Y$ to its pullback $Z \times_Y X \twoheadrightarrow X$.

\end{example}

 \begin{prop}\label{Prop_QuillenFunctorHomotopy}
Let $F \colon  \Ccal \rightleftarrows \Dcal \colon  G$ be a Quillen pair. Then the two functors:

\[F\colon  \Ccal^{\cofe} \rightarrow \Dcal^{\cofe} \qquad G\colon \Dcal^{\fibe} \rightarrow \Ccal^{\fibe} \]

Both send equivalences to equivalences and induces functors:

 \[ Ho(F)\colon  Ho(\Ccal^{\cofe}) \rightarrow Ho(\Dcal^{\cofe}) \qquad Ho(G)\colon  Ho(\Dcal^{\fibe}) \rightarrow Ho(\Dcal^{\fibe}). \]

Moreover, up to the equivalences of categories of Theorem~\ref{Th_LocCatEq} $Ho(F)$ is left adjoint to $Ho(G)$ on the homotopy categories.

\end{prop}

\begin{proof}
The adjunction property between $F$ and $G$ and the fact that $G$ sends fibrations to fibrations implies that $F$ sends acyclic cofibrations to acyclic cofibrations. As $Ho(\Ccal^{\cofe})$ and $Ho(\Dcal^{\cofe})$ are localization at acyclic cofibrations this shows that $F$ induces a functor $Ho(F)\colon  Ho(\Ccal^{\cofe}) \rightarrow Ho(\Dcal^{\cofe})$. Dually $G$ induces a functor $Ho(\Dcal^{\fibe}) \rightarrow Ho(\Ccal^{\fibe})$. This shows in particular that $F$ and $G$ send equivalences to equivalences.

Now given $X \in \Ccal^{\cofe}$ and $Y \in \Dcal^{\fibe}$, the adjunction isomorphism $Hom(X,G(Y)) \simeq Hom(F(X),Y)$ is compatible to the homotopy relation (because $F$ preserves cylinder objects and $G$ preserves path objects) hence it descend into an isomorphism:

\[ Hom_{Ho(\Ccal)}(X,G(Y)) \simeq Hom_{Ho(\Ccal)}(F(X),Y) \]

We easily check that this isomorphism is natural on the homotopy category (for example by restricting to $X$ and $Y$ bifibrant), and this concludes the proof.\end{proof}

\begin{cor}
In a weak model category:
\begin{itemize}
\item Pushouts along a cofibration between cofibrant objects send equivalences between cofibrant objects to equivalences.

\item Pullbacks along fibrations between fibrant objects send equivalences between fibrant objects to equivalences.

\end{itemize}

\end{cor}

\begin{proof} This is Proposition~\ref{Prop_QuillenFunctorHomotopy} applied to \cref{pullbackpushoutQuillenpair}.\end{proof}

\begin{prop}\label{prop:QuillenEquiv}
For a Quillen pair $F\colon \Ccal \rightleftarrows \Dcal \colon  G$ between two weak model categories the following conditions are equivalent:

\begin{enumerate}[label=(\roman*)]

\item $Ho(F) \colon  Ho(\Ccal^{\cofe}) \rightarrow Ho(\Dcal^{\cofe})$ is an equivalence of categories.

\item $Ho(G) \colon  Ho(\Dcal^{\fibe}) \rightarrow Ho(\Ccal^{\fibe})$ is an equivalence of categories.

\item For any $X \in \Ccal^{\cofe}$ and $Y \in \Dcal^{\fibe}$ a map $f\colon X \rightarrow G(Y)$ is an equivalence if and only if its adjoint transpose $F(X) \rightarrow Y$ is an equivalence.

\item For any $X \in \Ccal^{\cofe}$ the map $X \rightarrow G(F(X)^{\fibe})$ where $F(X) \overset{\sim}{\hookrightarrow} F(X)^{\fibe}$ is a fibrant replacement of $F(X)$ is an equivalence, and the dual condition holds for any $Y \in \Dcal^{\fibe}$.

\item For any $X \in \Ccal^{\cofe}$ the map $X \rightarrow G(F(X)^{\fibe})$ as in (iv) is an equivalence, and $G$ detects equivalences between (bi)fibrant  objects, i.e. if $f$ is a morphism in $\Dcal^{\fibe}$ (or even just $\Dcal^{\bfe}$) such that $G(f)$ is an equivalence then $f$ is also an equivalence.

\end{enumerate}

\end{prop}

A Quillen pair satisfying these conditions is called a (weak) Quillen equivalence. Also conditions $(iv)$ and $(v)$ do not depend on the choice of the fibrant replacement of $F(X)$ or on the cofibrant replacement of $G(Y)$ because of Proposition~\ref{Prop_QuillenFunctorHomotopy}.

\begin{proof}
The equivalence of $(i)$ and $(ii)$ is immediate form the adjunction property satisfied by $Ho(F)$ and $Ho(G)$. They imply $(iii)$ because if $f \colon  X \rightarrow G(Y)$ is an equivalence if and only if it is invertible in $Ho(\Ccal^{\bothe})$ and its adjoint transpose $f^*\colon  F(X) \rightarrow Y$ has been shown in the proof of Proposition~\ref{Prop_QuillenFunctorHomotopy} to represent the adjoint transpose of $f$ by the adjunction between $Ho(F)$ and $Ho(G)$, hence if those are equivalences of categories, $f$ will be an equivalence if and only if $f^*$ is an equivalence. Condition $(iii)$ immediately implies $(iv)$, and $(iv)$ implies $(i)$ and $(ii)$ as the maps described represent in the homotopy category the unit and co-unit of the adjunction between $Ho(F)$ and $Ho(G)$ hence asking them to be an equivalence makes $Ho(F)$ and $Ho(G)$ inverse of each other. So conditions $(i)$ to $(iv)$ are all equivalent.

Condition $(v)$ implies that the unit of the adjunction $Ho(F)$ and $Ho(G)$ is an isomorphism and that $Ho(G)$ is conservative, which by a classical category theoretic argument implies that $Ho(F)$ and $Ho(G)$ are equivalences, i.e. $(i)$ and $(ii)$. Conversely, condition $(iv)$ contains the first half of condition $(v)$ and the second half of condition $(v)$ follows from $(ii)$.\end{proof}

The last condition of \cref{prop:QuillenEquiv} can be further simplified:

\begin{prop}\label{prop:Gconservatif} For a right Quillen functor $G:\Dcal^\fibe \to \Ccal^\fibe$, the following conditions are equivalent:

\begin{enumerate}[label=(\roman*), noitemsep,topsep=2pt]

\item\label{prop:Gconservatif:conservatif} $G$ detects equivalences between (bi)fibrant objects, i.e. $Ho(G)$ is conservative.

\item\label{prop:Gconservatif:detect_acyclic} For any fibration between (bi)fibrant objects $p$, if $G(p)$ is acyclic then $p$ is acyclic.

\end{enumerate}

\end{prop}

\begin{proof}The implication $\ref{prop:Gconservatif:conservatif} \Rightarrow \ref{prop:Gconservatif:detect_acyclic}$ follows immediately from \cref{Prop_triv=equiv}.$\ref{Prop_triv=equiv:acycFib=equiv}$. We assume $\ref{prop:Gconservatif:detect_acyclic}$. Let $f: X \to Y$ an arrow between bifibrant objects such that $G(f)$ is an equivalence. Consider $X \overset{\sim}{\hookrightarrow} Z \overset{p}{\twoheadrightarrow} Y$ an (acyclic cofibration,fibration) factorization of $f$,  $G$ sends the acyclic cofibrations to an equivalence because of \cref{Prop_QuillenFunctorHomotopy}, hence $G(p)$ is an equivalence by $2$-out-of-$3$, hence it is acyclic, and hence $p$ is acyclic by $\ref{prop:Gconservatif:detect_acyclic}$, which proves that $f$ is an equivalence.\end{proof}

\makeatletter 
\renewcommand{\thetheorem}{\thesection.\arabic{theorem}} 
\@addtoreset{theorem}{section}
\makeatother

\section{Cisinski-Olschok type theorems}

\label{sec_monoidalcase}

The goal of this section is to provide simpler criterions for constructing a weak model structure out of two weak factorization systems in the special case where either:

\begin{enumerate}[label=(\roman*), noitemsep,topsep=2pt]

\item the underlying category has a well behaved and left adjoint weak cylinder functor (Theorem~\ref{Th::WMS_from_cylinder_path}),

\item the underlying category is monoidal closed and usual compatibility conditions between the monoidal structure and the factorization system are satisfied (Theorem~\ref{Th_tensorWMS}),

\item the underlying categories is enriched in a category that already has a weak model structure and this enrichment is compatible with the factorization system (Theorem~\ref{Th_tensorWMS} as well).

\end{enumerate}

The special case (i) is similar in form to M.~Olschok's generalization of D-C.~Cisinski's theory from \cite{cisinski2002theories}, see more precisely theorem $3.16$ from \cite{olschok2011left}.  In \cite{henry2019CombWMS} we will show how we can recover and generalize Olschok's theorem from Theorem~\ref{Th::WMS_from_cylinder_path}. The case (iii) can be seen as a generalization of the main result of \cite{lee2015building}, which essentially corresponds to a weak form of our Theorem~\ref{Th_tensorWMS} in the special case of a simplicially enriched category.

This section heavily relies on properties of the corner-product, recalled in Definition~\ref{pushout-productDef}, and what is often called the Joyal-Tierney calculus, introduced in the appendix of \cite{joyal2006quasi}, which we review in Appendix~\ref{Subsection_Joyal_tierney_calculus}.

\begin{assumption}In all this section, as well as in all the example treated in the rest of the paper, we will consider a category $\Ccal$ with a set of generating cofibrations $I$ and a set of ``generating anodyne maps''or ``Pseudo-generating\footnote{This terminology comes from section 9.9 of \cite{simpson2011homotopy}. In a model category, it refers to the fact that this set of acyclic cofibration is only sufficient to characterize fibrations between fibrant objects, but not necessarily all fibrations.} acyclic cofibrations'' $J$. Cofibrations will be the $I$-cofibrations and fibrations will be the $J$-fibrations in the sense of \cref{def_IfibCof}. It will always be the case that maps in $J$ are $I$-cofibrations.
\end{assumption}

\begin{theorem}\label{Th_tensorWMS}
Let $\Acal$ and $\Ccal$ be two complete and cocomplete categories such that:

\begin{enumerate}[label=(\roman*)]

\item There is a functor $\odot\colon  \Acal \times \Ccal \rightarrow \Ccal$ divisible on both sides as in \ref{DivisibleOnbothSide}.

\item\label{Th_tensorWMS:condWFS}$\Ccal$ is endowed with two classes of maps $I$ and $J$  such that ($I$-cof,$I$-fib) and ($J$-cof,$J$-fib) (as in Definition~\ref{def_IfibCof}) forms weak factorization systems\footnote{See the discussion of the small object argument in Appendix~\ref{section_the_small_object_arguments} for methods to show this type of conditions.}.

\item $\Acal$ is endowed with two classes of maps $I_{\Acal}$ and $J_{\Acal}$.

\item $J \subset I \textsc{-cof}$ and $J_{\Acal} \subset I_{\Acal}\textsc{-cof}$.

\item We have:

\[ I_{\Acal} \corner{\odot} I \subset I\textsc{-cof} \]

\noindent where $\corner{\odot}$ denotes the corner-product, or pushout-product, as defined in \ref{pushout-productDef}.

\item Any map in $I_{\Acal} \corner{\odot} J$ or in $J_{\Acal} \corner{\odot} I$ has the left lifting property with respect to all $J$-fibrations between $J$-fibrant objects.

\item There is an $I_{\Acal}$-cofibrant object $\Ib$ in $\Acal$ such that $\Ib \odot \_$ is isomorphic to the identity endofunctor of $\Ccal$.

\item\label{Th_tensorWMS:Cond_cylinder} There is in $\Acal$ a diagram of the form:

\[
\begin{tikzcd}[ampersand replacement=\&]
\Ib \coprod \Ib \arrow{d}{\triangledown} \arrow[hook]{r}{i} \& C \arrow{d} \\
\Ib \arrow[hook]{r}{\sim} \& D
\end{tikzcd}
\] 

\noindent such that $i$ is an $I_{\Acal}$-cofibration, and both the map $\Ib \hookrightarrow D$ and the first map $\Ib \rightarrow C$ are acyclic cofibrations, in the sense that they are $I_{\Acal}$-cofibrations with the left lifting property with respect to all $J_{\Acal}$-fibrations between $J_{\Acal}$-fibrant objects.
\end{enumerate}

Then there is a weak model structure on $\Ccal$ such that the fibrations between fibrant objects are the $J$-fibrations and the cofibrations between cofibrant objects are the $I$-cofibrations.
\end{theorem}

\begin{proof}
 $I$-cofibrations and $I_A$-cofibrations will be called cofibrations, $J$-fibrations and $J_A$-fibrations will be called fibrations, and $J$-cofibrations and $J_A$-cofibrations will be called anodyne maps. As in \cref{Def_trivcofib}, $I_A$ or $I$-cofibrations with the left lifting property against all $J_A$ or $J$-fibrations between $J_A$ or $J$-fibrant objects will be called acyclic cofibrations, and similarly for acyclic fibrations.

Any $I$-fibration is automatically a fibration, because $J \subset I \textsc{-cof}$, and in fact an $I$-fibration is an acyclic fibration as it has the right lifting property against all maps in $I$, hence against all cofibrations as well. Similarly, $J$-cofibrations (i.e. anodyne morphisms) are acyclic cofibrations.

The existence of weak factorization systems $(I\text{-cof}, I\text{-fib})$ and $(J\text{-cof}, I\text{-fib})$ implies that $\Ccal$ satisfies the factorization axioms. We use the symbol $a \pitchfork p$ to denote that $a$ has the left lifting property against $p$ as in Appendix~\ref{Subsection_Joyal_tierney_calculus}.

\bigskip

\textbf{Claim~1:}If $a\colon A \hookrightarrow B$ and $i\colon X \hookrightarrow Y $ are cofibrations between cofibrant objects in $\Acal$ and $\Ccal$ respectively, then $a \corner{\odot} i$ is also a cofibration between cofibrant objects.  Indeed, the assumption $I_{\Acal} \corner{\odot} I \subset I\textsc{-cof}$, together with Lemma~\ref{Lem_mainJTcalculus}, shows that such a map is always a cofibration and hence Lemma~\ref{Lem_JTcalculusCofDom} shows that its domain is cofibrant.

\bigskip

\textbf{Claim~2:} If $i\colon A \hookrightarrow B$ is a cofibration between cofibrant objects in $\Acal$ and $f\colon X \rightarrow Y$ is a fibration between fibrant objects in $\Ccal$ then $\cornerl{ i \backslash f}$  (as defined in \ref{pushout-productDef}) is a fibration between fibrant objects.

Indeed, the assumptions of the theorem show that $I_{\Acal} \corner{\odot} J \pitchfork f $ for $f$ any fibration between fibrant objects. Hence for any such $f$ and $i$ any $I_{\Acal}$-cofibration we have $J \pitchfork \cornerl{i \backslash a}$, i.e. $ \cornerl{ i \backslash a }$  is a $J$-fibration. We apply Lemma~\ref{Lem_JTcalculusCofDom} to the bi-functor $ \_ \backslash \_  \colon  \Acal \times \Ccal^{op} \rightarrow \Ccal^{op}$ to conclude that $\cornerl{ i \backslash a }$ is always a fibration between fibrant objects.

\bigskip

\textbf{Claim~3:} If $i\colon A \hookrightarrow B$ is a cofibration between cofibrant objects in $\Ccal$ and $f\colon X \rightarrow Y$ is a fibration between fibrant objects in $\Ccal$ then $\cornerl{ f/i }$ is a fibration between fibrant objects in $\Acal$.

Indeed, the assumptions of the theorem show that $J_{\Acal} \corner{\odot} I \pitchfork f$ for $f$ any fibration between fibrant objects. This shows that for $i$ and $f$ as in the claim $J_{\Acal} \pitchfork \cornerl{ f/i }$, i.e. that $\cornerl{f/i}$ if a $J_{\Acal}$-fibration. Then applying Lemma~\ref{Lem_JTcalculusCofDom} to the bi-functor $ \_ / \_  \colon   \Ccal^{op} \times \Ccal \rightarrow \Acal^{op}$ to get that $\cornerl{ i \backslash a }$ is always a $J$-fibration between $J$-fibrant objects.

\bigskip

From these observations we deduce:

\bigskip

\textbf{Claim~4:} If $a\colon A\hookrightarrow B$ and $i\colon X \hookrightarrow Y$ are two cofibrations between cofibrant objects in $\Acal$ and $\Ccal$ respectively and one of them is an acyclic cofibration then $a \corner{\odot} i$ is an acyclic cofibration.

\bigskip

It is a cofibration by Claim~1, so we need to prove that $a \corner{\odot} i \pitchfork p$ for $p$ a fibration between fibrant objects. This is equivalent to $a \pitchfork \cornerl{ p /i }$ and to $i \pitchfork \cornerl{ a \backslash p }$.
If we assume for example that $a$ is an acyclic cofibration, then Claim~$3$ shows that $\cornerl{ p /i }$ is a fibration between fibrant objects and hence $a \pitchfork \cornerl{ p /i }$ do hold. If instead $i$ is an acyclic cofibration, then Claim~$2$ shows that  $i \pitchfork \cornerl{ a \backslash p }$.

\bigskip

\textbf{Claim~5:} If $j$ is an acyclic cofibration between cofibrant objects in $\Acal$ and $p$ is a fibration between fibrant objects in $\Ccal$ then $\cornerl{ j \backslash p }$ is an acyclic fibration in $\Ccal$.

\bigskip

Indeed, as $j$ is in particular a cofibration, this map is a fibration between fibrant objects by Claim~$2$. We need to prove that it has the right lifting property with respect to all cofibrations between cofibrant objects. Let $i$ be such a cofibration in $\Ccal$, we have $j \corner{\odot} i \pitchfork p$. Claim~$4$ shows that $i \pitchfork \cornerl{ j \backslash p }$ which concludes the proof.
All the other similar expected claim will of course hold as well and are obtained with the same methods, but those we proved above are the only ones needed in the rest of the proof.

\bigskip

We can now construct relative weak cylinder for cofibrations and relative weak path objects for fibrations between bifibrant objects, which is sufficient to conclude because of Proposition~\ref{Prop_alternateCond1}.

Let $i\colon A \hookrightarrow B$ be a cofibration between bifibrant objects of $\Ccal$ then the map $w \coloneqq (\Ib \coprod \Ib \hookrightarrow \Ical ) \corner{\odot} (A \hookrightarrow B)$ is a cofibration:

\[w\colon  (B \coprod B) \coprod_{A \coprod A} ( \Ical \odot A ) \hookrightarrow \Ical \odot B \]

Moreover, the map $A \hookrightarrow D \odot A$ is an acyclic cofibration by Claim~$4$ because it is $(\Ib \hookrightarrow D) \corner{\odot} (0 \hookrightarrow A)$ and $(\Ib \hookrightarrow D)$ is an acyclic cofibration. As $A$ is fibrant, it admits a retraction $r\colon  D \odot A \rightarrow A$.

In particular, we define a map from the domain of $w$ to $B \coprod_A B$ by sending $B \coprod B$ to $B \coprod_A B$, and $\Ical \odot A$ is sent to $A$ by the map $r$ above (pre-composed with $I \odot A \rightarrow D \odot A$ and then to $B \coprod_A B$ by the natural map $A \rightarrow B \coprod_A B$. We call the pushout of $w$ along this map $I_A B$, which comes with a cofibration:

\[ B \coprod_A B \hookrightarrow I_A B \]

\noindent and we will show that it is a relative cylinder object with the expected properties.

The first map $B \hookrightarrow I_A B$ can be checked to be the pushout of the map $B \coprod_A \Ical \odot A \rightarrow \Ical \odot B$ (induced by the first map $B \hookrightarrow \Ical \odot B$ and the natural map $\Ical \odot A \rightarrow \Ical \odot B$) along the map $B \coprod_A \Ical \odot B$ (induced by the identity on $B$ and the obvious map $\Ical \odot A  \rightarrow \rightarrow B$). But the map $B \coprod_A \Ical \odot A \rightarrow \Ical \odot B$ mentioned above is exactly $(\Ib \hookrightarrow \Ical) \corner{\odot} (A \hookrightarrow B)$ hence it is an acyclic cofibration by Claim~$4$, and this show that the first map $B \hookrightarrow I_A B$ is indeed an acyclic cofibration.

Finally the map $(\Ib \hookrightarrow D ) \corner{\odot} (A \hookrightarrow B)$ is also an acyclic cofibration because of Claim~$4$. This map is:

\[ (D \odot A) \coprod_{A} B \hookrightarrow D \odot B. \]

If we consider the map $(D \odot A) \coprod_{A} B \rightarrow B$ induced by the identity of $B$ and the map $D \odot A \overset{r}{\rightarrow} A \rightarrow B$, then as $B$ is fibrant we can extend it to a map $r' \colon  D \odot B \rightarrow B$, which induces a map, also denoted $r'\colon  \Ical \odot B \rightarrow B$, by construction, this map is $r$ when restricted to $\Ical \odot A \hookrightarrow \Ical \odot B$ and is the co-diagonal map when restricted to $B \coprod B \hookrightarrow \Ical \odot B$, those properties exactly shows that $r'$ defines by the universal property of the pushout defining $I_A B$ a map $I_A B \rightarrow B$ which factor the codiagonal:

\[ B \coprod_A B \hookrightarrow I_A B \rightarrow B. \]

At this point we could, almost by the exact dual argument, construct a relative path object. But by Proposition~\ref{prop:Non-relative_on_one_side_isEnough} it is enough to show that every bifibrant object have a strong path object, or (by Remark~\ref{rk:Exist_strong_Cyl=exists_weak_cyl}) that every fibrant object has a weak path object. This is directly produced by applying $\left( \_ \backslash X \right)$ to the diagram in (\ref{Th_tensorWMS:Cond_cylinder}):

\[ 
\begin{tikzcd}
 D \backslash X \ar[d] \ar[r,two heads,"q","\sim"{swap}] & X \ar[d,"\Delta"] \\
 \Ical \backslash X \ar[r,two heads,"p"] & X \times X 
\end{tikzcd}
\]

Where $p$ is a fibration by Claim~$2$, and $q$ as well as the composite $\Ical \backslash X \twoheadrightarrow X$ are acyclic fibrations by Claim~$5$, applied to the fibration $X \twoheadrightarrow 1$ and in each case the corresponding (acyclic) cofibration in $\Acal$.\end{proof}

\begin{remark}\label{rk:SpanTrick_for_tensorWMS} Using the ``self-composed span trick'' lemma of \ref{lem:Span_trick}, condition (\ref{Th_tensorWMS:Cond_cylinder}) of Theorem~\ref{Th_tensorWMS} can be replaced by the sometimes simpler condition:

\begin{itemize}
\item[\ref{Th_tensorWMS:Cond_cylinder}'] There is in $\Acal$ an $I_{\Acal}$-cofibration of the form 

\[ \Ib \coprod X \overset{i}{\hookrightarrow} C \]

\noindent such that $X$ is $I_A$-cofibrant and both the map $\Ib \hookrightarrow C$ and $X \hookrightarrow C$ are acyclic cofibrations.

\end{itemize}

Indeed, applying the dual of Lemma~\ref{lem:Span_trick} to this span will produce exactly the weak cylinder object that we need.

\end{remark}

\begin{construction}\label{constr:A_generated_by_a cylinder} We conclude with a special case of interest of our theorem. Take $\Acal$ to be the category of presheaves over the following category $\Dcal$:

\[ P \overset{(e_0,e_1)}\rightrightarrows C \]

We define:

\[ J_A = \{ e_0,e_1 \colon  P \rightrightarrows C  \}  \]
\[ I_A =  \{ \emptyset \rightarrow P,  P \coprod P \overset{e_0,e_1}{\rightarrow} C \} \]

Where we have identified the objects $P$ and $C$ with the corresponding representable functors. Following the third point of \cref{Ex::divisiblebifunctor}, a divisible bi-functor $\widehat{\Dcal} \odot \Ccal \rightarrow \Ccal$ is given by two left adjoint functors $P, C$ from $\Ccal$ to $\Ccal$ with natural transformations $e_0,e_1 \colon  P \rightrightarrows C$. We assume that $P$ is the identity endofunctor. In this special case, Theorem~\ref{Th_tensorWMS} (with the modification of Remark~\ref{rk:SpanTrick_for_tensorWMS}) reduces to:

\end{construction}

\begin{theorem}[Variant of Cisinski-Olschok's theorem]\label{Th::WMS_from_cylinder_path}
Let $\Ccal$ be category with two classes of maps $I$ and $J$ such that:

\begin{enumerate}

\item $I$ and $J$ generates weak factorization systems and $J \subset I\textsc{-cof}$ as in Theorem~\ref{Th_tensorWMS}.(\ref{Th_tensorWMS:condWFS}).

\item $\Ccal$ is endowed with a left adjoint endofunctor $X \mapsto C X$. As well as natural transformations:

\[ Id \overset{e_0,e_1}{\rightrightarrows} C \]

\item For any $i\colon A \rightarrow B \in I$ the map:

\[(B \coprod B) \coprod_{A \coprod A} C A  \rightarrow C B\]

\noindent is an $I$-cofibration.

\item For any $i\colon A \rightarrow B$ in $I$ the two maps:

\[B \coprod_A C A \rightrightarrows C B \]

\noindent have the left lifting property against all $J$-fibrations between $J$-fibrant objects.

\item For any $j \colon A \rightarrow B \in J$ the map

\[(B \coprod B) \coprod_{A \coprod A} C A  \rightarrow C B\]

\noindent has the lift lifting property against all $J$-fibrations between $J$-fibrant objects.

\end{enumerate}
\end{theorem}

\begin{remark}
Notes that there are other options for the choice of small category $\Dcal$ and hence of $\Acal = \widehat{\Dcal}$ as in Construction~\ref{constr:A_generated_by_a cylinder} that gives variation of Theorem~\ref{Th::WMS_from_cylinder_path}. Take $\Dcal$ to have three objects $P,Q,C$ with maps $P \rightarrow C$ and $Q \rightarrow C$ with only $P$ acting as the identity, this corresponds to the most general form of the (dual of the) ``span trick Lemma''~\ref{lem:Span_trick}, where we only ask to have a cospan $X \leftarrow C \rightarrow P$ in order to construct a cylinder for $X$. Alternatively, we can also\footnote{The reader can consult $3.2.2$ in the \emph{first} arXiv version of the present paper for a full statement of this form of the theorem.} use a $\Dcal$ that has the shape of the diagram in Theorem~\ref{Th_tensorWMS}.(\ref{Th_tensorWMS:Cond_cylinder}), this gives a version where we have left adjoint functor $C$ and $D$, providing functorial weak cylinder, in this case we only need to ask the first leg inclusion $P \rightarrow C$ to be acyclic instead of both.
\end{remark}

\makeatletter 
\renewcommand{\thetheorem}{\thesubsection.\arabic{theorem}} 
\makeatother

\section{Simple examples}
\label{Section_Examples}

In this section we will mostly show how the framework above applies to some very simple examples. In terms of logical background, we now need slightly stronger assumptions in order to use the small object argument. As this is a subtle matter we refer to Appendix~\ref{section_the_small_object_arguments} for a precise discussion of what this means, though we do not have the final answer to that question yet. In any case, everything below would be valid either in the internal logic of an elementary topos with a natural number object or in (CZF).

\subsection{The model structure for setoids}
\label{ex_setoids}

Here we construct a model structure corresponding to the notion of setoids as presented in Appendix~\ref{section_setoids}.

We consider $\G$ the category of oriented graphs. By a graph $X$ we mean a set of vertices $V(X)$ and a set of arrow $R(X)$, with two maps $s,t:R(x) \rightrightarrows V(X)$. We endow $\G$ with a monoidal structure defined as follows:

If $X$ and $Y$ are two graphs, we define:

\[ V(X \otimes Y) \coloneqq  V(X) \times V(Y) \]
\[ R(X \otimes Y) \coloneqq  \left[ V(X) \times R(Y) \right] \coprod \left[ R(X) \times V(Y) \right] \]

Where the source and target map $s$ and $t$ are defined by:

\[ \begin{array}{c c c}
s(x,g) \coloneqq (x,s(g) & t(x,g) \coloneqq  (x , t(g)) & \text{If $x \in V(X)$ and $g \in R(Y)$} \\
s(f,y) \coloneqq  (s(f),y) & t(f,y)=(t(f),y) & \text{If $f \in R(X)$ and $y \in V(Y)$}
\end{array} \]

So for example the graph $(x \overset{f}{\rightarrow} y) \otimes (a \overset{g}{\rightarrow} b)$ is simply:

\[
\begin{tikzcd}[ampersand replacement=\&]
(x,a) \arrow{d}[swap]{(f,a)} \arrow{r}{(x,g)} \&  (x,b) \arrow{d}{(f,b)} \\
(y,a) \arrow{r}[swap]{(y,g)} \& (y ,b)\\
\end{tikzcd}
\]

This makes the category of graphs a symetric monoidal closed category. Morphisms $X \otimes Y \rightarrow Z$ corresponds to the definition in Appendix~\ref{section_prelim_setoids} of two variables functions between setoids.

We will use our Theorem~\ref{Th_tensorWMS} to endow the category of graphs with a ``monoidal'' weak model structure. By that we mean that we will apply the theorem with $\Acal=\Ccal=\G$, the bi-functor being the tensor product and with $I=I_A$ and $J=J_A$.

The set of generating cofibrations is $I= \{i_V ,i_R \}$, with:

\[ i_V \colon  \emptyset \hookrightarrow \bullet \qquad  i_R \colon  (\bullet \quad \bullet ) \hookrightarrow ( \bullet \rightarrow \bullet) \]

The small object argument applies in its ``good'' version of \ref{EasySOA}. 

\begin{lemma}
The $I$-cofibrations are the complemented inclusions, i.e. the monomorphisms $f \colon X \rightarrow Y$ such that for all $y \in Y$ either $y \in X$ or $y \notin X$.
\end{lemma}

\begin{proof}
The generating cofibrations satisfies this condition and it is stable under retract, coproduct, pushout and composition so this proves one inclusion. Conversely note that each such levelwise complemented inclusion can be constructed by first using pushout along $i_V$ to add all the missing vertices and then pushout along $i_R$ to add all missing arrows.\end{proof}

In particular every graph is $I$-cofibrant. The $I$-fibrations are the map $f \colon X \rightarrow Y$ such that for every cell $y \in Y$ there is an $x \in X$ such that $f(x)=y$ and for every arrow $v\colon f(x) \rightarrow f(y)$ in $Y$ there is an arrow $w\colon x \rightarrow y$ such that $f(w)=v$. Also, the corner-product conditions for cofibrations are easily checked: $i_V \corner{\otimes} i_V = i_V$, $(i_V \corner{\otimes} i_R) =(i_R \corner{\otimes} i_V) = i_R$ and $i_R \corner{\otimes} i_R $ is an isomorphism.

The generating anodyne maps will be given by $J=\{j_1, j_t, j_{inv}\}$:

\[ j_1 \coloneqq  ( x ) \hookrightarrow  (x \rightarrow y) \]
\[ j_t \coloneqq  \left(\parbox{3cm}{\begin{tikzcd}[ampersand replacement=\&]
\bullet \arrow{r} \& \bullet \arrow{r} \& \bullet \\
\end{tikzcd}\vspace{-0.65cm}} \right) \hookrightarrow
\left(\parbox{3cm}{\begin{tikzcd}[ampersand replacement=\&]
\bullet \arrow[bend left = 40]{rr} \arrow{r} \& \bullet \arrow{r} \& \bullet \\
\end{tikzcd}\vspace{-0.3cm}} \right) \]

\[ j_{i} \coloneqq  \left(\parbox{1.8cm}{\begin{tikzcd}[ampersand replacement=\&]
\bullet \arrow{r} \& \bullet  \\
\end{tikzcd}\vspace{-0.65cm}} \right) \hookrightarrow
\left(\parbox{1.8cm}{\vspace{-0.2cm}\begin{tikzcd}[ampersand replacement=\&]
\bullet \arrow{r} \& \bullet \arrow[bend right=40]{l}  \\
\end{tikzcd}\vspace{-0.65cm}} \right) \]

Here again the small object argument applies without any problems and gives us a weak factorization system in $J$-cofibrations/$J$-fibrations. 
The corner-product conditions against $i_V$ are all trivial as $i_V \corner{\otimes} f = f$ for all $f$, we only need to check the corner-product of the form $i_R \corner{\otimes} j_{?}$ (and the corner-product is symmetric as the tensor product is). We have:

\[ i_R \corner{\otimes} j_1 =  \left(\parbox{1.8cm}{\begin{tikzcd}[ampersand replacement=\&] \bullet \arrow{d} \arrow{r} \& \bullet \\
\bullet \arrow{r} \& \bullet
\end{tikzcd}} \right) \hookrightarrow \left(\parbox{1.8cm}{\begin{tikzcd}[ampersand replacement=\&] \bullet \arrow{d} \arrow{r} \& \bullet \arrow{d}  \\
\bullet \arrow{r} \& \bullet
\end{tikzcd}} \right)  \]

\begin{lemma}Any $J$-fibration $f\colon X \rightarrow Y$ between $J$-fibrant objects has the right lifting property against $i_R \corner{\otimes} j_1$.\end{lemma}
\begin{proof}
A lifting square of $f\colon X \rightarrow Y$ against $i_R \corner{\otimes} j_1$ corresponds to a solid diagram of the form: 

\[ \begin{tikzcd}[ampersand replacement=\&] \bullet \arrow{d} \arrow{r} \& \bullet \arrow[dotted]{d} \\
\bullet \arrow{r} \& \bullet
\end{tikzcd} \]

\noindent in $X$, together with a dotted filling in $Y$. Using the lifting property of $X$ against $j_t$ and $j_{i}$ we can extend this diagram into:

\[ \begin{tikzcd}[ampersand replacement=\&] \bullet \arrow{d} \arrow{r} \arrow{dr} \& \bullet \arrow[bend right=40]{l} \arrow[dotted]{d} \\
\bullet \arrow{r} \& \bullet
\end{tikzcd} \]

Using the two new arrows, the dotted filling then became a solution to a lifting problem against $j_t$, and hence it can be lifted from $Y$ to $X$ using the lifting property of $f$ against $j_t$.\end{proof}

The two other corner-product map $i_R \corner{\otimes} j_t$ and $i_R \corner{\otimes} j_{i}$ are both identity map, respectively of:

\[\parbox{3cm}{\begin{tikzcd}[ampersand replacement=\&]
\bullet \arrow[bend left = 30]{rr} \arrow{r} \arrow{d} \& \bullet \arrow{d} \arrow{r} \& \bullet \arrow{d} \\
\bullet \arrow[bend right = 30]{rr} \arrow{r} \& \bullet \arrow{r} \& \bullet \\
\end{tikzcd}} \quad \text{ and } \quad \parbox{1.8cm}{\vspace{-0.22cm}\begin{tikzcd}[ampersand replacement=\&]
\bullet \arrow[bend right=40]{r} \arrow{d} \& \bullet \arrow[bend right=40]{l} \arrow{d} \\
\bullet \arrow[bend right=40]{r} \& \bullet \arrow[bend right=40]{l} 
\end{tikzcd}}\]

In order to finish the proof that the conditions of Theorem~\ref{Th_tensorWMS} are satisfied we only need to construct a weak cylinder object for the graph $\bullet$. It is given by:

\[
\begin{tikzcd}
 \left(  \bullet \coprod \bullet \right) \ar[d,hook] \ar[r] & \bullet \ar[d,hook,"j_r"]\\
\left( \bullet \rightarrow \bullet \right) \ar[r] & \left( \tikz[baseline=(A.base)]{\node (A) {$\bullet$}; \draw[->] (A) to[loop right] (A); } \right) \\
\end{tikzcd}
\]

That is the cylinder object $\Ical$ and the reflexivity witness $\Dcal$ are respectively given by:

\[ \Ical \coloneqq  \left( \bullet \rightarrow \bullet \right) \qquad \Dcal = \left( \tikz[baseline=(A.base)]{\node (A) {$\bullet$}; \draw[->] (A) to[loop right] (A); } \right) \]

\noindent with the obvious map $\bullet \coprod \bullet \hookrightarrow \Ical$ (all the other maps being the unique possible map). The first leg $\bullet \hookrightarrow \Ical$ is the map $j_1$, so it is an anodyne map. And in order to conclude we need to show that $j_r$ is an acyclic cofibration, i.e.:

\begin{lemma}
The map

\[ j_r\colon  (\bullet) \hookrightarrow  \left( \begin{tikzcd} \bullet \arrow[loop right]\end{tikzcd} \right) \]
 has the left lifting property against all $J$-fibrations $f\colon X \rightarrow Y$ between $J$-fibrant objects.
\end{lemma}

\begin{proof}
 A lifting problem for $f:X \to Y$ against $j_r$ is a vertex $v$ in $X$ together with an arrow $r\colon f(v) \rightarrow f(v)$ in $Y$. Using that $X$ is $J$-fibrant we can find in $X$ a vertex $y$ and arrows $a\colon x \rightarrow y$ and $b\colon y \rightarrow x$, we now have a lifting problem against $j_t$, which has a solution as $f$ is a fibration.\end{proof}

It is also worth noting that:

\begin{prop}\label{Prop_fibrantGraph=setoid}
A graph $X$ is $J$-fibrant if and only it is a \emph{Setoid} in the sense of Definition~\ref{def:setoids}. 
\end{prop}

\begin{proof}
A structure of setoids on a graph $X$ is exactly the same as chosen lifting against $j_t$, $j_i$ and $j_r$ for the map $X \rightarrow 1$. We have seen that a fibrant object as the lifting property against $j_r$, and conversely for a map of the form $X \rightarrow 1$ the lifting property against $j_r$ clearly implies the lifting property against $j_1$ so this concludes the proof.\end{proof}

\begin{theorem}
There is a weak model structure on the category $\G$ of graphs such that:

\begin{itemize}

\item Every object is cofibrant, cofibrations are the complemented monomorphisms of graphs, I.e. monomorphisms $f \colon X \rightarrow Y$ such that for each vertex or edge $y \in Y$, either $y \in X$ or $y \notin X$.

\item Fibrant objects are the setoids.

\item Fibrations and acyclic fibrations between fibrant objects are the $I$-fibrations and $J$-fibrations.

\item Two maps between fibrant objects $f,g \colon  X \rightrightarrows Y$ are homotopic if and only if they are equivalent in the sense of Definition~\ref{def:setoids_properties}.\ref{def:setoids_properties:relation_morphisms}.

\item The equivalences between fibrant objects corresponds to the notion of isomorphisms of setoids as in Definition~\ref{def:setoids_properties}.\ref{def:setoids_properties:isomorphisms}.

\end{itemize}

\end{theorem}

Note that (as every object is cofibrant) this weak model structure can be seen, at least non-constructively, to be a right semi-model structure. But it is not a full Quillen model structure: indeed the map $(\bullet \quad \bullet) \rightarrow \bullet $ is an $I$-fibration as there is no arrow to lift in its target, but is not an equivalence.

\begin{proof}
The first three points follow immediately from Theorem~\ref{Th_tensorWMS}, all the assumptions have been checked in the discussion above. The fourth point is exactly the description of a homotopy as a map $\Ical \otimes X \rightarrow Y$. The last point is also immediate: a map between bifibrant objects is an equivalence if and only if it is invertible in the homotopy category, and once homotopy are translated into equivalence of maps between setoids then this is exactly the condition of the theorem.\end{proof}

\begin{remark} This example also shows that for the model structure constructed out of Theorem~\ref{Th_tensorWMS} or \ref{Th::WMS_from_cylinder_path}, the acyclic cofibrations are not always the $J$-cofibrations, for examples the maps:

\[j_2\colon  y \hookrightarrow (x \rightarrow y) \qquad  j_r\colon  (\bullet) \hookrightarrow  \left( \begin{tikzcd}[ampersand replacement=\&] \bullet \arrow[loop right]\end{tikzcd} \right) \]

\noindent cannot be written as retract of composite of pushouts of coproduct of maps in $J$, but are acyclic cofibrations.

\end{remark}

\subsection{The projective model structure for chain complexes}
\label{subsec:ChainCplex}

We consider chain complexes of arbitrary degree, with a homological (i.e. degree decreasing) differential, so sequences of $R$-modules: 

\[ \dots C_{-1} \overset{\partial}{\leftarrow} C_0 \overset{\partial}{\leftarrow} C_1 \overset{\partial}{\leftarrow} \dots \overset{\partial}{\leftarrow} C_n \overset{\partial}{\leftarrow} \dots \]

\noindent subject to the condition $\partial \circ \partial = 0$, with morphisms being the morphisms of diagram. It is endowed with its usual closed monoidal structure. 



\begin{construction}
The generating cofibrations are the maps:

\[ i_k \colon \left( \parbox{9cm}{\begin{tikzcd}[ampersand replacement=\&]
 \dots \& \arrow{l} 0\arrow{d} \& \arrow{l} 0 \arrow{d} \& \arrow{l} R \arrow{d}{1} \&  \arrow{l} 0 \arrow{d} \& \arrow{l} 0 \arrow{d} \& \arrow{l} \dots \\
 \dots \& \arrow{l} 0 \& \arrow{l} 0 \& \arrow{l} R \& \arrow{l}{1} R \& \arrow{l} 0 \& \arrow{l} \dots \\ 
\end{tikzcd} \vspace{-0.4cm}}\right)  \]

Where the two nontrivial components are in degree $k-1$ and $k$. Taking a pushout by $i_k$ mean adding an element to $c_k$ with a specified differential. In particular the unit object can be obtained as a pushout of $i_0$ (so is cofibrant) and, more generally, the objects that can be obtained from the zero object as finite iterated pushouts of maps in $I = \{(i_k), k \in \mathbb{Z}\}$ are exactly the complexes which are free in each degree with a finite number of generators in total. General cofibrant objects are a little more complex to describe, but they are in particular retract of free modules (projective) in each degree. The corner-product condition for cofibrations is very easy to check: a computation shows that $i_k \corner{\otimes} i_{k'}$ is a pushout of $i_{k+k'}$.

The generating anodyne maps are given by:

\[ j_k \colon \left( \parbox{9cm}{\begin{tikzcd}[ampersand replacement=\&]
 \dots \& \arrow{l} 0\arrow{d} \& \arrow{l} 0 \arrow{d} \& \arrow{l} 0 \arrow{d} \&  \arrow{l} 0 \arrow{d} \& \arrow{l} 0 \arrow{d} \& \arrow{l} \dots \\
 \dots \& \arrow{l} 0 \& \arrow{l} 0 \& \arrow{l} R \& \arrow{l}{1} R \& \arrow{l} 0 \& \arrow{l} \dots \\ 
\end{tikzcd} \vspace{-0.4cm}}\right)  \]

\noindent with the two nontrivial components being in degrees $k$ and $k+1$, hence a pushout of $j_k$ adds both an element in degree $k+1$ and its differential in degree $k$. And $j_k \corner{\otimes} i_{k'}$ is just $i_{k'}$ tensored by the target of $j_k$ and is a pushout of $j_{k+k'}$.

\end{construction}

The cylinder object for the unit is given by:

\[\dots 0 \leftarrow R \oplus R \leftarrow R \leftarrow 0 \leftarrow \dots \]

\noindent with $R \oplus R$ in degree $0$, and $\partial \colon  R \rightarrow R \oplus R$ is $r \mapsto (-r,+r)$. The two maps from the unit are just the two co-product inclusions in dimension $0$. Moreover, the two maps $R \rightarrow R \oplus R$ corresponding to the first component and the diagonal map also identifies $R \oplus R$ as the coproduct of $R$ and $R$ in another way, and this show that our interval can be decomposed as the coproduct of the unit and the target of $j_0$, hence showing that the map from the unit to the interval is anodyne.

\begin{theorem}
There is a weak model structure on the category of chain complexes such that:

\begin{enumerate}[label=(\roman*)]

\item All objects are fibrants. A map is a fibration if on each component it admits a (possibly non-linear) section.

\item Cofibrant objects are objectwise projective\footnote{i.e. retract of a free module. We could also restrict to free modules as we are not assuming that cofibrations have to be stable under retract.} $R$-modules (but not all objectwise projective are necessarily cofibrant).

\item Two maps $f,g\colon  X \rightrightarrows Y$ with $X$ cofibrant are homotopic if there they are homotopic in the sense of homological algebra, i.e. if we have a collection of linear maps $h \colon X_n \rightarrow Y_{n+1}$ such that $\partial h + h \partial  = f -g$.

\item A map $f\colon X \rightarrow Y$ between two chain complexes is an equivalence if and only if for each $n$ the map $f\colon H_n(X) \rightarrow H_n(Y)$ is an isomorphism of setoids, where $H_n(X)$ denotes the group quotient $\{x \in X_n | \partial x =0 \} / \{ \partial x | x \in X_{n+1} \}$ constructed as a setoid.

\end{enumerate}

\end{theorem}

Of course, classically, this is in fact a Quillen model structure.

\begin{proof}
We just apply Theorem~\ref{Th_tensorWMS} with the choices explained above. Fibrations are characterized by the right lifting property against the $j_k$, as the $j_k$ all have a retraction (the $0$ map) any object is fibrant, and a map from the target of $j_k$ to $X$ is just the choice of an element in $X_{k+1}$, which implies the description of fibrations given. Condition (iii) is just a spelled out description of what is a map $\Ical \otimes X \rightarrow Y$. Condition (iv) can be deduced from Appendix~\ref{section_pisetoids} with some work, we will treat in details the corresponding statement for simplicial sets as Proposition~\ref{prop:Charac_of_simp_equiv} which is similar but harder.\end{proof}

\section{Simplicial examples}
\label{sec:simplicialEx}
\subsection{Generalities on simplicial sets and their cofibrations}
\label{subsec:GenSimp}

Let $\Delta$ be the category whose objects are the finite non-empty ordinals:

\[ [n]  = \{0,\dots,n\} \]

\noindent for $n \geqslant 0$ and whose morphisms are the order preserving maps. We denote $\widehat{\Delta}$ the category of presheaves of sets over $\Delta$, called simplicial sets. $\Delta [n]$ denotes the representable presheaf corresponding to $[n]$. For a simplicial set $X$, $X([n])$ is sometimes abbreviated to $X_n$ and for $f\colon [n] \rightarrow [m]$ we denote by $f^*$ the corresponding map $X([m]) \rightarrow X([n])$.

A cell in $X([n])$ is said to be degenerate if it is of the form $s^*y$ for $s\colon [n] \rightarrow [m]$ a surjection (also called a degeneracy). Using the factorization of maps in $\Delta$ as surjection followed by an injection any cell of the form $v^* y$ with $v$ a non-injective map is degenerate. We say that a cell is \emph{non-degenerate} if it is not degenerate, but one should be careful: being degenerate is not always a decidable property.

\begin{lemma}\label{Delta_abs_pushout}
In $\Delta$ a pushout of two degeneracies: 

\[\begin{tikzcd}[ampersand replacement=\&]
\left[ n \right] \ar[dr,phantom,"\ulcorner"{very near end}] \arrow[two heads]{r} \arrow[two heads]{d} \& \left[ i \right] \arrow[dotted,two heads]{d} \\
\left[ j \right] \arrow[dotted,two heads]{r} \& \left[ k \right] \\
\end{tikzcd}\]

 always exists, its two structural maps are again degeneracies and it is an absolute pushout (i.e. preserved by any functor).
\end{lemma}

\begin{proof}
The standard proof of this fact is constructive. A direct proof with a rather explicit computation specifically for the category $\Delta$ can be found in the first pages of \cite{joyal2008notes}. It can also be deduced from the more general theory of elegant Reedy categories introduced in \cite{bergner2013reedy}: the property in the lemma is one of the equivalent definition of elegant Reedy categories, and there are other equivalent definition considerably easier to check for the category $\Delta$.\end{proof}

The following is a constructive version of the classical Eilenberg-Zilber lemma:

\begin{lemma}\label{Lem_EilenbergZibler}
\begin{enumerate} [label=(\roman*)]

\item[]

\item If a cell $x \in X_n$ is degenerate in two ways, i.e. if $x= d_1^* y =d_2^* v$ with $d_1$ and $d_2$ two degeneracies, then there exists a cell $t$ such that $y=d_3^* t$ and $v=d_4^*t$ with $d_3$ and $d_4$ two degeneracies and $d_3 d_1 = d_4 d_2$ in $\Delta$.

\item If a cell $x$ has an expression of the form $x = d^* y$ for $d$ a degeneracy and $y$ non-degenerate, then this expression is unique.

\item Given a cell $x \in X_n$ if for every expression $x=d^* y$ with $d$ a degeneracy it is decidable whether $y$ is degenerate or not, then $x$ admit a (unique) expression of the form $d^* y$ with $d$ a degeneracy and $y$ non-degenerate.

\end{enumerate}

\end{lemma}

\begin{proof}
(i) Is a translation of the fact that, by Lemma~\ref{Delta_abs_pushout} the pushout of $d_1$ and $d_2$ in $\Delta$ exists and is preserved by $X\colon \Delta \rightarrow Set^{op}$.

For (ii), if $x$ has two such expressions  $x= d_1^* y =d_2^* v$ then the first point implies that $y$ and $v$ have to be degeneracies of a same cell $t$, but as they are non-degenerate those degeneracies have to be identities, hence $y$ and $v$ are both equal to $t$ and $d_1=d_2$. Finally (iii) follows by induction on $n$: The result is trivially true for $x \in X_0$, and for $x \in X_n$ either\footnote{Here we use the assumption for the degeneracy map $d=\text{Id}$.} $x$ is non-degenerate, in which case the result is trivially true, or $x= d^* y$ for $d$ a degeneracy, but then $y$ also satisfies the hypothesis of our claim and has strictly smaller dimension, so that $y=d'^* z$ for $z$ non-degenerate and $d'$ a degeneracy and $x= (d' d)^* z$.\end{proof}


\begin{construction}\label{constr:partiaDelta_Lambda}

We consider the following subobjects of $\Delta[n]$:

\[ (\partial \Delta[n])_k = \{ f\colon  [k] \rightarrow [n] \text{ non-surjective} \} \]

\[ (\Lambda^i[n])_k = \left\lbrace f\colon  \Delta^k \rightarrow \Delta^n \middle| \parbox{5cm}{$f$ is not surjective,  nor a surjection onto $\{0,\dots,n \} \backslash \{i\}$ } \right\rbrace \]

We denote by $\partial[n] \colon  \partial \Delta[n] \hookrightarrow \Delta[n]$ and $\lambda^k[n]\colon \Lambda^k[n] \hookrightarrow \Delta[n]$ the natural inclusion. Let also $\partial^i[n]$, or simply $\partial^i \colon  \Delta[n-1] \rightarrow \Delta[n]$ be the $i$-th face map, i.e. the map that at the level of finite ordinal is injective and skip $i$.

Alternatively, $\partial \Delta[n]$ is the union (in $\widehat{\Delta}$) of the image of all the $\partial^i[n]$ and $\Lambda^i[n]$ is the union of the image of all the $\partial[n]^j$ for $j \neq i$. Geometrically, $\partial \Delta[n]$ corresponds to the boundary of $\Delta[n]$ and $\Lambda^i[n]$ to this same boundary minus the interior of the face opposed to the $i$-th vertex.

\end{construction}

The model structures we will consider on the category of simplicial sets have for generating cofibrations:

\[ I = \{ \partial[n] \colon  \partial \Delta[n] \hookrightarrow \Delta[n] \} \]

The small object argument produces (constructively) a weak factorization system on the category of simplicial sets into  ``$I$-cofibrations'' and ``$I$-fibrations''. This is the ``good'' version of the small object argument described in \ref{EasySOA}. In classical mathematics, it follows from the Eilenberg-Zilber lemma that $I$-cofibrations are exactly the monomorphisms and hence that every object is $I$-cofibrant. Using the constructive version of the Eilenberg-Zilber lemma, we get instead:

\begin{prop}\label{prop:simplicial_Cofibration} 
The $I$-cofibrations between simplicial sets, are the map $f\colon X \rightarrow Y$ such that:

\begin{itemize}

\item $f$ is a levelwise complemented monomorphisms, i.e. for all $n$, $f\colon X_n \rightarrow Y_n$ identifies $X_n$ with a complemented (i.e. decidable) subset of $Y_n$.

\item For all cell $y \in Y_n$ which is not in the image of $X_n$, the proposition ``$y$ is a degenerate cell'' is decidable. 
\end{itemize}

In particular:

\begin{itemize}

\item $I$-cofibrant objects are the simplicial sets where it is decidable if a cell is degenerate or not.

\item $I$-cofibrations between $I$-cofibrant objects are just the levelwise complemented monomorphisms.

\end{itemize}

\end{prop}

This recovers in particular that classically every object is cofibrant and cofibrations are just the monomorphisms. The fact that not every object is cofibrant constructively is fairly new, but it was somehow expected from some negative results of T.~Coquand, M.~Bezem and E.~Parmann \cite{bezem2015kripke}, \cite{bezem2015non} about the homotopy theory of Kan complexes in constructive mathematics, and the key point of all their obstructions is exactly the undecidability of degeneratness in general.

\begin{proof}
As we are working in a presheaf category, co-limits are computed levelwise and so the ``good case'' of the small object argument presented in \ref{EasySOA} applies. In particular, cofibrations are retract of $\omega$-composition of pushouts of coproducts of generating cofibrations.

The generating cofibrations $\partial \Delta [n] \hookrightarrow \Delta[n]$ satisfy all the conditions of the proposition and these conditions are transferred to coproducts, pushouts, transfinite compositions and retracts, so that any $I$-cofibration satifies them as well.

Conversely,  assume that $f\colon  A \hookrightarrow B$ is a map satisfying the conditions in the proposition, then essentially the usual proof that every monomorphisms of simplicial set is a cofibration can be carried over constructively, thanks to those additional decidability assumptions:

First if $x$ is a cell in $B$ not in $A$, then if $x= v^* y$ for $v$ a degeneracy, then the cell $y$ cannot be in $A$ either (otherwise $x$ would be). In particular it is decidable if $y$ is degenerate or not, hence our Eilenberg-Zilber Lemma~\ref{Lem_EilenbergZibler} shows that $x=d^* y$ for a unique degeneracy map $d$ and non-degenerate cell $y$.

Let $A^n$ be the subset of cells of $B$ which are either in $A$ or degeneracies of a cell of dimension strictly less than $n$. So $A^0=A$ and $B = \bigcup_n A^n$. Each $A^n$ is a sub-simplicial set and they are all levelwise complemented. We claim that for each $n$, $A^n$ is obtained from $A^{n-1}$ as a pushout of a coproduct of copies of the map $\partial \Delta[n] \hookrightarrow \Delta[n]$. For each cell $d \in B_n$ which is neither degenerate nor in $A$, the composed map $\partial \Delta[n] \rightarrow \Delta[n] \overset{d}{\rightarrow} B$ factor in $A^{n-1}$, as its only non-degenerate cells are of dimension strictly smaller than $n$. Let $A^{n-1} \rightarrow C$ be the pushout of the coproduct of one copy of $\partial \Delta[n] \hookrightarrow \Delta[n]$ for each such cell $d$. We have a natural map from $C$ to $B$, it is rather immediate from Lemma~\ref{Lem_EilenbergZibler} and our various decidability assumption that this map is a monomorphism, and that it identifies $C$ with $A^{n}$.\end{proof}

\begin{prop}\label{prop:pushout_product_Simpcofibrations}
If $i$ and $i'$ are $I$-cofibrations then $i \corner{\times} i'$ also is.
\end{prop}

\begin{proof}
It is enough to check the result for two generating cofibrations $\partial[n]$ and $\partial[m]$ and in this case it is immediate that $\partial[n] \corner{\times} \partial[m]$ satisfies the conditions of Proposition~\ref{prop:simplicial_Cofibration}.\end{proof}

\subsection{The weak Kan-Quillen model structure}
\label{subsec:KanQuillenMS}

The goal of this subsection is to prove the following:

\begin{theorem}\label{th:KanQuillen_WMS} There is a weak model structure on the category of simplicial sets such that:

\begin{itemize}

\item The fibrant objects and fibrations between fibrant objects are characterized by the right lifting property against simplicial horn inclusion:

\[ \lambda^k[n]\colon  \Lambda^k[n] \hookrightarrow \Delta[n] \]

\item The cofibrant objects and cofibrations between them are these of Proposition~\ref{prop:simplicial_Cofibration}.

\item Acyclic fibrations between fibrant objects are characterized by the lifting property against the boundary inclusion $\partial [n]\colon \partial \Delta[n] \hookrightarrow \Delta[n]$.

\end{itemize}
\end{theorem}

The class of equivalences between fibrant objects will be described in \ref{prop:Charac_of_simp_equiv}. The theorem will be proved by applying Theorem~\ref{Th_tensorWMS} to the cartesian monoidal structure. The proof will be completed in \ref{proof:KanQuillen_WMS}. As usual, the important point is to check the corner-product condition which we will deduce from:

\begin{lemma}[Joyal]\label{lemma:Joyal_retract_KanQuillen}
The following set of morphisms generates the same weak factorization systems:

\begin{enumerate}

\item The set of horn inclusion $\lambda^k[n] : \Lambda^k[n] \hookrightarrow \Delta[n]$,

\item The set of morphisms $ i_{\epsilon} \corner{\times} \partial[n]$ with $\partial [n] \colon  \partial \Delta[n] \hookrightarrow \Delta[n]$ the boundary inclusion and $i_0,i_1 \colon  \Delta[0] \rightrightarrows \Delta[1]$ are the two endpoint inclusion.

\end{enumerate}

\end{lemma}

We call \emph{anodyne map} the left class of this weak factorization system.

\begin{proof}
This corresponds to theorem $3.2.3$ in \cite{joyal2008notes}, see also proposition $2.1.2.6$ in \cite{lurie2009higher} (which is slightly different, but the lemma can be deduced by combining this statement and its dual). The proof given in both these references are completely constructive:

One first observes that the maps $i_{\epsilon} \corner{\times} \partial [n]$ can be explicitly constructed as a pushout of horn inclusion, hence the set (ii) is included in the left class generated by (i). Then, as an application of Joyal-Tierney calculus, we observe that the left class generated by (ii) contains all the morphisms $i_{\epsilon} \corner{\times} v$ for any cofibration $v$. Another explicit construction shows that the morphism $\lambda^k[n]\colon  \Lambda^k[n] \hookrightarrow \Delta[n]$ is a retract of $ i_{\epsilon} \corner{\times} \lambda^k[n]$ (for $\epsilon =1$ if $k>0$ and $\epsilon =0$ if $k<n$). As $(\Lambda^k[n] \hookrightarrow \Delta[n])$ is a cofibration it does show that the set (i) is included in the left class generated by (ii).\end{proof}

\begin{cor}\label{cor:Simplicial_corner_prod}

In the category of simplicial sets, if $i$ is a cofibration and $j$ is an anodyne morphism, i.e. in the left class of the weak factorization system of Lemma~\ref{lemma:Joyal_retract_KanQuillen}, then $j \corner{\times} i $ is also an anodyne morphism.
\end{cor}

\begin{proof}
This follows directly from Lemma~\ref{lemma:Joyal_retract_KanQuillen} and the results of Appendix~\ref{Subsection_Joyal_tierney_calculus}. It is enough to check that if $i$ is a cofibration and $j =  i_{\epsilon} \corner{\times} \partial[n]$ is one of the generators, then $j \corner{\times} i = i_{\epsilon} \corner{\times} \partial[n] \corner{\times} i = i_{\epsilon} \corner{\times} ( \partial[n] \corner{\times} i )$. But by Proposition~\ref{prop:pushout_product_Simpcofibrations}, the map $\partial[n] \corner{\times} i$ is a simplicial cofibration, hence the map $i_{\epsilon} \corner{\times} ( \partial[n] \corner{\times} i )$ is in the class generated by the $i_{\epsilon} \corner{\times} \partial[n]$, i.e. is anodyne, which proves the result.\end{proof}

\begin{proof1}\label{proof:KanQuillen_WMS} We apply Theorem~\ref{Th_tensorWMS} to the cartesian monoidal structure on $\widehat{\Delta}$. The corner-product axiom for cofibration has been proved in \ref{prop:pushout_product_Simpcofibrations} and for anodyne morphism in Corollary~\ref{cor:Simplicial_corner_prod}. The good version of the small object argument applies to both these classes. The unit for the cartesian tensor product is $\Delta[0]$ and is cofibrant. Finally a cylinder for $\Delta[0]$ is given by:

\[ \Delta[0] \coprod \Delta[0] \overset{\partial[1]}{\hookrightarrow} \Delta[1] \rightarrow \Delta[0] \]

The two maps $\Delta[0] \rightrightarrows \Delta[1]$ are part of our generating acyclic cofibrations, so this concludes the proof. The description of acylic fibrations is immediate from the description of cofibrations and the fact that the generating cofibrations have cofibrant domains.\qed
\end{proof1}

The end of subsection~\ref{subsec:KanQuillenMS} is devoted to the proof of Proposition~\ref{prop:Charac_of_simp_equiv} below that recovers a constructive version of the usual characterization of equivalences in terms of homotopy groups.

\begin{construction}

Given a fibrant simplicial sets $X$, and $x \in X([0])$, we define, following Appendix~\ref{section_pisetoids}:

\[ \pi_n(X,x) \coloneqq  \pi_{\partial[n]} (X,x) \]

\noindent where $x$ denotes the constant morphisms $\partial \Delta[n] \rightarrow \Delta[0] \overset{x}{\rightarrow} X$.  For $n=0$, we define $\pi_0(X) = \pi_{\Delta[0]/\emptyset}(X, !)$ where $!$ denotes the unique morphism $\emptyset \rightarrow X$. ``$\pi$'' is defined in Appendix~\ref{section_pisetoids}, we remind the reader that $\pi_i(X,x)$ is defined as a setoid (see Appendix~\ref{section_setoids}) whose quotient set is the usual homotopy group. Assuming the axiom of choice these setoids can be identified with the usual homotopy groups, but constructively they need to be considered as a different objects and contain more informations than the usual homotopy group.

If follows from Remark~\ref{rk:basic_functoriality_pi_setoid} that if $f\colon X \rightarrow Y$ is an equivalence then the induced morphism:

\[ \pi_i(f) \colon  \pi_i(X,x) \rightarrow \pi_i(Y,f(x)) \]

\noindent is an equivalence of setoids.

\end{construction}

\begin{prop}\label{prop:Charac_of_simp_equiv} A morphism $f\colon X \rightarrow Y$ between fibrant simplicial sets is an equivalence if and only if for all $i \geqslant 0$, and, in the case $i>0$, for all $x \in X$, the morphism:

\[ \pi_i(X,x) \rightarrow \pi_i(X,f(x)) \]

\noindent is an equivalence of setoids.  

\end{prop}

\begin{proof}

By Theorem~\ref{Th:Charac_WE_Pi_Setoids} we need to show that under the assumption on $\pi$-setoids in the proposition, the map $\pi_{\partial[n]}(X,\lambda) \rightarrow \pi_{\partial[n]}(Y, f \circ\lambda)$ is a surjection of setoids (as in \ref{def:setoids_properties}.(iv)) for all $n$ and $\lambda \colon  \partial \Delta[n] \rightarrow X$. It follows immediately from our assumption in the case $n=0$ or when $\lambda$ is a constant map. 

The proof will be in two part:
\begin{enumerate}[label=(\roman*),noitemsep,topsep=2pt]
\item We show that given some element $v \in \pi_{\partial[n]}(X,\lambda)$ we can construct a structure of surjection on the map $\pi_{\partial[n]}(X,\lambda) \rightarrow \pi_{\partial[n]}(Y, f \circ\lambda)$.
\item We show that given an element in $\pi_{\partial[n]}(Y, f \circ\lambda)$ we can construct an element in $\pi_{\partial[n]}(X,\lambda)$.
\end{enumerate}
The combination of these two constructions provides the structure of surjection: any element $w$ in $\pi_{\partial[n]}(Y, f \circ\lambda)$ gives an element $v \in \pi_{\partial[n]}(X,\lambda)$, that in turn can be used to construct a surjection structure on  $\pi_{\partial[n]}(X,\lambda) \rightarrow \pi_{\partial[n]}(Y, f \circ\lambda)$, which can be used to produce a preimage (and a witness) for $w$.

We start with (i). $\lambda \colon  \partial \Delta[n] \rightarrow X$ and let $v\colon \Delta[n] \rightarrow X$ be an element of $\pi_{\partial[n]}(X,\lambda)$. The general idea is that $v$ shows that $\lambda$ is homotopic to a constant morphism hence the $\pi$-sets concerned are equivalent to these appearing from a constant map and for these the problem is already solved.

More precisely, by \cref{Lem_Pi_sets}.\ref{Lem_Pi_sets:Inv_POdomain} there is an equivalence of setoids \[ \pi_{\partial[n]}(X,\lambda) \simeq \pi_{\partial[n]'}(X,v) \] where $\partial[n]'$ is the morphism $\Delta[n] \hookrightarrow \Delta[n] \coprod_{\partial \Delta[n]} \Delta[n]$. Now as $\Delta[n]$ is equivalent to $\Delta[0]$ (it is possible to construct an explicit homotopy equivalence) the morphism $v$ is homotopic to a constant morphism $v'\colon  \Delta[n] \rightarrow X$, by \cref{Lem_Pi_sets}.\ref{Lem_Pi_sets:Inv_homotopy_of_Basepoint} this produces an equivalence of setoids $ \pi_{\partial[n]'}(X,v) \simeq \pi_{\partial[n]'}(X,v')$, using again \cref{Lem_Pi_sets}.\ref{Lem_Pi_sets:Inv_POdomain} this $\pi$-setoid is also equivalent to $\pi_{\partial[n]''}(X,v')$ where $\partial[n]''$ is obtained from $\partial[n]'$ by collapsing $\Delta[n]$ to $\Delta[0]$:

\[ \partial[n]'' \colon  \Delta[0] \rightarrow \left( \Delta[n] \coprod_{\partial \Delta[n]} \Delta[n] \right)  \coprod_{\Delta[n]} \Delta[0] \simeq \Delta[n] \coprod_{\partial \Delta[n]} \Delta[0] \]

Applying \cref{Lem_Pi_sets}.\ref{Lem_Pi_sets:Inv_POdomain} one more time this shows that our $\pi$-set is equivalent to $\pi_{\partial[n]}(X,v')$ where $v'$ is the constant map with the same value as the $v'$ mentioned earlier. All these equivalences are functorial in $X$, so this is shows that it is equivalent to put a surjection structure on $\pi_{\partial[n]}(f,\lambda)$ or on $\pi_{\partial[n]}(f,v')$, but the second case follows from the observation above that the problem is already solved for constant morphisms $\partial \Delta[n] \rightarrow X$.

We now prove (ii). We need to show that given any element in $\pi_{\partial[n]}(Y, f \circ \lambda)$ we can construct an element of $\pi_{\partial[n]}(X,\lambda)$. Here the informal idea is that $\lambda\colon  \partial \Delta[n] \rightarrow X$ can itself be, by a construction we will explain below, thought of as an element of $\pi_{n-1}(X,\lambda_0)$ (where $\lambda_0$ the composite of $\lambda$ with the vertices $0 \colon \Delta[0] \rightarrow \partial \Delta[n]$ ). Having an element in $\pi_{\partial[n]}(Y, f \circ \lambda)$ allows to show that the corresponding element of $\pi_{n-1}(X,\lambda_0)$ has a trivial image in $\pi_{n-1}(Y,f \circ \lambda_0)$, but as the map  $\pi_{n-1}(X,\lambda_0) \rightarrow  \pi_{n-1}(Y,f \circ \lambda_0)$ is a bijection of setoids there should also be a trivialization in $X$ which, but the same construction, corresponds to an element of $\pi_{\partial[n]}(X,\lambda)$.

To make this formal, we need to clarify how a function $\lambda\colon \partial \Delta[n] \rightarrow X$ corresponds to an element of $\pi_{n-1}(X,\lambda_0)$, and how a trivialization of such an element corresponds to an extension to $\Delta[n] \rightarrow X$. We start with some constructions: Consider the two pushout diagrams defining the objects $S_n$ and $B_n$ (the choice of the value of $i$ is irrelevant here):

\[
\begin{tikzcd}
 \Lambda^i[n] \ar[r,hook] \ar[dr,phantom,"\ulcorner",very near end] \ar[d] & \partial \Delta[n] \ar[r,hook] \ar[dr,phantom,"\ulcorner",very near end] \ar[d,"\theta"] & \Delta[n] \ar[d,"\theta'"] \\
\Delta[0] \ar[r,hook,"p"] & S_{n-1} \ar[r,hook,"i_n"] & B_{n} 
\end{tikzcd}
\]

All the vertical maps, as well as both the horizontal composites are equivalences ( $\Delta[0] \rightarrow \Lambda^i[n]$ is an acylic cofibration by $2$-out-of-$3$).

Note that $S_{n-1}$ is also isomorphic to $\Delta[n-1] \coprod_{\partial \Delta[n-1]} \Delta[0]$  through the inclusion of $\Delta[n-1]$ in $\partial \Delta[n]$ as the $i$-th face.

As the morphism $\theta\colon \partial \Delta[n] \rightarrow S_{n-1}$ is an equivalence and $X$ is fibrant, there exists a morphism $\mu\colon  S_{n-1} \rightarrow X$ such that the composite $\mu \theta$ is homotopy equivalent to $\lambda$, and as above, combining point \ref{Lem_Pi_sets:Inv_POdomain} and \ref{Lem_Pi_sets:Inv_homotopy_of_Basepoint} of Lemma~\ref{Lem_Pi_sets} there is an equivalence of $\pi$-setoids (natural in $X$) $\pi_{\partial[n]}(X,\lambda) \simeq \pi_{i_n}(X,\mu )$. This gives us in particular an element in $\pi_{i_n}(Y,f \mu )$. The morphism $\mu\colon  S_{n-1} \rightarrow X$, can be seen, by the observation that $S_{n-1}$ is isomorphic to $\Delta[n-1]\coprod_{\partial \Delta[n-1]} \Delta[0]$, as an element of $\pi_{\partial[n-1]}(X,\mu p )$.

Finally, we will prove that there is an element in $\pi_{i_n}(X,\mu )$ if and only if $\mu$ is trivial as an element of  $\pi_{\partial[n-1]}(X,\mu p )$ (trivial means equivalent to the constant map), this allows to conclude the proof by the informal argument explained above.

An arrow $ \Delta[n] \rightarrow X$ (constant on the boundary) is homotopy equivalent to a point relative to $\partial \Delta[n]$ if it can be extended into a morphism $\Delta[n] \times \Delta[1] $ which is constant on $\partial \Delta[n] \times \Delta[1] \cup \Delta[n] \times \{1\} $. I.e. a morphism $\mu\colon  S_n \rightarrow X$ is homotopically constant (relative to its base point) if it can be extended along:

\[ S_n \hookrightarrow \left( \Delta[n] \times \Delta[1] \coprod_{\partial \Delta[n] \times \Delta[1] \cup \Delta[n] \times \{1\}} \Delta[0] \right) = B'_{n+1} \]

The cofibration $\partial \Delta[n] \times \Delta[1] \cup \Delta[n] \times \{1\} \hookrightarrow \Delta[n] \times \Delta[1]$ used in the diagram above is the corner-product of $\partial[n]$ by one of the endpoint inclusion $\Delta[0] \rightarrow \Delta[1]$ hence is an acyclic cofibration, hence the inclusion $\Delta[0] \hookrightarrow B'_n$ is an acyclic cofibration. It follows that $B_{n+1}$ and $B'_{n+1}$ are equivalent in the homotopy category $Ho(\partial \Delta[n]/\widehat{\Delta})$ (their maps to $\Delta[0]$ are equivalences). This shows, by Lemma~\ref{Lem_Pi_sets}.\ref{Lem_Pi_sets:Inv_equiv_target}, that $\pi_{B_{n+1}/S_n}(X,\mu)$ and $\pi_{B'_{n+1}/S_{n}}(X,\mu)$ are equivalent. The first one being inhabited exactly means that $\mu$ is trivial as an element of $\pi_{\partial[n-1]}(X,\mu p ))$ by definition, and so this concludes the proof. \end{proof}

\subsection{The weak Joyal-Lurie model structure on marked simplicial sets}
\label{subsec:LurieMS}

In this subsection we will construct a weak version of a variant of the Joyal model structure for quasicategories due to J.~Lurie, which we will refer to as the Joyal-Lurie model structure. It is a model structure on the category of marked simplicial sets, that also models quasicategories.

\begin{definition}
A marked simplicial set $(X,\Ecal)$ is a simplicial set $X$ together with a set of ``marked'' $1$-cells: $\Ecal \subset X([1])$ containing all degenerate cells.
A morphism of marked simplicial set is a morphism of simplicial sets that send marked cells to marked cells.
The category of marked simplicial sets will be denoted by $\widehat{\Delta}^m$.

\end{definition}

We will sometime make the abuse of language to say that a simplicial set has no marked cell to means that only its degenerate cells are marked.

\bigskip

The idea of this model structure is as follows. In the Kan-Quillen model structure constructed in the previous section, the fibrant objects can be thought of as ``$\infty$-groupoids'', where the $0$-cells are objects, the $1$-cells are morphisms and the higher cells encodes arrows of higher dimension with more complicated boundary, for example a $2$-cell corresponds to a $2$-arrow of the form:

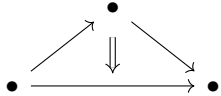
\[
\begin{tikzcd}
  & |[alias=T]| \bullet \ar[dr]  & \\
\bullet \ar[rr,""{name=S}] \ar[ur] &  & \bullet \arrow[Rightarrow, from=T, to=S,shorten <= 5pt] \\
\end{tikzcd}
\]

The Joyal-Lurie model structure (like the Joyal model structure) models a notion of $(\infty,1)$-category, i.e. where not all $1$-cell are invertible, and the marked arrows corresponds to the invertible ones.

It is very similar to Joyal model structure, which is a model structure on plain simplicial sets, where invertibility of $1$-cells is instead defined explicitly by the existence of an inverse. The Joyal-Lurie model structure is slightly better behaved (for example, it is a simplicial model structure), more expressive (minor modification allows to model cartesian fibration of quasi-category over a base), and actually simpler to construct. The main reason for this is that the generators of the Joyal-Lurie model structure, as well as its interval object are simpler to describe, and it leads simpler combinatorics when checking the corner-product conditions. Though a relatively direct proof of the corner-product conditions for the additional generator of the Joyal model structure follows from Lemma A.4 of \cite{dugger2011mapping}, and can be used to provide a constructive proof of the existence of a weak Joyal model stucture on plain simplicial sets along the same line of what we will do in the present section, with a bit more work.

\begin{construction} We introduce the following marked simplicial sets:
 \begin{itemize}[topsep=2pt]
  \item $\Delta^0[n]$ (resp. $\Delta^n[n]$) denotes respectively $\Delta[n]$ where only the $1$-cell corresponding to $\{0,1\}$ (resp. $\{n-1,n\}$) is marked (and the degenerate cells).

  \item The object $\Delta^0[1] = \Delta^1[1]$, i.e. $\Delta[1]$ with its unique non-degenerate cell being marked, will often be denoted $I$.
   
  \item $\Delta^i[n]$ for $0<i<n$, or $\Delta[n]$ denotes just $\Delta[n]$ with only the degenerate cells marked.

  \item $\Lambda^i[n]$ is defined as in \ref{constr:partiaDelta_Lambda} but endowed with the marking induced by $\Delta^i[n]$. I.e. no non-degenerate marked cell when $0<i<n$ and only one when $i=0$ or $i=n$.

  \item If $X$ is any simplicial set, $X^{\sharp}$ denotes the marked simplicial set where all $1$-cells of $X$ are marked and $X^{\flat}$ denotes the marked simplicial set where only degenerate cells are marked.

\end{itemize}

\end{construction}

\begin{construction}\label{Const:Marked_generator}

\begin{itemize}

\item[]

\item The set $I$ of generating cofibrations of the Joyal-Lurie model structure are the:

\[ \partial[n] \colon  \partial \Delta[n] \hookrightarrow \Delta[n] \]

With no markings, and the arrow

\[ \iota\colon  \Delta[1] \rightarrow I \]

\noindent which is the identity of the underlying simplicial sets.

\item The set $J$ of (pseudo) generating anodyne map of the Joyal-Lurie model structure are the:

\[ \lambda^k[n]\colon  \Lambda^k[n] \hookrightarrow \Delta^k[n] \]

\noindent for all $0 \leqslant k \leqslant n$, and the morphism:

\[ S\colon  \Delta[3]^{2/6} \rightarrow \Delta[3]^{\sharp} \]

\noindent where $\Delta[3]^{2/6}$ denotes $\Delta[3]$ where the cells corresponding to $\{0,2\}$ and $\{1,3\}$ (as well as the degenerate cells) are marked and $\Delta[3]^{\sharp}$ is the one where all $1$-cells are marked.

\end{itemize}

This last arrow $S$ essentially corresponds to the $2$-out-of-$6$ property: a morphism $\Delta[3] \rightarrow X$ is interpreted as a series of three composable arrows $f,g,h$, with their composites. Saying that it extend to $\Delta[3]^{2/6}$ means that $g \circ f$ and $h \circ g$ are marked, and saying that it extent to $\Delta[3]^{\sharp}$ means that $f,g,h$ and their composite are all marked. Hence the lifting property of an object against $S$ enforces that marked cells satisfies the $2$-out-of-$6$ property.

\end{construction}

\begin{remark}\label{remark:2_out_of_3_marked_generator}
The usual fact that the $2$-out-of-$6$ property implies the $2$-out-of-$3$ property shows in this case that the three arrows that encode the $2$-out-of-$3$ property similarly to how $S$ encode the $2$-out-of-$6$ property, whose domain are respectively:

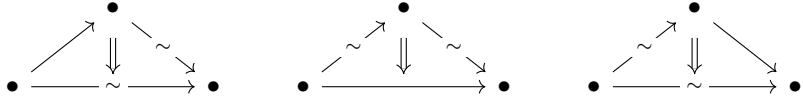
\[ \begin{tikzcd}
  & |[alias=T]| \bullet \ar[dr,"\sim"{description}]  & \\
\bullet \ar[rr,""{name=S},"\sim"{description}] \ar[ur] &  & \bullet \arrow[Rightarrow, from=T, to=S,shorten <= 5pt] \\
\end{tikzcd} \qquad
\begin{tikzcd}[ampersand replacement=\&]
   \& |[alias=T]| \bullet \ar[dr,"\sim"{description}]  \& \\
 \bullet \ar[rr,""{name=S}] \ar[ur,"\sim"{description}] \&  \& \bullet \arrow[Rightarrow, from=T, to=S,shorten <= 5pt] \\
 \end{tikzcd} \qquad
\begin{tikzcd}
  & |[alias=T]| \bullet \ar[dr]  & \\
\bullet \ar[rr,""{name=S},"\sim"{description}] \ar[ur,"\sim"{description}] &  & \bullet \arrow[Rightarrow, from=T, to=S,shorten <= 5pt] \\
\end{tikzcd}
 \]

\noindent and with all the same target::

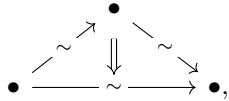
\[
\begin{tikzcd}
  & |[alias=T]| \bullet \ar[dr,"\sim"{description}]  & \\
\bullet \ar[rr,""{name=S},"\sim"{description}] \ar[ur,"\sim"{description}] &  & \bullet, \arrow[Rightarrow, from=T, to=S,shorten <= 5pt] \\
\end{tikzcd}
 \]

\noindent are all pushout of $S$, along the three different degeneracy morphisms $\Delta[3] \rightarrow \Delta[2]$.

The middle one, corresponding to the fact that marked arrows are stable under composition, will be denoted $C$.

\end{remark}

\begin{remark}
Product in the category of marked simplicial sets are simply given by $(X,\Ecal) \times (X',\Ecal') = (X \times X', \Ecal \times \Ecal')$ in particular they commutes to all colimits in each variables. In fact the category of marked simplicial sets is cartesian closed.
\end{remark}

\begin{lemma}\label{lem:MarkedCof}
An arrow between marked simplicial sets is an $I$-cofibration if and only if the underlying map of simplicial sets is a cofibration in the sense of Proposition~\ref{prop:simplicial_Cofibration}. In particular if $i$ and $i'$ are $I$-cofibration then $i \corner{\times} i'$ is an $I$-cofibration.
\end{lemma}

\begin{proof}
It immediately follows from the proof of Proposition~\ref{prop:simplicial_Cofibration} that the unmarked inclusion $\partial \Delta[n] \hookrightarrow \Delta[n]$ generates all maps $A \hookrightarrow B$ whose underlying simplicial map is a cofibration and such that only cells in $A$ and degenerate cells are marked. Taking further pushout by $\Delta[1] \hookrightarrow I$ has the effect of making any set of cells in $B$ marked. So any morphisms whose underlying simplicial map is a cofibration is an $I$-cofibration. The converse is immediate as the forgetful functor from marked simplicial sets to simplicial sets commutes to colimits.
The second part of the claim hence follows from the similar statement for cofibration of simplicial sets, proved in \ref{prop:pushout_product_Simpcofibrations}.\end{proof}

The following lemmas are due to Joyal in the unmarked case and Lurie in the case of marked simplicial sets, we give references to \cite{lurie2009higher}, but lots of the proof originally comes from \cite{joyal2008notesQcat}. They all have completely explicit combinatorial proofs and together will allows us to check the corner-product condition between anodyne morphisms and cofibrations.

\begin{lemma}\label{lem:marked_retract}

\begin{enumerate}

\item[]

\item $\lambda^n[n] \colon  \Lambda^n[n] \hookrightarrow \Delta^n[n]$ is a retract of $\lambda^1[1] \corner{\times} \lambda^n[n]$.

\item $\lambda^0[n] \colon  \Lambda^0[n] \hookrightarrow \Delta^0[n]$ is a retract of $\lambda^0[1] \corner{\times} \lambda^0[n]$.

\item $\partial[n] \corner{\times} \lambda^1[1]$ is an iterated pushout of maps of the form $\lambda^i[n+1]$ for $i>0$.

\item $\partial[n] \corner{\times} \lambda^0[1]$ is an iterated pushout of maps of the form $\lambda^i[n+1]$ for $i<n+1$.

\item $\lambda^i[n]\colon  \Lambda^i[n] \hookrightarrow \Delta[n]$ for $0<i <n$ is a retract of $\lambda^1[2] \corner{\times} \lambda^i[n]$.

\item $\partial[n] \corner{\times} \lambda^1[2]$ is a pushout of the coproduct of the $\lambda^i[n+2]$ for $0<i<n+2$.

\end{enumerate}

\end{lemma}

\begin{proof}
\begin{enumerate}
\item Proved in the proof of proposition 3.1.1.5 of \cite{lurie2009higher}.

\item This is dual to ($1$).

\item Proved in the proof of proposition 3.1.1.5 of \cite{lurie2009higher}.

\item This is dual to ($3$).

\item This only involves plain (unmarked) simplicial sets, and is proved in the proof of proposition $2.3.2.1$ of \cite{lurie2009higher}.

\item This only involves plain (unmarked) simplicial sets, and is proved in the proof of proposition $2.3.2.1$ of \cite{lurie2009higher}.

\end{enumerate}
\end{proof}

\begin{remark}\label{remark:POprod_marked_bijection}
If $f$ and $h$ are morphisms of marked simplicial sets, and $h$ is an isomorphism of the underlying simplicial sets then $f \corner{\times} h$ is also an isomorphism of the underlying simplicial sets. Indeed the forgetful functor to simplicial sets preserves corner-product and a corner-product by an isomorphisms is an isomorphism.

More precisely, if we have $f:X \to Y$ and $h:Z' \to Z$, then the the target of $f \corner{\times} h$ is $Z \times Y$, where marked cells are all pair of cells that are marked in $Y$ and in $Z$, while its domain has the same underlying simplicial sets, but the marked cells are only the pairs of cells marked in $Y$ and $Z'$, together with all pairs of a cell marked in $Z$ and the image of a cell marked in $X$.

In particular, if we additionally assume that any cell marked in $Y$ is the image of a cell marked in $X$, for example if $f: X \to Y$ is an inclusion of unmarked simplicial sets where all the $0$-cells of $Y$ are in $X$ (this is necessary so that the degenerate $1$-cells in $Y$ are indeed image of marked cell in $X$), then $f \corner{\times} g$ is an isomorphism.

\end{remark}

\begin{lemma}\label{lem:Pro_prod_iota_iota}

Let $\iota \colon \Delta[1] \rightarrow I$ as in \ref{Const:Marked_generator}. The morphism $\iota \corner{\times} \iota$ is a pushout of the morphism $C$ of Remark~\ref{remark:2_out_of_3_marked_generator}, in particular it is also a pushout of the morphism $S$ of Construction~\ref{Const:Marked_generator}.

\end{lemma}

\begin{proof}

The domain of $\iota \corner{\times} \iota$ is $\Delta[1] \times \Delta[1]$ where all $1$-cells have been marked except the one corresponding to the diagonal map $\Delta[1] \rightarrow \Delta[1] \times \Delta[1]$. Its target is $I \times I$ i.e. $\Delta[1] \times \Delta[1]$ with all arrows marked. So we can indeed realize it as a pushout of $C$ in two different ways, using either of the two non-degenerate cells $\Delta[2] \rightarrow \Delta[1]\times \Delta[1]$.\end{proof}

\begin{prop}\label{Prop:Po_prod_MarkedTrivCof}

The corner-product of a $J$-cofibration with an $I$-cofibration (as defined in \ref{Const:Marked_generator}) is a $J$-cofibration.

\end{prop}

\begin{proof}

We first observe that all the morphisms of the form:

\[ \lambda^i[n] \corner{\times} \iota \]

\noindent are $J$-cofibrations. Indeed, for $n>1$ they are isomorphisms by Remark~\ref{remark:POprod_marked_bijection}, for $n=1$, the map $\lambda^0[1] \corner{\times} \iota$ is:

\[
\lambda^0[n] \corner{\times} \iota \colon  \left(\parbox{1.8cm}{\begin{tikzcd}[ampersand replacement=\&]
  \bullet  \ar[dr,""{name=M,description},phantom] \ar[dr] \ar[r,"\sim"{description}] \ar[d,"\sim"{description}] \& \bullet \ar[d,"\sim"{description}] \ar[Rightarrow, to=M] \\
\bullet \ar[r] \ar[Rightarrow, to=M] \& \bullet \\
\end{tikzcd}\vspace{-20pt}}\right) \rightarrow 
\left(\parbox{1.8cm}{\begin{tikzcd}[ampersand replacement =\&]
  \bullet  \ar[dr,""{name=M,description},phantom] \ar[dr,"\sim"{description}] \ar[r,"\sim"{description}] \ar[d,"\sim"{description}] \& \bullet \ar[d,"\sim"{description}] \ar[Rightarrow, to=M] \\
\bullet \ar[r,"\sim"{description}] \ar[Rightarrow, to=M] \& \bullet \\
\end{tikzcd}\vspace{-20pt}}\right)
\]

\noindent which is an iterated pushout of two of the morphisms in Remark~\ref{remark:2_out_of_3_marked_generator}, hence is a $J$-cofibration. This works similarly for $\lambda^1[1] \corner{\times} \iota$. Moreover by Lemma~\ref{lem:Pro_prod_iota_iota},  $S \corner{\times} \iota$ is an iterated pushout of $C$ because $S$ is an iterated pushout of $\iota$, hence it is a $J$-cofibration as well. So far we have proved that if $f$ is a $J$-cofibration then $\iota \corner{\times} f$ is also a  $J$-cofibration.

It follows from points $(3),(4)$ and $(6)$ of Lemma~\ref{lem:marked_retract} that if $j$ is $\lambda^1[2], \lambda^0[1]$ or $\lambda^1[1]$ then $\partial[n] \corner{\times} j$ is a $J$-cofibration. As this has been proved for $\iota \corner{\times} j$ above, it shows that for any $I$-cofibration $i$, and $j$ any of the three maps above $j \corner{\times} i$ is a $J$-cofibration.

It follows from point $(1),(2)$ and $(5)$ of Lemma~\ref{lem:marked_retract} that all generators $\lambda^i[n]$ are in the class generated by the $j \corner{\times} i$ for $j = \lambda^1[2], \lambda^0[1]$ or $\lambda^1[1]$ and $i$ an $I$-cofibration, but this class is clearly stable by corner-product by an $I$-cofibration (by associativity of the corner-product) and is included in $J$ as showed above, so this proves that $\lambda^i[n] \corner{\times} i$ is a $J$-cofibration when $i$ is an $I$-cofibration.

Finally, $S \corner{\times} \partial[n]$ is an isomorphism for all $n>1$ and is just $S$ when $n=0$, hence it is always a $J$-cofibration, which concludes the proof. \end{proof}

\begin{theorem}[Joyal-Lurie model structure]\label{Th:LurieModelStructure}

There is a weak model structure on the category of marked simplicial sets such that:

\begin{itemize}

\item The cofibrations are the $I$-cofibrations of Construction~\ref{Const:Marked_generator}, i.e. the morphisms that are cofibrations of the underlying simplicial sets.

\item The fibrations are the $J$-fibrations of Construction~\ref{Const:Marked_generator}.

\end{itemize}

\end{theorem}

It is easy to see, assuming a bit of theory of quasi-category that the fibrant objects are exactly the quasi-categories, where the marked cells are the equivalence, hence this model structure is exactly the same as Lurie's model structure on marked simplicial sets from \cite{lurie2009higher}.

\begin{proof}
We proceed exactly as for the proof of Theorem~\ref{th:KanQuillen_WMS}: We checked the corner-product condition in Lemma~\ref{lem:MarkedCof} and proposition  \ref{Prop:Po_prod_MarkedTrivCof}, and $I$ produces an interval for the cartesian unit $\Delta[0]$ which satisfies all the conditions of Theorem~\ref{Th_tensorWMS}.\end{proof}

\begin{remark}\label{remark:Unsaturated_Version_of_Lurie_MS}
 If in the definition of $J$ we remove $S$, and takes instead the three morphisms of Remark~\ref{remark:2_out_of_3_marked_generator} imposing the weaker $2$-out-of-$3$ condition instead of the $2$-out-of-$6$ condition, the proof of Proposition~\ref{Prop:Po_prod_MarkedTrivCof} and Theorem~\ref{Th:LurieModelStructure} remains completely unchanged, and we also obtain a weak model structure.

This weak\footnote{It can be shown to be a full model structure classically, for example using theorem 3.16 of \cite{olschok2011left}.} model structure is different from the Joyal-Lurie model structure: its fibrant objects are still quasi-categories but their marked cells are only forced to satisfies the $2$-out-of-$3$ condition. In particular it is no longer necessary that all equivalences are marked.

For example taking any quasi-category $\Xcal$ if we mark all the $1$-cells which are equal to an identity $1$-cell in the homotopy category of $\Xcal$, then the resulting object is indeed fibrant in this modified version of the model structure: it has the right lifting property against the inner horn inclusion because it is a quasi-category, the lifting property against the marked outer horn inclusion follows from the fact that the marked cells are in particular invertible and a classical lemma in quasi-category theory (see\footnote{We have not checked the constructivity of this claim explicitly.} proposition 1.2.4.3 of \cite{lurie2009higher}) , and the lifting property against the three maps of Remark~\ref{remark:2_out_of_3_marked_generator} follows from the fact that identities in any category satisfies the $2$-out-of-$3$ property. This shows that this model structure has strictly more fibrant objects than the Joyal-Lurie model structure.

More generally the fibrant objects of this second model structure are quasi-categories, together with a subgroupoid of their homotopy category of ``marked cell'' which contains all identity and satisfies $2$-out-of-$3$. And they are fibrant for the Joyal-Lurie model structure if and only if this subgroupoid is the maximal subgroupoid (contain all isomorphisms).

\end{remark}

\subsection{The weak Verity model structure for complicial sets}
\label{subsec:complicialsets}

In this section, we will discuss a weak model structure on ``stratified simplicial'' sets whose fibrant objects are the so-called (weak) complicial sets (see \cite{verity2008weak}). The classical version has been introduced by D.~Verity in \cite{verity2008weak} and is a Quillen model structure. This will be done using Theorem~\ref{Th_tensorWMS} with the monoidal structure given by the cartesian product of stratified simplicial sets.\footnote{We thank Harry Gindy and Viktoriya Ozornova who independently pointed out a mistake in my attempt at a simpler proof of the corner-product conditions in an earlier version.}

Intuitively, this model structure is supposed to model weak $(\infty,\infty)$-categories\footnote{More precisely, it is intended to model ``inductive'' $\infty$-categories, i.e. the projective limit of the tower of ``$(\infty,k)$-cat'' and functors that sends an $(\infty,k+1)$-category $\Ccal$ to the $(\infty,k)$-category obtained by drooping the non-invertible $k$-arrows of $\Ccal$.}. It also provides models for $(\infty,n)$-categories for any $n$, the case $n=0$ and $n=1$ corresponding exactly to the model structure constructed in Subsection~\ref{subsec:KanQuillenMS} and \ref{subsec:LurieMS}. In the sense of this definition, $(\infty,\infty)$-categories are supposed to be the fibrant objects, which are called ``complicial sets'' or ``weak complicial sets''.

\begin{definition}
A stratified simplicial set is a simplicial set $X$, together with a set of cells $tX \subset X$ called ``thin'' such that no $0$-cell is thin and all degenerate cell are thin. The category of stratified simplicial sets is denoted $\widehat{\Delta}^s$.

An $n$-stratified simplicial set is a stratified set where all $m$-cells for $m>n$ are thin.

Morphisms of stratified simplicial sets are the morphisms sending thin cells to thin cells.

\end{definition}

The reason we say ``stratified'' and ``thin'' instead of ``marked'' as in Subsection~\ref{subsec:LurieMS} is both to avoid confusion and because these are the standard terminology used by most text on complicial sets.

\begin{remark}\label{discus:intuitivecomplicialsets} Note that a $1$-stratified simplicial set is essentially the same as a marked simplicial set as in Subsection~\ref{subsec:LurieMS}. Similarly to what happened with the Joyal-Lurie model structure, the idea is that ``thin'' cells are invertible higher cells.
\end{remark}

\begin{construction}
Following \cite{verity2008weak}:

\begin{itemize}

\item $\Delta[n]$ denotes $\Delta[n]$ with no non-degenerate thin cells. $\Delta[n]_t$ is $\Delta[n]$ where only the unique non-degenerate $n$-cell is thin.

\item $\Delta^k[n]$ is $\Delta[k]$ where all the non-degenerate cells $\alpha\colon [r] \hookrightarrow [n]$ which contain $\{k-1,k,k+1\}$ are thin.

\item $\Lambda^k[n]$ is endowed with the stratification induced from $\Delta^k[n]$.

\item $\Delta^k[n]'$ is $\Delta^k[n]$ where in addition the faces $\partial^{k-1}[n], \partial^{k+1}[n] \colon  [n-1] \rightarrow [n]$ are thin.

\item $\Delta^k[n]''$ is $\Delta^k[n]'$ where in addition $\partial^k[n]$ is thin, i.e. it is $\Delta^k[n]$ where all the $(n-1)$-cells are made thin.

\end{itemize}

\end{construction}

\begin{construction}\label{const:generating_stratified_cofib}

The set of generating cofibrations $I$  is made of:
\begin{itemize}

\item The $\partial \Delta[n] \hookrightarrow \Delta[n]$, the unmarked boundary inclusion.

\item The $\Delta[n] \rightarrow \Delta[n]_t$.

\end{itemize}

Similarly to Lemma~\ref{lem:MarkedCof}, the $I$-cofibrations are just the morphisms whose underlying simplicial morphism is an $I$-cofibration. In particular, if $i$ and $i'$ are $I$-cofibrations then $i \corner{\times} i'$ is again an $I$-cofibration from the similar result for simplicial sets proved in \ref{prop:pushout_product_Simpcofibrations}.
\end{construction}

\begin{construction}
\label{Constr:generating_stratified_trivcof}Still following \cite{verity2008weak}, we take as generating anodyne maps of stratified simplicial sets the set $J^s = J^{Horn} \coprod J^{Thin}$ made of:

\begin{itemize}

\item The ``complicial horn inclusions'':

\[J^{Horn}=\{ \Lambda^k[n] \hookrightarrow \Delta^k[n] \} \]

\noindent for $n \geqslant 1$ and $k \in [n]$.

\item The ``complicial thinness extensions'':

\[ J^{Thin} = \{ \Delta^k[n]' \hookrightarrow \Delta^k[n]'' \} \]

For $n \geqslant 2$ and $k \in [n]$.

\end{itemize}

\end{construction}

All the maps in $J^s$ are clearly $I^s$-cofibrations.

\begin{definition}
A complicial set is a stratified simplicial set which have the right lifting property against all the generating anodyne map.
\end{definition}

Those are sometimes called weak complicial sets, the original definition of complicial sets being the stratified simplicial sets that have the unique right lifting property against maps in $J^s$.

In terms of the intuitive idea that thin cells corresponds to invertible higher cells, the lifting property against complicial thinness extensions is implementing properties like $2$-out-of-$3$ and the fact that cells that are actually invertible up to thin cell are themselves thin, while the lifting properties against the complicial horn inclusions are implementing weak composition operations (up to thin cells) and the fact that thin cells are actually invertible (up to higher thin cells).

The fact that the corner-product condition for complicial sets are satisfies have been proved by D.~Verity in \cite{verity2008weak} using an explicit combinatorial argument that seems constructive to us. Completely reproducing to show its constructivity in details seem to be outside of the scope of the present paper, hence we will admit the following:

\begin{prop}\label{prop:pushoutProduct_Comp_anodyneCof}
The corner-product $j \corner{\times} i$ of a $J$-cofibration (in the sense of Construction~\ref{Constr:generating_stratified_trivcof}) with an $I$-cofibration (in the sense of Construction~\ref{const:generating_stratified_cofib}) is again a $J$-cofibration.
\end{prop}

\begin{proof}
As usual it is enough to check it on the generators, and this is done in \cite{verity2008weak} as lemma 72. Note that \cite{verity2008weak} use the symbol $\circledast$ to denote the cartesian product of stratified simplicial sets. This is to emphasize that it corresponds to the pseudo Gray tensor product.\end{proof}

\begin{theorem}\label{Th:WMS_stratifiedSimp}
There exists a weak model structure on the category of stratified simplicial sets such that:

\begin{itemize}

\item The cofibrant objects are those in which degeneracies of cell is decidable.

\item The cofibrations between cofibrant objects are the levelwise complemented monomorphisms.

\item Fibrant objects and fibrations are defined by the lifting property against the class $J^s$ of Construction~\ref{Constr:generating_stratified_trivcof}.

\item The acyclic fibrations between fibrant objects are the maps that detect thinness and whose underlying simplicial maps have the lifting property against all the $\partial \Delta[n] \hookrightarrow \Delta[n]$.

\end{itemize}

\end{theorem}

This model structure for complicial sets is the same as the one in \cite{verity2008weak}, or rather a ``weak'' version of it, but we now know it exists constructively, even predicatively.

\begin{proof}
We apply Theorem~\ref{Th_tensorWMS} to the category of stratified simplicial sets seen as a cartesian closed category (for the cartesian product as our bi-functor). The small object argument is immediately applicable. The corner-product conditions have been proved in \ref{prop:pushout_product_Simpcofibrations} for cofibrations and admitted in \ref{prop:pushoutProduct_Comp_anodyneCof} for acyclic cofibrations. The interval for the unit is given by $\Delta[0] \coprod \Delta[0] \rightarrow \Delta[1]_t \rightarrow \Delta[0]$.\end{proof}

It also immediately follows that:

\begin{theorem}\label{Th:WMS_n-stratifiedSimp}

There is a weak model structure on the category of $n$-stratified simplicial sets, whose cofibrations and fibrations are the morphisms that are cofibrations and fibrations as morphisms of stratified simplicial sets.

\end{theorem}

\begin{proof}
Consider $T_n$ the functor from the category of stratified simplicial sets to the category of $n$-stratified simplicial sets that makes thin all the cells of dimension greater than $n$. We take as generating cofibrations and trivial cofibrations of this model structure the image by $T_n$ of the generators of the Verity model structure (from Theorem~\ref{Th:WMS_stratifiedSimp}) on stratified simplicial sets. We then apply Theorem~\ref{Th_tensorWMS} exactly as in the proof of~\ref{Th:WMS_stratifiedSimp}, the corner-product conditions follow from the fact that the functor $T_n$ preserves product and pushouts. \end{proof}

\begin{remark}
It is immediate to see that the case $n=0$ of Theorem~\ref{Th:WMS_n-stratifiedSimp} corresponds exactly to the weak model structure of Theorem~\ref{th:KanQuillen_WMS}: the image of the generators given in \ref{Constr:generating_stratified_trivcof} and \ref{const:generating_stratified_cofib} by the functors that makes everything thin are exactly the generators of the weak model structure of Theorem~\ref{th:KanQuillen_WMS}.

The case $n=1$ corresponds almost to the Joyal-Lurie model structure (Theorem~\ref{Th:LurieModelStructure}), but not quite.

Indeed, the functor $T_1$ that makes every cell above dimension $1$ thin, sends the $\Lambda^k[n] \hookrightarrow \Delta^k[n]$ of Construction~\ref{Constr:generating_stratified_trivcof} to the maps of construction~\ref{Const:Marked_generator}, the $\Delta^k[n]' \rightarrow \Delta^k[n]''$ of Construction~\ref{Constr:generating_stratified_trivcof} are sent to isomorphisms for $n>2$ and to the three maps of Remark~\ref{remark:2_out_of_3_marked_generator} for $n=2$, but the morphisms $S$ of \ref{Const:Marked_generator} is not obtained. So we do not obtain the Joyal-Lurie model structure but its ``unsaturated'' modification mentioned in Remark~\ref{remark:Unsaturated_Version_of_Lurie_MS}.
\end{remark}

\begin{remark}\label{rk:complicial_saturation}

It is also possible to modify the weak model structures of Theorem~\ref{Th:WMS_stratifiedSimp} and \ref{Th:WMS_n-stratifiedSimp} in order to add an analogue of this map $S$ to the generators. The fibrant objects are then called saturated complicial sets (or $1$-saturated complicial set depending on which generators we add), we refer the reader to section 3 of \cite{riehl2018complicial} for the details of this.

\end{remark}

\subsection{Semi-simplicial versions}
\label{Sec_ex_semi-simplicial sets}

In this section we construct versions of all the weak model structures constructed above on the category of semi-simplicial sets (see \ref{def:semi_simpsets}) instead of simplicial sets. The main advantages of these semi-simplicial versions is that even constructively, every object will be cofibrant.

It seems that the existence of these model structures have been overlooked in classical mathematics. This is probably due to the fact that it is well known that this type of model structure on semi-simplicial sets ``cannot exists'', in the sense that they are not Quillen model structures. They are at best right semi-model structures (see Remark~\ref{rk:ssWMS_are_not_Quillen}).

\begin{definition}\label{def:semi_simpsets}
Let $\Delta_+ \subset \Delta$ be the subcategory of finite non-empty ordinals and injective order preserving morphisms. A presheaf on $\Delta_+$ is called a semi-simplicial set. We denote by $\Delta_+[n]$ the representable semi-simplicial sets attached to the ordinal $[n]$. 
\end{definition}

Informally, a semi-simplicial set is a ``simplicial set without degeneracies''.

\begin{construction}
 The forgetful functor $\widehat{\Delta}\rightarrow \widehat{\Delta_+}$ has a left adjoint $X \mapsto \scomp{X}$, called \emph{simplicial completion} which ``freely adds degeneracies''.

$\scomp{X}$ admits an explicit description, which is very typical of the theory of Reedy categories:
\[ \scomp{X}_n = \{ (s,x) | s \colon  [n] \twoheadrightarrow [m] \text{ order preserving surjection} , x \in X_m \} \]

\noindent the functoriality on an order preserving map $f\colon  [n'] \rightarrow [n]$ is given by forming the composite $s \circ f$ and factoring it into a surjection $g$ followed by a monomorphism $i$:

\[\begin{tikzcd}
\left[ n' \right] \arrow[two heads]{d}{g} \arrow{r}{f} & \left[ n\right] \arrow[two heads]{d}{s} \\
\left[ m'\right] \arrow[hook]{r}{i} & \left[ m\right] \\ 
\end{tikzcd}\]

And we define $f^*(s,x) \coloneqq (g,i^* x)$.  In particular, in $\scomp{X}$ a pair $(s,x)$ is equal to $s^* (Id_{[k]},x)$ and as $x \mapsto (Id_{[k]},x)$ is the unit of adjunction $X \rightarrow \scomp{X}$, we will simply denote $(Id_{[k]},x)$ by $x$ and identifies $X$ with its image in $\scomp{X}$. Hence, $\scomp{X}$ contains $X$ as a sub-semi-simplicial set and a general cell of $\scomp{X}_n$ is of the form $s^* x$ for a unique $x \in X_m$ and a unique degeneracy map $s \colon  [n] \twoheadrightarrow [m]$.

In particular, if $X$ is a semi-simplicial set, it identifies naturally with the set of cells of $\scomp{X}$ which are non-degenerate. Moreover, in $\scomp{X}$ degeneracies are decidable (i.e. $\scomp{X}$ is cofibrant) and a face of a non-degenerate cell is always non-degenerate. Conversely, given a simplicial $Y$ set with these properties, then its subset of non-degenerate cells is a semi-simplicial set $X$, and because of the Eilenberg-Zilber lemma (\ref{Lem_EilenbergZibler}) there is a canonical isomorphism $Y \simeq \scomp{X}$. Putting all this together, we deduce:

\end{construction}

\begin{prop}\label{prop:semi_simp_are_simp}

The category of semi-simplicial sets is equivalent to the non-full subcategory of simplicial sets such that:

\begin{itemize}
\item Objects are the simplicial sets in which it is decidable if a cell is degenerate or not, and any face of a non-degenerate cell is non-degenerate.
\item Morphisms are the morphisms which send non-degenerate cells to non-degenerate cells.
\end{itemize}

\end{prop}

\begin{construction}
\label{cons:semi-simplicial_Tensproduct}
The cartesian product of simplicial sets induces, through the identification of Proposition~\ref{prop:semi_simp_are_simp}, a monoidal structure on the category of semi-simplicial sets. We denote this tensor product by $X \otimes Y$ and it is characterized by the (functorial) identification:

\[ \overline{ X \otimes Y } = \overline{X} \times \overline{Y} \]

The tensor product of semi-simplicial sets always contains their cartesian product as semi-simplicial sets, but it is in general larger: indeed a couple $(x,y) \in \scomp{X} \times \scomp{Y}$ of two degenerate cells in $x \in \scomp{X}$ and $y \in \scomp{Y}$ can be non-degenerate in the product, and hence being a cell of the semi-simplicial tensor product without being a pair of cells of the semi-simplicial sets. It is not very hard to see that this monoidal structure on semi-simplicial sets is closed as it commutes to colimits in each variables.

\end{construction}

\begin{remark}\label{semi_simp_PartialDelta_Lambda}
The simplicial sets $\Delta[n]$, $\partial \Delta[n]$ and $\Lambda^k[n]$ (see \ref{constr:partiaDelta_Lambda}) all satisfies the conditions of Proposition~\ref{prop:semi_simp_are_simp}, hence they are the simplicial completion of semi-simplicial sets, which we denote:
\[ \Delta_+[n] \qquad \partial \Delta_+[n] \qquad \Lambda^k_+[n] \]

Note that the $\Delta_+[n]$ are exactly the representable semi-simplicial sets.

\end{remark}

Similarly we will also consider:

\begin{itemize}

\item The category $\widehat{\Delta_+}^m$ of marked semi-simplicial sets, which are semi-simplicial sets with a collection of $1$-cells called ``marked'' cells.

\item The category $\widehat{\Delta_+}^s$ of stratified semi-simplicial sets, which are semi-simplicial sets with a collection of cells (not containing any $0$-cells) called ``thin'' cells.

\end{itemize}

Note that if $X$ is a marked or stratified semi-simplicial set, then $\scomp{X}$ have a unique marking/stratification (as a simplicial sets) compatible to the one on $X$: non-degenerate cells are marked/thin if and only if they are marked/thin as cells of $X$ and all degenerate are marked/thin.

In particular, the identification of Proposition~\ref{prop:semi_simp_are_simp} extend to the marked and stratified case and identifies respectively the categories of marked or stratified semi-simplicial sets with the (non-full) subcategories of marked or stratified simplicial sets satisfying the conditions of Proposition~\ref{prop:semi_simp_are_simp}. In particular, exactly as in Construction~\ref{cons:semi-simplicial_Tensproduct} there is a unique closed monoidal structure on marked and stratified semi-simplicial sets, that makes simplicial completion into a monoidal functor.

Also, as in \ref{semi_simp_PartialDelta_Lambda} we denote by:

\[ \Delta^k_+[n] \qquad \Delta_+[n]_t \qquad \Delta^k_+[n]' \qquad \Delta^k_+[n]'' \]

\noindent the semi-simplicial versions of all the simplicial objects we introduce in the previous subsections. Their simplicial completions identify with the corresponding simplicial objects.

\begin{theorem}\label{th:semi-simp_WMS}

For each of the weak model structure constructed in \cref{th:KanQuillen_WMS,Th:LurieModelStructure,remark:Unsaturated_Version_of_Lurie_MS,Th:WMS_stratifiedSimp,rk:complicial_saturation} on the category of plain, marked or stratified simplicial sets, there is a weak model structure on the category of plain, marked or stratified semi-simplicial sets such that:

\begin{enumerate}[label=(\roman*)]

\item\label{th:semi-simp_WMS:generators} Its generating cofibrations and anodyne are same as the simplicial version, seen through the equivalence of Proposition~\ref{prop:semi_simp_are_simp}.

\item\label{th:semi-simp_WMS:cofibrations} Cofibrations are the levelwise complemented monomorphisms, i.e. the monomorphisms $f\colon X \rightarrow Y$ such that for all $x \in Y([n])$ we have $x \in X([n]) \vee x \notin X([n])$. In particular, every object is cofibrant.

\item\label{th:semi-simp_WMS:monoidal} The model structure is monoidal for the semi-simplicial tensor product of \ref{cons:semi-simplicial_Tensproduct}.

\item\label{th:semi-simp_WMS:equivalence} The forgetful functor from simplicial sets to semi-simplicial sets is both a left and a right Quillen equivalence. In particular, the simplicial completion functor $X \mapsto \scomp{X}$ is a left Quillen equivalence.

\end{enumerate}

\end{theorem}

\begin{proof}

First we observe that taking the maps $\partial \Delta_+[n] \hookrightarrow \Delta_+[n]$, and $\Delta_+[n] \hookrightarrow \Delta_+[n]_t$ (in the marked/stratified case) as specified in point $\ref{th:semi-simp_WMS:generators}$ as generating cofibrations gives the class of cofibrations described in point $\ref{th:semi-simp_WMS:cofibrations}$. In the case of plain semi-simplicial sets this is proved exactly as the proof of \ref{prop:simplicial_Cofibration} (ignoring the treatment of degeneracies), the extension to the marked/stratified case work exactly as in Lemma~\ref{lem:MarkedCof}.

It immediately follows that cofibrations satisfies the corner-product condition with respect to the tensor product of \ref{cons:semi-simplicial_Tensproduct} as corner-products of generating cofibrations clearly satisfies the condition of point $\ref{th:semi-simp_WMS:cofibrations}$.

The key results are Proposition~\ref{prop:SSSunitEquiv} and its Corollary~\ref{cor:freeSimpSeeTrivCof} below which allows us to deduce the corner-product conditions for the monoidal structure on (marked/stratified) semi-simplicial sets from the similar condition for (marked/stratified) simplicial sets:

If $f$ is a cofibration and $g$ is an acyclic cofibration of (marked/stratified) semi-simplicial sets then  $\scomp{f}$ and $\scomp{g}$ are respectively a cofibration and an acyclic cofibration of (marked/stratified) simplicial sets simply because this is true for the generators. As the simplicial completion functor $X \mapsto \scomp{X}$ is monoidal and preserves colimits, it satisfies:

 \[ \scomp{f \corner{\otimes} g  } = \scomp{f  } \corner{\times} \scomp{g \vphantom{f}} \]

\noindent and $\scomp{f} \corner{\times} \scomp{g \vphantom{f}}$ is an acyclic cofibration because of the corner-product conditions for (marked/stratified) simplicial sets. Finally Corollary~\ref{cor:freeSimpSeeTrivCof}.$\ref{cor:freeSimpSeeTrivCof:scomp_detect_acyclic_cof}$ implies that $f \corner{\otimes} g$ is an acyclic cofibrations of (marked/stratified) semi-simplicial sets because $\scomp{f \corner{\otimes} g  }$ is one.

$\Delta_+[0]$ is the unit for the monoidal product, and the dual of the self-composed span trick of \ref{lem:Span_trick} applied to $\Delta_+[1]_t$ provides a weak cylinder object for it. The small object argument applies to semi-simplicial sets in its ``good form'' (from \ref{EasySOA}) hence, Theorem~\ref{Th_tensorWMS}, proves the existence of a model structure satisfying points $\ref{th:semi-simp_WMS:generators}$,$\ref{th:semi-simp_WMS:cofibrations}$ and $\ref{th:semi-simp_WMS:monoidal}$.

It is clear that the simplicial completion functor is a left Quillen functor as it sends the generating (acyclic) cofibrations to the generating (acyclic) cofibrations. Moreover, Proposition~\ref{prop:SSSunitEquiv} applied to a cofibration $\emptyset \hookrightarrow X$ shows that the unit of adjunction $X \rightarrow \scomp{X}$ is anodyne for each (marked/stratified) semi-simplicial set $X$.

To conclude that the simplicial completion/forgetful functor is a Quillen equivalence we will use point $(v)$ of Proposition~\ref{prop:QuillenEquiv} and check that the forgetful functor detects equivalences between bifibrant objects.

Let $f\colon X \rightarrow Y$ be a morphism between two bifibrant (marked/stratified) simplicial sets such that its image by the forgetful functor is an equivalence.

We factor $f$ (in the category of simplicial sets) as $p \circ i$, with $i$ an anodyne morphism followed by a fibration $p$. Corollary~\ref{cor:freeSimpSeeTrivCof}.$\ref{cor:freeSimpSeeTrivCof:Forget_pres_anodyne}$ shows that $i$ is an equivalence both in simplicial set and in semi-simplicial sets, hence in both category $f$ is an equivalence if and only if $p$ is an equivalence, i.e. an acyclic fibration.

But being an acyclic fibration is characterized by the lifting property against maps in the image of the simplicial completion functor, so $p$ is an acyclic fibration if and only if its image by the forgetful functor is an acyclic fibration.

The forgetful functor preserve all limits and colimits so is also a right adjoint functors, and it also preserves cofibrations and anodyne morphisms, by Corollary~\ref{cor:freeSimpSeeTrivCof}.$\ref{cor:freeSimpSeeTrivCof:Forget_pres_anodyne}$, hence it is a right Quillen functor. It is already known to induce an equivalence on the homotopy category by its action on bifibrant objects, because it is a left Quillen equivalence, so it is a right Quillen equivalence. \end{proof}

\begin{remark}\label{rk:ssWMS_are_not_Quillen} None of the weak model structures given by Theorem~\ref{th:semi-simp_WMS} can be Quillen model structures: in all of them the map $\Delta_+[0] \coprod \Delta_+[0] \rightarrow \Delta_+[0]$ is a ``trivial fibrations'' (in the sense that is has the right lifting property against all cofibrations) that is not an equivalence. As all their objects are cofibrants, they are, at least classically, right semi-model categories (see \cite{henry2019CombWMS}). However, constructively it is not completely clear how a map whose target is not fibrant can be factored as a cofibration which is an equivalence followed by a ``fibration'' that have the lifting property against all cofibrations that are equivalences, while non-constructively, this factorization can be obtained as a ``left saturation'' in the sense of section 4 of \cite{henry2019CombWMS}.
\end{remark}

The end of the paper is about proving Proposition~\ref{prop:SSSunitEquiv} and Corollary~\ref{cor:freeSimpSeeTrivCof}, used in the proof above. We will focus on the case of stratified semi-simplicial sets, and the case of the weak model structure of Theorem~\ref{Th:WMS_stratifiedSimp} as it is the most general one and all the others cases easily follows from it. In particular when one says ``anodyne'' we refer to the class of maps generated by the semi-simplicial versions of the sets given in Construction~\ref{Constr:generating_stratified_trivcof}. If one is only interested in the unmarked case, this would simplify considerably the proof of lemmas \ref{lem:X->CX_anodyne} and \ref{lemma::coneDelta}, but leave the rest of the proof mostly unchanged.

\begin{remark} Another possible approach to prove Theorem~\ref{th:semi-simp_WMS} would be to rely on the proof that a semi-simplicial Kan complex can be endowed with choices of degeneracy maps making it into a simplicial set. This was originally proved in \cite{rourke1970delta} using topological methods. A combinatorial proof has been given in \cite{mcclure2013semisimplicial}, and a different combinatorial proof extending the result to the case of quasi-categories has been given in \cite{steimle2018degeneracies}. These results probably allows to give a different proof of Theorem~\ref{th:semi-simp_WMS} in the case of the Kan-Quillen and the Joyal-Lurie model structure,  bypassing the end of the paper for these cases. A version of this claim for the Verity model structure, while plausible, is unknown. Moreover we have not been able to make the proofs of \cite{mcclure2013semisimplicial} or \cite{steimle2018degeneracies} constructive, in fact we are very unsure whether the claim that semi-simplicial Kan complexes can be endowed with the structure of a simplicial sets has an interesting constructive content. Hence the rest of the paper seems necessary both for the semi-simplicial version of weak complicial sets and for the constructiveness of the semi-simplicial versions of the Kan-Quillen and Joyal-Lurie model structure. \end{remark}

Before moving to the proof of Proposition~\ref{prop:SSSunitEquiv} and Corollary~\ref{cor:freeSimpSeeTrivCof}, we need some preliminaries:

\begin{construction}
\label{constr:semi_simpe_cones}Let $X$ be a stratified semi-simplicial set. We define a stratified semi-simplicial $CX$, which is essentially a semi-simplicial version of the join of $X$ with $\Delta_+[0]$, this definition only serves a technical purpose and we do not want to develop the theory of the join, so we will give a very explicit definition of this object.

The cells of $CX$ are:

\begin{itemize}

\item For each $k$-cell $x$ of $X$, $x$ is also a $k$-cell of $CX$.

\item $*$ a cell of dimension $0$ of $CX$.

\item For each $k$-cell $x$ of $X$, $x^*$ is a $k+1$-cell of $CX$.

\end{itemize}

The face operations are defined as follows:

\begin{itemize}

\item $X$ is a subobject of $CX$, i.e. for a cell of the form $x$ for $x \in X$, face operations are as in $X$.
\item A face map $i\colon [k] \hookrightarrow [n]$ either factors through $[n-1]$, in which case it restricts to a map $i':[k] \to [n-1]$ or satisfies $i(k)=n$, in which case, as $i$ is injective, we can restrict it to a map $i':[k-1] \rightarrow [n-1]$. We define:

\[ \left\lbrace \begin{array}{cccl}
i^*(x^*) & = & i'^*(x) & \text{If $i$ restricts to $i':[k] \to [n-1]$.} \\
i^*(x^*) & = & * & \text{If $k=0$ and $i(0)=n$.}\\
i^*(x^*) & = & (i'^*(x))^* & \text{If $k \neq 0$, $i(k)=n$, and $i$ restricts} \\
        &    &            & \text{to $i':[k-1] \to [n-1]$.} \\
\end{array} \right. \]

\end{itemize}
The functoriality on $\Delta_+$ of this definition can be checked by a case by case analysis. Thinness in $CX$ is defined by the fact that $\alpha$ and $\alpha^*$ are thin in $CX$ if and only $\alpha$ is thin in $X$.

\end{construction}

\begin{example} If $X$ is $\Delta_+[n]$ then $CX$ is $\Delta_+[n+1]$ with the canonical morphism $X \rightarrow CX$ corresponding to the inclusion $[n] \subset [n+1]$. Indeed, the $0$-cell $*$ corresponds to $\{n+1\} \subset [n+1]$, if $\alpha \subset [n]$ is a cell of $\Delta_+[n]$ then the corresponding cell $\alpha$ of $\Delta_+[n+1]$ is simply $\alpha \subset [n] \subset [n+1]$, and the cell $\alpha^*$ is $\alpha \cup \{n+1\} \subset [n+1]$. It is relatively immediate to check that all face maps as defined above identifies with these of $\Delta_+[n+1]$.

For a case with markings, if $X = \Delta^k_+[n]$ then if $k<n$, $C X \simeq \Delta^k_+[n+1]$. Indeed a cell $\beta \subset [n+1]$ is marked in $\Delta^k_+[n+1]$ if and only if it contains $\{k-1,k,k+1\}$ which is a subset of $[n] \subset [n+1]$ and so cells of the form $\alpha$ or $\alpha^*$ for $\alpha \subset [n]$ are indeed thin in $\Delta^k_+[n+1]$ if and only if $\alpha$ is thin as a cell of $\Delta^k_+[n]$. 

In the case $k=n$,  $CX$ has more thin cells than $\Delta^k_+[n+1]$: the cells of $CX$ that are thin are exactly the cells that contain $\{n-1,n\}$, while thin cells of $\Delta^k_+[n+1]$ are these that contain $\{n-1,n,n+1\}$.  In particular, $C\Delta^n_+[n]$ can be described as the pushout:

\[
\begin{tikzcd}
  \Delta_+[n] \ar[d] \ar[dr,phantom,"\ulcorner"{very near end}] \ar[r] & \Delta^n_+[n+1]  \ar[d] \\
 \Delta_+^n[n] \ar[r] & C \Delta_+^n[n] \\
\end{tikzcd}
\]

\end{example}

\begin{remark}\label{rk:C_is_cocontinuous}  When seen as a functor from stratified semi-simplicial sets to \emph{pointed} stratified semi-simplicial sets (pointed by the cell $*$), $C$ commutes to all colimits, hence it is also a left adjoint functors. \end{remark}

\begin{lemma}\label{lem:X->CX_anodyne}
If $X \overset{\sim}{\hookrightarrow} Y$ is anodyne in $\widehat{\Delta_+}^s$ then:

\[ CX \coprod_X Y \rightarrow CY \]

\noindent is again anodyne.

\end{lemma}


\begin{proof}
By \cref{Rk::JTcalculus_transnat,rk:C_is_cocontinuous}, it is enough to check it in the case of the generating anodyne maps  $\Lambda_+^k [n] \hookrightarrow \Delta_+^k[n]$ and $\Delta_+^k[n]' \hookrightarrow \Delta_+^k[n]''$.

In the first case the resulting map:

\[ C \Lambda_+^k[n]  \coprod_{\Lambda_+^k[n]} \Delta_+^k[n] \rightarrow C \Delta_+^k[n] \]

 only misses cells of the form $x^*$ when $x$ is not in $\Lambda^k_+[n]$, that is the two cells: $t^*$ and $\partial^k t^*$ for $t \in \Delta_+^k[n]$ the top dimensional cell, and $\partial^k t^*$ its $k$-th face. They can be both added by a pushout of $\Lambda^k_+[n+1] \hookrightarrow \Delta^k[n+1]$, indeed $\partial^k(t^*) = (\partial^k t) ^*$, and for any $\alpha \colon  [v] \hookrightarrow [n+1]$ which contains $\{k-1,k,k+1\}$, $\alpha^* (t^*)$ is thin because if $\alpha$ factors into $[n]$, then this is $\alpha^* t$, which is thin as $\alpha$ contains $\{k-1,k,k+1\}$. If $n+1$ is in the image of $\alpha$, then this is equal to $\alpha'^*(t)^*$, where $\alpha'$ is the restriction of $\alpha$ missing $(n+1)$, this cell is thin if and only if $\alpha'^* t $ is thin in $\Delta_+^k[n]$ i.e. if $\alpha'$ contains $\{k-1,k,k+1\} \wedge [n]$ in its image, which is always the case.

In the case of $\Delta_+^k[n]' \hookrightarrow \Delta_+^k[n]''$, the resulting map 

\[ C \Delta_+^k[n]'  \coprod_{\Delta_+^k[n]'} \Delta_+^k[n]'' \rightarrow C \Delta_+^k[n]'' \]

\noindent is only making one additional cell thin ($(\partial^k t)^*$), and it is a pushout of a $\Delta_+^k[n+1]' \hookrightarrow \Delta_+^k[n+1]''$.
Indeed consider the cell $t^* \in C\Delta_+[n]$, which gives a morphism $\Delta_+[n+1] \rightarrow C\Delta_+[n]$ (in fact, an isomorphism). The corresponding map  $ \Delta_+^k[n+1]' \rightarrow C \Delta_+^k[n]' $ can be checked to preserve thinness, and taking the pushout of  $\Delta_+^k[n+1]' \hookrightarrow \Delta_+^k[n+1]''$ along the map  $\Delta_+^k[n+1]' \rightarrow C \Delta_+^k[n]'  \coprod_{\Delta_+^k[n]'} \Delta_+^k[n]''$ exactly makes the cell $(\partial^k t)^*$ thin. \end{proof}

\begin{lemma}\label{lemma::coneDelta}
The map $\Delta_+[n] \hookrightarrow \Delta_+^{n+1}[n+1]$ induced by the canonical inclusion $[n] \subset [n+1]$ is anodyne.
\end{lemma}

\begin{proof}
 For $X$ a semi-simplicial set (without marking or stratification), we consider the semi-simplicial set $C C X$ where $C $ is as constructed in \ref{constr:semi_simpe_cones}.
In order to distinguishes the cell ``$x^*$'' coming from the two applications of $C$ we will use the symbol $*$ for the first application and $+$ for the second, i.e. the cells of $CCX$ are $*$,$+$,$*^+$, $x$,$x^*$,$x^+$ and $x^{*+}$ for $x$ a cell of $X$.

We will define a stratified semi-simplicial set $DX$ whose underlying semi-simplicial set is $CCX$ and in which the thin cells are all the cells of the form $*^+$ and $x^{*+}$. And we consider the natural inclusion of $\eta_x\colon  C X \hookrightarrow D X$ sending the cells $*, \alpha$ or $\alpha^*$ to the cells with the same name.

We claim that for all semi-simplicial set $X$, the map $C X \hookrightarrow D X$ is anodyne. Applying this to $X = \Delta_+[n-1]$ (or $\emptyset$ for $n=0$) immediately gives the lemma.

This claim can be proved by induction on cells of $X$, indeed for $X = \emptyset$, $CX = \Delta_+[0]$ and $D X = \Delta_+^1[1]$ so that $\eta_{\emptyset}$ is one of our generating anodyne map. Everytime we add a $k$-cell $x$ to $X$ (to get a new semi-simplicial set $X'$), it adds two cells $x,x^*$ to $CX$ and two additional cells $x^+$ and $x^{*+}$ to $D X$. The map $C X'= CX \cup\{x,x^*\} \hookrightarrow D X \cup \{x,x^*\}$ is already known to be anodyne by induction, as it is a pushout of $CX \to DX$, so it remains to see that $D X \cup \{x,x^*\} \hookrightarrow D X' =  D X \cup\{x,x^*,x^+,x^{*+} \}$ is anodyne.

If $x$ is a $k$-cell, then $x^{*+}$ is a $(k+2)$-cell, and $\partial^{k+1} x^{*+} = x^+$, moreover any $\lambda \colon  [n] \rightarrow [k+2]$ which contains $\{k+1,k+2\}$ in its image satisfies $\lambda^*(x^{*+}) = (\lambda^{*}(x))^{*+}$ where $\lambda'$ is the restriction of $\lambda$ as a map $[n-2] \rightarrow [k]$, in particular $\lambda^*(x^{*+})$ is thin. This shows that the map $D X \cup \{x,x^*\} \hookrightarrow D X' = D X \cup\{x,x^*,x^+,x^{*+} \}$ is a pushout of $\Lambda^{k+1}[k+2] \overset{\sim}{\hookrightarrow} \Delta^{k+1}[k+2]$ and proves the lemma. \end{proof}

\begin{prop}\label{prop:SSSunitEquiv}
For any $i\colon X \hookrightarrow  Y$ a cofibration of stratified semi-simplicial sets, the map:

\[ \scomp{X} \coprod_X Y \hookrightarrow \scomp{Y} \]

\noindent is an anodyne map of stratified semi-simplicial sets.

\end{prop}

\begin{proof}
Note that this map is indeed a cofibration (it is easy to check from the explicit formula $\scomp{X}$). As $X \mapsto \scomp{X}$ is a left adjoint functor, checking that the proposition is true for $\Delta_+[n] \hookrightarrow \Delta_+[n]_t$ and $\partial \Delta_+[k] \hookrightarrow \Delta_+[k]$ for all $k<n$ implies that it is automatically true for any cofibration $X \hookrightarrow Y$ of stratified simplicial set such that the cell in $Y$ not in $X$ are of dimension $<n$.

Note that in the case where the map $X \rightarrow Y$ is an isomorphism of the underlying semi-simplicial set (so that it is only a change of stratification) then the map $\scomp{X} \coprod_X Y \hookrightarrow \scomp{Y}$ is an isomorphism. Hence the proposition automatically holds for the  $\Delta_+[n] \hookrightarrow \Delta_+[n]_t$. 

\bigskip

We will prove this claim by induction, more precisely we assume that the result holds for all $\partial \Delta_+[k] \hookrightarrow \Delta_+[k]$ for $k<n$, and hence for any cofibration between objects of dimension $<n$, and we will show that it holds for $\partial \Delta_+[n] \hookrightarrow \Delta_+[n]$, i.e. that:

\[\partial \Delta [n] \coprod_{\partial \Delta_+[n]} \Delta_+[n] \hookrightarrow \Delta[n] \]

\noindent is anodyne, where $\Delta[n]$ and $\partial \Delta [n]$ are endowed with their stratification coming from the category $\widehat{\Delta}^s$ of stratified simplicial sets, i.e. all the degenerate cells are thin.

The $k$-cells of $\Delta[n]$ are all maps $[k] \rightarrow [n]$. The subobject 

\[ S= \partial \Delta [n] \coprod_{\partial \Delta_+[n]} \Delta_+[n] \]

\noindent corresponds to all non-surjective maps, and the identity of $[n]$. This map does not appears to be directly a (transfinite) composite of pushout of the generating cofibrations, but only a retract of such map, so we need to explicitly construct ``bigger'' objects these will be retract of.

We define $T_n$ the semi-simplicial set such that:
\[ T_n([k]) \coloneqq  \{ f\colon  [k] \rightarrow [n] \cup \{*\} | \text{ $f$ is order preserving and $f^{-1}\{*\} = \emptyset$ or $\{k\}$ } \} \]

(where ``$*$'' is added as a maximal element of $[n]$.)

$\Delta[n]$, seen as a semi-simplicial sets, naturally identify as a retract of $T_n$:

\[\Delta[n] \rightarrow T_n \rightarrow \Delta[n] \]

Where the first map corresponds to the inclusion of the sub-complex of cells such that $f^{-1}\{*\} = \emptyset$ and the second map send a cell $[k] \rightarrow [n] \cup \{*\}$ to its composite with the map sending $*$ to $n$. We endow $T_n$ with the stratification where a cell is thin if and only if its image in $\Delta[n]$ is thin, and this retraction is in the category of stratified semi-simplicial sets. 

In particular it is enough to show that the composite:

\[ \partial \Delta [n] \coprod_{\partial \Delta_+[n]} \Delta_+[n] \hookrightarrow \Delta[n] \hookrightarrow T_n \]

 is anodyne.

If $x$ is a cell of $\Delta[n]$ of dimension $k$, then we denote by $x^*$ the unique cell of $T_n$ of dimension $k+1$ which is not in $X$ and such that $\partial^{k+1}x^*=x$, i.e. $x^*$ is $x$ on $[k] \subset [k+1]$ and $*$ on $k+1$. The cells of $T_n$ are exactly the $x \in \Delta[n]$, the $x^* \in \Delta[n]$ and one additional cell of dimension $0$, denoted $*$. So as semi-simplicial sets $T_n=C \Delta[n]$ (but the stratification are not the same).

We now define for any $i\geqslant n$:
\[ T_n^i([k]) = \{ \alpha \in T_n([k]) | \alpha^{-1}[n] \rightarrow [n] \text{ is not surjective} \text{ or } |\alpha^{-1} [n] | \leqslant i\} \]

( $|\alpha^{-1} [n] |$ denote the cardinal of  $\alpha^{-1} [n]$).

We then check that $T^{i-1}_n \hookrightarrow T^i_n$ is a pushout of a coproduct of several copies of $\Lambda^{i+1}_+[i+1] \hookrightarrow \Delta^{i+1}_+[i+1]$. First, the cells of $T^i_n$ that are not in $T^{i-1}_n$ are exactly the $\alpha \in \Delta[n]$ which are surjective and of dimension $i$, and the $\alpha^*$ for such $\alpha$. For each such $\alpha$ we can add $\alpha^*$ and $\alpha$ together with a pushout of $\Lambda^{n+1}_+[i+1] \hookrightarrow \Delta^{i+1}_+[n+1]$. Indeed, $\alpha^*$ is a cell of dimension $i$, such that all its faces except its $(i+1)$-face are in $T^{i-1}_n$ and its $i+1$-face is $\alpha$. Moreover, $\alpha^*$ is always thin and for any $v\colon [u] \rightarrow [i+1]$ which contains $i$ and $i+1$ in its image, $v^*(a^*)$ is always thin, as its image in $\Delta^n$ will take the value $n$ at least twice (in $i$ and $i+1$) so is a non-injective cell.

This proves that $T^n_n \rightarrow T_n$ is anodyne. So it remains to show that

\[ \partial \Delta [n] \coprod_{\partial \Delta_+[n]} \Delta_+[n] \hookrightarrow T^n_n \]

\noindent is anodyne. Note that at the level of the underlying semi-simplicial sets $T^n_n$ is exactly $C ( \partial \Delta [n] \coprod_{\partial \Delta_+[n]} \Delta_+[n]) $, but endowed with a different stratification. More precisely, there is a morphism:

 \[ C \left( \partial \Delta [n] \coprod_{\partial \Delta_+[n]} \Delta_+[n] \right) \rightarrow T^n_n \]

Which makes thin the cells $a^*$ for $a \in \Delta_+[n]$ which contains $n$ in its image. Indeed the cells of $T^n_n$ are thin if and only if there image in $\Delta[n]$ (by the map sending $*$ to $n$) is thin, i.e. non injective, while a cell $a$ or $a^*$ in $ C ( \partial \Delta [n] \coprod_{\partial \Delta_+[n]} \Delta_+[n])$ is thin if and only if $a$ is non-injective. So the only case a cell can be non-thin in $C ( \partial \Delta [n] \coprod_{\partial \Delta_+[n]} \Delta_+[n])$ and thin in $T^n_n$ is if it is of the form $a^*$, with $a$ injective, but the image of in $\Delta[n]$ non-injective, hence, with $a \in \Delta_+[n]$ but containing $n$ in its image.

By our induction hypothesis, the map $\partial \Delta_+[n] \rightarrow \partial \Delta [n] = \scomp{\partial \Delta_+[n]}$ is anodyne, hence by Lemma~\ref{lem:X->CX_anodyne} applied to its pushout $\Delta_+[n] \hookrightarrow \partial \Delta [n] \coprod_{\partial \Delta_+[n]} \Delta_+[n]$, the map:

\begin{equation}\label{arrow1} \partial \Delta[n] \coprod_{\partial \Delta_+[n]} C( \Delta_+[n]) \hookrightarrow C \left( \partial \Delta [n] \coprod_{\partial \Delta_+[n]} \Delta_+[n] \right) \end{equation}

\noindent is also anodyne. Note that $C( \Delta_+[n])$ is exactly $\Delta_+[n+1]$, and making thin all the cells $\alpha^*$ for $\alpha \in \Delta_+[n]$ which contains $n$ in their image, exactly means making all the cells of $\Delta_+[n+1]$ which contains $n$ and $n+1$ thin, i.e. it is the marking of $\Delta^{n+1}_+[n+1]$. This means that: 

\[\partial \Delta[n] \coprod_{\partial \Delta_+[n]} \Delta^{n+1}_+[n+1] \hookrightarrow T^n_n \]

\noindent is anodyne as a pushout of the map (\ref{arrow1}) (the pushout just serving to make a few additional cells thin). Finally $\Delta_+[n] \hookrightarrow \Delta^{n+1}_+[n+1] $ is anodyne by Lemma~\ref{lemma::coneDelta} and hence this shows that 

\[\partial \Delta[n] \coprod_{\partial \Delta_+[n]} \Delta_+[n] \hookrightarrow T^n_n \]

\noindent is anodyne and concludes the proof. \end{proof}

\begin{cor}\label{cor:freeSimpSeeTrivCof}
\begin{enumerate}[label=(\roman*)]

\item[]

\item\label{cor:freeSimpSeeTrivCof:scomp_pres_anodyne} If $f \colon X \overset{\sim}{\hookrightarrow} Y$ is anodyne in $\widehat{\Delta_+}^s$, then $\scomp{f} \colon  \scomp{X} \rightarrow \scomp{Y}$ is also anodyne in $\widehat{\Delta_+}^s$.

\item\label{cor:freeSimpSeeTrivCof:Forget_pres_anodyne} If $f \colon X \rightarrow Y$ is anodyne in $\widehat{\Delta}^s$, then its image in $\widehat{\Delta_+}^s$ is also anodyne.

\item\label{cor:freeSimpSeeTrivCof:scomp_detect_acyclic_cof} If $f \colon  X \rightarrow Y$ is a cofibration in $\widehat{\Delta_+}^s$ and $\scomp{f}\colon \scomp{X} \rightarrow \scomp{Y}$ is an acylic cofibration in $\widehat{\Delta_+}^s$ or is anodyne in $\widehat{\Delta}^s$ then $f$ is an acyclic cofibration in  $\widehat{\Delta_+}^s$.

\end{enumerate}

\end{cor}

\begin{proof}

\begin{enumerate}[label=(\roman*)]

\item As $A \overset{\sim}{\hookrightarrow} B$ is anodyne, the map:

\[ \scomp{A} \overset{\sim}{\hookrightarrow} \scomp{A} \coprod_A B \]

\noindent is also anodyne, and by Proposition~\ref{prop:SSSunitEquiv}, the map 

\[ \scomp{A} \coprod_A B \hookrightarrow \scomp{B} \]

\noindent is anodyne, which proves the claim.

\item The forgetful functor from $\widehat{\Delta}^s$ to $\widehat{\Delta_+}^s$ is a left adjoint functor. Hence it is enough to check the result on generating anodyne map: $\Lambda^k[n] \hookrightarrow \Delta^k[n]$ and $\Delta^k[n] ' \hookrightarrow \Delta^k[n]''$, i.e. that these map are anodyne in $\scomp{\Delta_+}^s$. But this follows immediately from the previous point applied to $\Lambda^k_+ \hookrightarrow \Delta_+^k[n]$ and $\Delta_+^k[n] ' \hookrightarrow \Delta_+^k[n]''$.

\item We consider the square:

\[\begin{tikzcd}[ampersand replacement=\&]
X \arrow[hook]{r}{f} \arrow[hook]{d}{\sim} \& Y \arrow[hook]{d}{\sim}  \\
\scomp{X} \arrow[hook]{r}{\scomp{f}} \& \scomp{Y} 
\end{tikzcd}\]

\noindent in $\widehat{\Delta_+^s}$. Because of the previous, point if $\scomp{f}$ is anodyne in $\widehat{\Delta}^s$, then it is also anodyne in $\widehat{\Delta_+}^s$. So in both case, the composite:

\[ X \hookrightarrow Y \overset{\sim}{\hookrightarrow} \scomp{Y} \]

\noindent is acyclic in $\widehat{\Delta_+}^s$, hence this implies that $X \hookrightarrow Y$ is an acyclic cofibration (last point of Lemma~\ref{lem:acyclicCof_basics})

\end{enumerate}

\end{proof}

\appendix
\section{Setoids}
\label{section_setoids}
\subsection{Preliminaries on Setoids and Setoid-categories}
\label{section_prelim_setoids}
Setoids are a way to represent ``quotient sets'' without actually taking quotient. A setoid is given by an underlying set $X$ endowed with an equivalence relation, except that the equivalence relation does not have to be subset of $X \times X$, but only a set endowed with two maps to $X$:

 \[ X_R \rightrightarrows X.\]

So this is what we might want to call a ``proof relevant equivalence relation''. More precisely:

\begin{definition}\label{def:setoids}
A \emph{Setoid} $X$ is the data of:

\begin{itemize}
\item A set of elements $X$.
\item A set of relations $X_R$ with two maps $s,t\colon X_R \rightrightarrows X$. An element of $a\in X_R$ such that $s(a)=x$ and $t(a)=y$ is represented by $x \overset{a}{\Rightarrow} y$ or $a\colon  x \Rightarrow y$.

\item For each $x \in X$ there is a chosen relation $\refl{x}\colon  x \Rightarrow x$.
\item For each relation $a\colon  x \Rightarrow y$, there is a chosen relation $\inv{a}\colon  y \Rightarrow x$.
\item For each pair of ``composable'' relations: $a\colon  x \Rightarrow y$, $b\colon  y \Rightarrow z$ there is a composed relation $\comp{a}{b}\colon x \Rightarrow z$.
\end{itemize}

\end{definition}

But no other axioms  (``associativity'' of the composition, or compatibility between composition and inverse) are required.

We define moreover:
\begin{definition}\label{def:setoids_properties}
\begin{enumerate}[label=(\roman*)]

\item[]

\item A morphism of setoids $f\colon X \rightarrow Y$ is a morphism of the underlying graphs $(X,X_R) \rightarrow (Y,Y_R)$.

\item\label{def:setoids_properties:relation_morphisms} A relation $r\colon f \Rightarrow g$ between two morphisms $f,g\colon X \rightrightarrows Y $ is a function $r$ from $X$ to $Y_R$ such that for all $x$, $r(x)\colon f(x) \Rightarrow g(x)$.

\item A morphism $f \colon X \rightarrow Y$ of setoids is said to be an injection if for each relation $ r \colon  f(x) \Rightarrow f(y)$ in $Y$, there is a chosen relation $f^{inj}(r)\colon x \Rightarrow y$.

\item A morphism $f \colon X \rightarrow Y$ of setoids is said to be a surjection if for all $y \in Y$ there is a chosen $f^s(y)\in X$ and a chosen $f^{sw}(y)\colon y \Rightarrow f(f^s(y))$.

\item\label{def:setoids_properties:isomorphisms} A morphism of setoids is said to be an isomorphism if it is both a surjection and an injection.

\item If $X$, $Y$ and $Z$ are setoids, a 2-variable function $f\colon X \times Y \rightarrow Z$ means a function which to every $x \in X$ and $y \in Y$ associate $f(x,y) \in Z$, to every $\alpha\colon x_1 \Rightarrow x_2$ in $X_R$ and $y \in Y$ associate $f(\alpha,y) \colon  f(x_1,y) \Rightarrow f(x_2,y)$ and to every $\beta\colon y_1 \rightarrow y_2$ and $x \in X$ associate $f(x,\beta)\colon f(x,y_1) \Rightarrow f(x,y_2)$. 

\end{enumerate}
\end{definition}

Of course\footnote{if we are working in a regular category.} if $X$ is a setoid, then ``$\exists r\colon  x \Rightarrow y$'' is an equivalence relation on the sets of vertices of $R$, and for any setoid $X$ there is an associated a quotient set $|X|$. Assuming the axiom of choice, two setoids are isomorphic (in the sense of existence of an ``isomorphisms'' as above) if and only if there quotient set are isomorphic, and the categories of setoids (with equivalence class of maps between them) is equivalent to the category of set through this quotient set functor. But this statement is exactly equivalent to the axiom of choice.

There are essentially two reasons to introduce setoids:

\begin{itemize}

\item If we work in weaker logical framework where quotients and/or existential quantifications are not available (like in Martin-Löf type theory, or in the internal logic of a category with finite limits) then they actually replace the use of quotient.

\item If we work without the axiom of choice, then setoids keep track of more information than the quotient sets, and this information can sometimes be relevant.

\end{itemize}

In the present paper we are mostly interested by the second aspect: the use of this extra information that setoids carry will allows us to recover some constructive characterization of equivalences as the maps that ``induce bijections on all $\pi_n$'' where the $\pi_n$ will be defined as setoids. And it is known that a similar characterization in terms of $\pi_n$ defined as sets fail. We will also use setoids to define the homotopy category without referring to existential quantification or quotient sets, but the real reason we are doing this is because it makes the treatement of $\pi$-setoids smoother if the homotopy category is defined in terms of setoids.

We also emphasize that when talking about setoids we consider the precise data of the ``transitivity'', ``reflexivity'' and ``symmetry'' operations on its relation completely irrelevant. We only care about the fact that they exists and that each setoids comes with a canonical choice of these. This is made apparent in the fact that they do not play any role in the definition of morphisms, so that two different setoid structures on a graph are automatically isomorphic as setoids. In particular in the rest of the paper when we say that something is a setoid we will often not make the choice of these operations explicit, but we always mean that at least one explicit choice exists. Similarly for the ``structure'' of being an injection, a surjection or an isomorphism on a morphism of setoids.

\begin{remark}\label{rk:Setoid_ForExists_conventions} If we follow the convention explained in Section~\ref{sec:Logical_framework} that every statement of the form ``$\forall x \exists y $'' should be interpreted as the existence of a function attaching a $y$ to each $x$. Then the fact that a morphism of setoids is injective can be written more naively as ``if $f(x) \sim f(y)$ then $x \sim y$ '' (where $\sim$ means there is a relation between $x$ and $y$), i.e. $\forall r\colon f(x) \Rightarrow f(y), \exists r' \colon  x \Rightarrow y$). Similarly, surjectivity of $f\colon X \rightarrow Y$ can be rewritten as for all $y \in Y$ there is an $x \in X$ such that $f(x) \sim y$.
\end{remark}

The following easy lemma should be noted:

\begin{lemma}
A setoid morphism $f \colon X \rightarrow Y$ is an isomorphism if and only if it is invertible in the category of setoids and equivalence classes of morphisms, i.e. if there is a setoid morphism $g\colon Y \rightarrow X$ and relations $f \circ g \Rightarrow Id_Y$ and $g \circ f \Rightarrow Id_X$.
\end{lemma}

We mean here that given the structure of an isomorphism on $f$ we can construct explicitly such an inverse, and that conversely given the structure of such an inverse we can construct the structure of an isomorphism of $f$. The proof is an immediate translation of the usual fact that an injective and surjective map is bijective using the convention of Remark~\ref{rk:Setoid_ForExists_conventions}.

\begin{definition}

A setoid-category $C$ is the data of the following structure:

\begin{itemize}

\item A set of objects $C_o$.

\item For each pair of objects $x$, $y$ in $C_o$ a setoid of arrows $C(x,y)$ from $x$ to $y$.

\item For each object $x \in C_o$ a chosen arrow $Id_X \colon  x \rightarrow x$.

\item For each $x,y,z \in C_o$, a 2-variables composition morphism:
\[ \_ \circ \_ \colon  Hom(y,z) \times Hom(x,y) \rightarrow Hom(x,z). \]

\item For each arrow $f \colon  x \rightarrow y $ two chosen ``identity witnesses'': 

\[ l_f\colon  (f \circ Id_x) \Rightarrow f \quad \text{ and } \quad r_f\colon  (Id_y \circ f) \Rightarrow f. \]

\item For each triple of composable arrows $ f,g,h$ an associativity witness: 

\[ \alpha_{f,g,h} \colon  (f \circ g) \circ h \Rightarrow f \circ (g \circ h). \]

\end{itemize}

\end{definition}

Of course this definition is engineered so that if we take the quotient set of all the setoid of morphisms we get an ordinary category (the homotopy category in some sense).

Very similarly, and respecting the idea that everything that we need in the definition should be given by some operations, and not using any kind of existential or universal quantification, we also define the following notions:

\begin{itemize}

\item Functors between setoid-categories.

\item Presheaves of setoids on a setoid-category.

\item Invertible arrows in a setoid-category.

\item Fully faithful functors and essentially surjective functors.

\end{itemize}

And we can check that:

\begin{itemize}

\item Given two setoids, morphisms between them and relations between these morphisms form a setoid.

\item This makes the category of setoids into a setoid-category.

\item A presheaf is the same as a contravariant functor to the category of setoid.

\item we can define the Yoneda embedding and prove the Yoneda lemma.

\item A functor $F\colon  \Ccal \rightarrow \Dcal$ between setoid-categories is fully faithful and essentially surjective if and only if there is a functor $G \colon  \Dcal \rightarrow \Ccal$ and natural isomorphisms $\lambda\colon  G \circ F \rightarrow Id_{\Ccal}$ $\mu\colon  F \circ G \rightarrow Id_{\Dcal}$.

\end{itemize}

\subsection{$\pi$-Setoids}
\label{section_pisetoids}

The goal of this subsection is to show how we can get back the usual simpler characterization of equivalences in terms of ``bijection on all $\pi_n$''. If $\pi_n$ are defined as sets this cannot be constructive. But in a rather unexpected way, it appears that by defining the $\pi_n$ as setoids we do get such a characterization.

In all this section we fix $\Ccal$ a weak model category.

\begin{definition}
Let $i\colon A \hookrightarrow B$ be a cofibration with cofibrant domain, let $X$ be a fibrant object of $\Ccal$ and let $x\colon A \rightarrow X$ be any morphism. We define:

\[ \pi_{i}(X,x) \coloneqq  Hom_{Ho(A/\Ccal)}(B,X) \]

\noindent as a setoid.

\end{definition}

We will also use the notation $\pi_{B/A}(X,x)$. More explicitly, $\pi_i(X,x)$ is the setoid of maps from $B \rightarrow X$ which makes the triangle:

\[\begin{tikzcd}[ampersand replacement=\&]
A \arrow{r}{x} \arrow[hook]{d}{i} \& X \\
B \arrow[dotted]{ur} \\
\end{tikzcd}\]

\noindent commutes, and the relation is given by the homotopy relation in $A/\Ccal$, that is the homotopy relation relative to $A$, which is either parametrized by maps $I_A B \rightarrow X$ or maps from $B$ to $PX$ such that the restriction to $A$ is a trivial homotopy. The choice of the path or cylinder is irrelevant and it is a setoid.

\begin{remark}\label{rk:basic_functoriality_pi_setoid}
If $f\colon  X \rightarrow Y$ is any map between two fibrant objects there is a morphism of setoids:

\[ \pi_{i}(f,x) \colon  \pi_{i}(X,x) \rightarrow \pi_{i}(Y,f(x)) \]

If $f$ is an equivalence between two fibrant objects then all these maps $\pi_i(f,x)$ are isomorphisms of setoids because of the Hom-set definition of $\pi$-setoids.

Conversely, if all the $\pi_i(f,x)$, for all $i$ and all $x$, are bijections then $f$ is an equivalence: in fact only asking this for $i\colon \emptyset \hookrightarrow A$ already shows means that $Hom_{Ho(\Ccal)}(A,f)\colon  Hom_{Ho(\Ccal)}(A,X) \rightarrow Hom_{Ho(\Ccal)}(A,Y)$ are bijections for all cofibrant objects $A$ and as every object in the homotopy category is equivalent to a cofibrant object this immediately gives that $f$ is an isomorphism in the homotopy category (in fact it is enough to know it for $i\colon \emptyset \hookrightarrow X$ and $i\colon \emptyset \hookrightarrow Y$).

Our goal is to find  more convenient small set of cofibrations $i$ on which to test whether a map is an equivalence. For example, in the category of spaces we only want to test in the case of the maps $i \colon  \{* \} \hookrightarrow \Scal^n$ from a point to the $n$-sphere.

\end{remark}

\begin{example}\label{ex:surjectivity_on_Pi} Given a morphism $f\colon X \rightarrow Y$ saying that the induced morphisms:

\[ \pi_i(X,x) \overset{\pi_i(f,x)}{\rightarrow} \pi_i(Y,f(x)) \]

\noindent is a surjection of setoids means that given a square of the form:

\[
\begin{tikzcd}
  A \ar[d,hook,"i"] \ar[r,"x"] & X \ar[d,"f"] \\
B \ar[r,"y"] & Y \\
\end{tikzcd}
\]

Admit a diagonal filling such that the upper triangle commutes and the lower triangle commutes up to homotopy relative to $A$. Indeed such a square means that $y$ is an element of $\pi_i(Y,f(x))$, and surjectivity of $\pi_i(f,x)$ means that to each such square we can attach an element of $v \in \pi_i(X,x)$, i.e. a diagonal filling making the upper triangle commutes, and a relation in $\pi_i(Y,f(x))$ between $y$ and $f(v)$, i.e. a homotopy $h$ relative to $A$ making the lower triangle commutes. This filling can be represented as a diagram:

\[\begin{tikzcd}
A \arrow[hook]{dr}{i} \arrow{rr}{x} \arrow[hook]{ddd}{i} & &   X \arrow{ddd}{f}  \\
& B \arrow[hook]{d} \ar[ur,dotted,"v"] & \\
&  I_A B \ar[dr,dotted,"h"] & \\
B \arrow[hook]{ur} \arrow{rr}{y} & & Y \\
\end{tikzcd}\]

We say that the map $f$ has the \emph{weak right lifting property} against $i$.

\end{example}

 We start by some lemmas on invariance properties of the $\pi$-setoids.

\begin{lemma}\label{Lem_Pi_sets}
\begin{enumerate}[label=(\arabic*)]

\item[]

\item\label{Lem_Pi_sets:Inv_equiv_target} Any isomorphism $(B,i) \rightarrow (B',i')$ in $Ho(A/\Ccal^{\cofe})$ induces an isomorphism $\pi_{i}(X,x) \simeq \pi_{i'}(X,x)$, natural in $(X,x) \in A/\Ccal$, by pre-composition.

\item\label{Lem_Pi_sets:Inv_POdomain} Given a pushout square of cofibrant objects:

\[  \begin{tikzcd}
    A \ar[d,"g"] \ar[r,hook,"i"] \ar[dr,phantom,"\ulcorner"very near end] & B \ar[d,"g'"] \\
    A' \ar[r,hook,"i'"] & B' \\
  \end{tikzcd} \]

\noindent then for any map $x\colon A' \rightarrow X$, then pre-composition with $g'$ induces a natural (in $(X,x) \in A'/\Ccal$) isomorphism of setoids:

\[ \pi_{i'}(X,x) \overset{\sim}{\rightarrow} \pi_i(X,x \circ g) \]

\item\label{Lem_Pi_sets:Inv_homotopy_of_Basepoint} If $h\colon  IA \rightarrow X $ is a homotopy between two maps $x,x'\colon  A \rightrightarrows X$ then there is\footnote{See the proof below for its precise construction.} an isomorphism of setoids:

\[ \pi_{i}(X,x) \simeq \pi_i(X,x') \]

\noindent natural in $(X,h) \in IA/\Ccal$.

\item\label{Lem_Pi_sets:RLP=Epi_on_Pi_i} A fibration between fibrant objects $ p\colon  X \twoheadrightarrow Y$ has the right lifting property with respect to $i\colon A \hookrightarrow B$ if and only if the map $ \pi_i(X,x) \overset{\pi_i(p,x)}{\rightarrow} \pi_i(X,p(x))$ is surjective for all $x\colon A \rightarrow X$.

\end{enumerate}
\end{lemma}

\begin{proof}
Point \ref{Lem_Pi_sets:Inv_equiv_target} is trivial from the definition in terms of homotopy Hom-setoids. Point \ref{Lem_Pi_sets:Inv_POdomain}, when formulated in terms of the homotopy Hom-setoid definition corresponds to the adjunction formula in the homotopy category of Proposition~\ref{Prop_QuillenFunctorHomotopy} for the Quillen pair: $P_f \colon  A/\Ccal \leftrightarrow A'/\Ccal \colon  U_f$. For the third we need to construct the isomorphism:

Given a cofibration $A \hookrightarrow B$ and a cylinder object $IA$ we construct a cylinder object $IB$ for $B$ such that there is a cofibration $IA \hookrightarrow IB$ compatible to the boundary inclusion. This can be done by factoring the map $B \coprod_A IA \coprod_A B \hookrightarrow IB \overset{\sim}{\twoheadrightarrow} B^f$ where $B^f$ is a fibrant replacement of $B$. Using \ref{Lem_Pi_sets:Inv_POdomain} we obtain a bijection:

\[ \pi_{i}(X,x) \overset{\simeq}{\rightarrow} \pi_{i'}(X,h) \]

\noindent where: $i'\colon IA \hookrightarrow B \coprod_A IA$. Now the map $B \coprod_A IA \rightarrow IB$ is a homotopy equivalence (in $A/\Ccal)$ hence by point $1.$ there is an isomorphism:

\[\pi_{i''}(X,h) \overset{\simeq}{\rightarrow} \pi_{i'}(X,h) \]

With $i''$ the cofibration $i''\colon IA \hookrightarrow IB$. The same construction for $x'$ gives us an isomorphism:

\[ \pi_{i}(X,x) \simeq \pi_{i''}(X,h) \simeq \pi_{i}(X,x') \]

As all the individual isomorphisms mentioned are natural in $X$, the total bijection is also natural in $X$.

For \ref{Lem_Pi_sets:RLP=Epi_on_Pi_i}, we have seen in Example~\ref{ex:surjectivity_on_Pi} that saying that $\pi_i(p,x)$ is surjective for all $x$, means that $p$ has the weak right lifting property against $i$. In particular, this will be the case if $p$ has the actual right lifting property against $i$. Conversely, if $p$ is a fibration with this weak lifting property, then any lifting problem against a cofibration $i$ can be, as in Example~\ref{ex:surjectivity_on_Pi}, extended into:
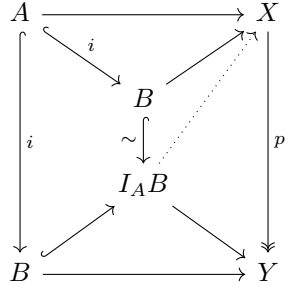
\[\begin{tikzcd}[ampersand replacement=\&]
A \arrow[hook]{dr}{i} \arrow{rr} \arrow[hook]{ddd}{i} \& \&   X \arrow[two heads]{ddd}{p}  \\
\& B \arrow[hook]{d}[swap]{\sim} \arrow{ur} \& \\
\&  I_A B \arrow{dr} \arrow[dotted]{uur} \& \\
B \arrow[hook]{ur} \arrow{rr} \& \& Y \\
\end{tikzcd}\]

Hence we can construct the dotted diagonal lift using that $p$ is a fibration and $B \hookrightarrow I_A B$ is an acyclic cofibration, and this gives a diagonal lift, which concludes the proof. \end{proof}

\begin{definition}
In a weak model category $\Ccal$, a set of cofibrations $I$ is said to be a pseudo-generating set of cofibrations if any fibration between fibrant objects which has the lifting property against all maps in $I$ is an acyclic fibration.
\end{definition}

\begin{theorem}\label{Th:Charac_WE_Pi_Setoids}
Let $\Ccal$ be a weak model category with $I$ a pseudo-generating set of cofibrations of $\Ccal$.

Then a map $f$ between fibrant objects is an equivalence if and only if it induces a surjection of setoids:

 \[ \pi_i(X,x) \rightarrow \pi_i(Y,f(x)) \]

\noindent for all $i\colon A \hookrightarrow B$ in $I$ and $x\colon A \rightarrow X$.

\end{theorem}

See Proposition~\ref{prop:Charac_of_simp_equiv} for an example of how this theorem, combined with the various invariance properties of $\pi$-setoids proved in Lemma~\ref{Lem_Pi_sets} can be used to recover usual characterizations of equivalences in concrete model categories.

\begin{proof}
Let $\widetilde{X}$ be a bifibrant replacement of $X$ and consider an (acyclic cofibration,fibration) factorization of the composite map:

\[\begin{tikzcd}[ampersand replacement=\&]
\widetilde{X} \arrow[hook]{r}{\sim} \arrow[two heads]{d}{\sim} \&   V \arrow[two heads]{d}{p}  \\
X \arrow{r}{f} \& Y \\
\end{tikzcd}\]

As the top map and the left maps are equivalences between fibrant objects, the right map satisfies the same condition as $f$ of surjectivity on $\pi$-sets, and hence, as it is a fibration, by the last point of Lemma~\ref{Lem_Pi_sets} it has the right lifting property with respect to $I$, hence it is an acyclic fibration, hence an equivalence and hence $f$ is an equivalence. \end{proof}

\begin{remark}
Using Example~\ref{ex:surjectivity_on_Pi}, this theorem can be rephrased in a way not involving $\pi$-setoids explicitly. It says that a morphism between fibrant objects is a weak equivalence if and only if it has the weak right lifting property (as in Example~\ref{ex:surjectivity_on_Pi}) against a set of pseudo-generating cofibrations.

This is essentially the ``HELP lemma'' of R.~M.~Vogt in \cite{vogt2011help}, or the observation by J.~Bourke in \cite{bourke2017equipping} that the map $f$ between fibrant objects is an equivalence if and only if it is an injective object in the category of arrows against the arrow from $A \hookrightarrow B$ to $B \hookrightarrow I_A B$.
\end{remark}

\makeatletter 
\renewcommand{\thetheorem}{\thesection.\arabic{theorem}} 
\makeatother

\section{Corner-product and Joyal-Tierney calculus}
\label{Subsection_Joyal_tierney_calculus}

This appendix reviews the now well-known ``Joyal-Tierney calculus'' introduced in \cite{joyal2006quasi}, though lots of aspects involved here were known before.

Let $\Ecal_1$, $\Ecal_2$ and $\Ecal_3$ be three complete and cocomplete categories endowed with a functor:

\[ \begin{array}{c c c}
\Ecal_1 \times \Ecal_2 & \rightarrow & \Ecal_3 \\
(A,B) & \mapsto & A \odot B
\end{array}
\]

\begin{definition}
\label{DivisibleOnbothSide}

We say that $\odot$ is \emph{left divisible} if for all $X_1 \in \Ecal_1$ the functor $X_2 \mapsto X_1 \odot X_2$ has a right adjoint, denoted $X_3 \mapsto X_1 \backslash X_3$, and that it is right divisible if for all $X_2 \in \Ecal_2$ the functor $X_1 \mapsto X_1 \odot X_2$ has a right adjoint, denoted $X_3 \mapsto X_3/X_2$. That is, $\odot$ is divisible on both sides (we just say ``divisible'' in that case) if there are adjunction isomorphisms:

\[ Hom(X_1 \odot X_2,X_3) \simeq Hom(X_1,X_3/X_2) \simeq Hom(X_2,X_1 \backslash X_3) \]

\noindent for $X_i \in \Ecal_i$. Note that $/$ and $\backslash$ are automatically functors $\Ecal_3 \times \Ecal_2^{op} \rightarrow \Ecal_1$ and $\Ecal_1^{op} \times \Ecal_3 \rightarrow \Ecal_2$.

\end{definition}

\begin{example}
\label{Ex::divisiblebifunctor}We mostly have three types of divisible functor in mind here:

\begin{itemize}

\item $\Ecal_1=\Ecal_2=\Ecal_3$ is a monoidal closed category, $\odot$ is the tensor product, and $X \backslash Y$ and $Y / X$ correspond to the left and right internal Hom object.

\item $\Ecal_1$ is a monoidal category and $\Ecal_2 = \Ecal_3$ is a tensored and co-tensored $\Ecal_1$-enriched category. Then $Y /X$ corresponds to the $\Ecal_1$-valued Hom object, $\odot$ is the tensoring action of $\Ecal_1$ on $\Ecal_2$ and $X \backslash Y$ is the cotensor action.

\item If $\Ecal$ and $\Fcal$ are complete cocomplete categories, $\Ccal$ is a small category, and $\widehat{\Ccal}$ is the category of presheaves of sets over $\Ccal$. Then a divisible bi-functor $\widehat{\Ccal} \times \Ecal \rightarrow \Fcal$, is the same as a functor $c \mapsto \lambda_c$ from $\Ccal$ to the category of left adjoint functors from $\Ecal$ to $\Fcal$. Using ends and coends notation the correspondence is given by:

\[ \Scal \odot E = \int^{\Ccal} \Scal(c) \times \lambda_c(E) \qquad \Scal \backslash F = \int_{\Ccal} (\lambda_c^*(F))^{\Scal(c)}  \]
\[ F /E = \left( c \mapsto Hom_{\Fcal}(\lambda_c(E),F)\right) \]

\end{itemize}

However, the ``associativity'' properties present on the first two situations appear to play no role in what follows and it is convenient to work in this general setting (with all three categories possibly distinct) for better typing and symmetries. See for example the next lemma. This also allows to consider situations where there is a non-associative ``tensor product'', typically a tensor product that will be associative only up to homotopy, as for example the tensor product of Dendroidal sets.
\end{example}

\begin{lemma}\label{Lem_bifunctor_Sym}
Let $\odot\colon \Ecal_1 \times \Ecal_2 \rightarrow \Ecal_3$ be a divisible bi-functor. Then the two bi-functors:

\[\begin{array}{c c c c c c c}
\Ecal_1 \times \Ecal_3^{op} & \rightarrow  &\Ecal_2^{op} & & \Ecal_2 \times \Ecal_3^{op} & \rightarrow & \Ecal_1^{op}\\
(X_1,X_3) & \mapsto  & (X_1 \backslash X_3) & &  (X_2,X_3) &\mapsto &(X_3/X_2) 
\end{array}
\]

\noindent are both divisible on both sides.

\end{lemma}

\begin{proof}
As $\odot$ is divisible on both sides there are functorial isomorphisms: \[ Hom_{\Ecal_3} ( X_1 \odot X_2, X_3) \simeq Hom_{\Ecal_2}(X_2, X_1 \backslash X_3) \simeq Hom_{\Ecal_1}(X_1,X_3/X_2) \]

By just taking opposite categories, this gives functorial isomorphisms: \[ Hom_{\Ecal_3^{op}}( X_3, X_1 \odot X_2) \simeq Hom_{\Ecal_2^{op}}(X_1 \backslash X_3,X_2) \simeq Hom_{\Ecal_1}(X_1,X_3/X_2) \]

\noindent which shows that $X_1 \backslash X_3$ is divisible on both sides when seen as a functor $\Ecal_1 \times \Ecal_3^{op} \rightarrow \Ecal_2^{op}$. Similarly, we have:
 \[ Hom_{\Ecal_3^{op}}( X_3, X_1 \odot X_2) \simeq Hom_{\Ecal_2}(X_2,X_1 \backslash X_3) \simeq Hom_{\Ecal_1^{op}}(X_3/X_2,X_1) \]

\noindent which shows that $(X_3/X_2)$ is divisible on both sides when seen as a functor $\Ecal_3^{op} \times \Ecal_2 \rightarrow \Ecal_1^{op}$. \end{proof}

\begin{construction}
\label{pushout-productDef} Let $Ar(\Ecal_i)$ be the category of arrows of $\Ecal_i$, whose morphisms are the commutative squares. Following A.~Joyal and M.~Tierney in \cite{joyal2006quasi}, given a bi-functor $\odot\colon \Ecal_1 \times \Ecal_2 \rightarrow \Ecal_3$ we define a bi-functor:
\[ \corner{\odot}\colon  Ar(\Ecal_1) \times Ar(\Ecal_2) \rightarrow Ar(\Ecal_3) \]
called the ``pushout-product'' or ``corner-product''. For $f_1 \colon X_1 \rightarrow Y_1 \in \Ecal_1$ and $f_2\colon X_2 \rightarrow Y_2 \in \Ecal_2$ the map $ f_1 \corner{\odot} f_2 $ is the map:

\[ f_1 \corner{\odot} f_2 \colon  \left( X_1 \odot Y_2 \right) \coprod_{\left( X_1 \odot X_2 \right)} \left( Y_1 \odot  X_2 \right) \rightarrow Y_1 \odot Y_2 \]

\noindent induced by the square:

\[
\begin{tikzcd}[ampersand replacement=\&]
X_1 \odot X_2 \arrow{r}{X_1 \odot f_2 } \arrow{d}{f_1 \odot X_2 } \& X_1 \odot Y_2 \arrow{d}{f_1 \odot Y_2} \\
Y_1 \odot X_2 \arrow{r}{Y_1 \odot f_2} \& Y_1 \odot Y_2 \\
\end{tikzcd}
\]

If $\odot$ is left or right divisible, then $\corner{\odot}$ also is, with the division functors given by $\cornerl{ f_1 \backslash f_3 }$ and  $\cornerl{f_3 / f_2}$ defined as:

\begin{itemize}

\item For $f_1\colon X_1 \rightarrow Y_1 \in \Ecal_1$ and $f\colon X_3 \rightarrow Y_3 \in \Ecal_3$, we denote by $\cornerl{ f_1 \backslash f_3 }$ the map:

\[ \cornerl{ f_1 \backslash f_3 } \colon  Y_1 \backslash X_3 \rightarrow \left( Y_1 \backslash Y_3 \right) \fprod_{\left( X_1 \backslash Y_3 \right) } \left( X_1 \backslash X_3 \right)\]

\noindent induced by the square:

\[
\begin{tikzcd}[ampersand replacement=\&]
Y_1 \backslash X_3 \arrow{r}{f_1 \backslash X_3} \arrow{d}{Y_1 \backslash f_3} \& X_1 \backslash X_3 \arrow{d}{X_1 \backslash f_3} \\
Y_1 \backslash Y_3 \arrow{r}{f_1 \backslash Y_3} \& X_1 \backslash Y_3
\end{tikzcd}
\]

\item Dually, for $f_2\colon  X_2 \rightarrow Y_2 \in \Ecal_2$ and $f_3\colon X_3 \rightarrow Y_3 \in \Ecal_3$ the map $\cornerl{ f_3 / f_2}$ is the map:

\[ \cornerl{ f_3 / f_2} \colon  X_3/Y_2 \rightarrow \left( X_3/X_2 \right) \fprod_{(Y_3/X_2)} \left(Y_3/Y_2 \right)\]

\noindent induced by the square:

\[
\begin{tikzcd}[ampersand replacement=\&]
X_3/Y_2 \arrow{r}{X_3/f_2} \arrow{d}{f_3/Y_2} \& X_3/X_2 \arrow{d}{f_3/X_2} \\
Y_3/Y_2 \arrow{r}{Y_3/f_2} \& Y_3/X_2
\end{tikzcd}
\]

\end{itemize}

\end{construction}

\begin{example}
\label{exemplesOdotprime}Here are some important examples of values of $f \corner{\odot} g$. We are assuming that $\emptyset \odot E_2 \simeq E_1 \odot \emptyset \simeq  \emptyset$ where $\emptyset$ denotes the initial objects of the three categories $\Ecal_1$, $\Ecal_2$ and $\Ecal_3$. This is the case as soon as $\odot$ is divisible.

\begin{itemize}

\item $( 0 \rightarrow X_1 ) \corner{\odot} (0 \rightarrow X_2) = \left( 0 \rightarrow X_1 \odot X_2 \right)$

\item $( 0 \rightarrow X_1) \corner{\odot} (f\colon X_2 \rightarrow Y_2) = (X_1 \odot f \colon  X_1 \odot X_2 \rightarrow X_1 \odot Y_2)$.

\end{itemize}
\end{example}

\begin{remark}
If we consider $X_3/X_2$ and $X_1 \backslash X_3$ as divisible bi-functors $\Ecal_3^{op} \times \Ecal_2 \rightarrow \Ecal_1^{op}$ and $\Ecal_1 \times\Ecal_3^{op} \rightarrow \Ecal_2^{op}$ following Lemma~\ref{Lem_bifunctor_Sym} then their ``corner'' versions are simply $\cornerl{f_3 / f_2}$ and $\cornerl{f_1 \backslash f_3}$. This follows from the explicit formula for $\cornerl{f_3 / f_2}$ and $\cornerl{f_1 \backslash f_3}$ given in \ref{pushout-productDef}. 
\end{remark}

We also have the following easy but very important proposition (also observed by A.~Joyal and M.~Tierney in \cite{joyal2006quasi}):

\begin{prop}
If we denote by $f \pitchfork g$ the fact that $f$ has the left lifting property with respect to $g$, then we have the following equivalences:

\[ f_1 \corner{\odot} f_2 \pitchfork f_3 \Leftrightarrow f_1 \pitchfork \cornerl{ f_3 /f_2 } \]

\noindent as soon as $\odot$ is right divisible, and: 

\[ f_1 \corner{\odot} f_2 \pitchfork f_3 \Leftrightarrow f_2 \pitchfork \corner{ f_1 \backslash f_3 } \]

\noindent as soon as $\odot$ is left divisible.

\end{prop}

More precisely, if we think of a lifting problem (i.e. a square) as a morphism in the arrow category, then a given lifting problem $f_1 \corner{\odot} f_2 \rightarrow f_3$ has a solution if and only if its adjoint transpose $f_1 \rightarrow \cornerl{ f_3 /f_2}$ and $f_2 \rightarrow \cornerl{ f_1 \backslash f_3}$ have solutions, in fact there is even a bijection between the sets of solutions of these different lifting problems.

\begin{definition}\label{def_IfibCof}
\begin{itemize}

\item[]

\item If $I$ and $F$ are sets of maps we write $I \pitchfork F$ for\footnote{Following are usual convention, we mean the existence of a structure producing a solution of each lifting problem of an $i \in I$ against a $f\in F$.}  $i \pitchfork f$ for all $i \in I$ and all $f \in F$. 

\item If $I$ is a set of maps, an arrow $f$ is an $I$-fibration if and only if $I \pitchfork f$. We denote by $I\textsc{-fib}$ the class of $I$-fibrations.

\item An arrow $f$ is an $I$-cofibration if $f \pitchfork I\textsc{-fib}$. We denote by $I\textsc{-cof}$ the class of $I$-cofibrations.

\end{itemize}
\end{definition}

We clearly have $I\textsc{-cof} \pitchfork I\textsc{-fib}$. In situations where the small object argument applies (see Appendix~\ref{section_the_small_object_arguments}) $I$-cofibrations and $I$-fibrations form a weak factorization system. If we assume enough classical logic, or if we are in the ``good'' case of the small object arguments as in \ref{EasySOA}, then $I$-cofibrations are the retracts of transfinite compositions of pushouts (of coproducts) of maps in $I$.

\begin{remark}
Assuming divisibility of $\odot$, the equivalence:

\[ I_1 \corner{\odot} I_2 \pitchfork I_3 \Leftrightarrow I_1 \pitchfork \cornerl{ I_3/I_2} \Leftrightarrow I_2 \pitchfork \cornerl{ I_1 \backslash I_3} \]

\noindent holds as well for sets of maps. We also have the following easy equivalences:

\[ I \pitchfork F \Leftrightarrow F \subset I\textsc{-fib} \Leftrightarrow I\textsc{-cof} \pitchfork F \]

\[ J \subset I\textsc{-cof}  \Leftrightarrow  J \pitchfork I\textsc{-fib} \Leftrightarrow  J\textsc{-cof} \subset I\textsc{-cof}\]

\end{remark}

The following lemma follows formally from these relations:

\begin{lemma}\label{Lem_mainJTcalculus}
Let $\Ecal_1,\Ecal_2$ and $\Ecal_3$ be complete and cocomplete categories endowed with a divisible bi-functor $\odot$ as above, for each $i$ let $I_i$ be a class of arrows in $\Ecal_1$ and assume that $I_1 \corner{\odot} I_2 \subset I_3\textsc{-cof}$ then:

\begin{enumerate}[label=(\roman*)]

\item\label{Lem_mainJTcalculus:I1I2toI3} $I_1\textsc{-cof} \corner{\odot} I_2\textsc{-cof} \subset I_3\textsc{-cof}$

\item\label{Lem_mainJTcalculus:I1backI3toI2} $\cornerl{ I_1\textsc{-cof} \backslash I_3\textsc{-fib} } \subset I_2\textsc{-fib}$

\item\label{Lem_mainJTcalculus:I3/I2toI1} $\cornerl{ I_3\textsc{-fib}/ I_2\textsc{-cof} } \subset I_1\textsc{-fib}$

\end{enumerate}

\end{lemma}

Note that the three stability properties correspond to the ``same'' stability property for the three ways of dualizing the bi-functors $\odot$ following Lemma~\ref{Lem_bifunctor_Sym} (and exchanging cofibrations and fibrations when dualizing a category). This being said, that does not make the proof of these three points symmetric as the assumptions of the lemma are not symmetric under these dualizations.

\begin{proof}
As $I_1 \corner{\odot} I_2 \subset I_3\textsc{-cof}$ we have $I_1 \corner{\odot} I_2 \pitchfork I_3\textsc{-fib}$ hence $I_2 \pitchfork \cornerl{ I_1 \backslash I_3\textsc{-fib} }$ which can be rewritten as  $\cornerl{ I_1 \backslash I_3\textsc{-fib} } \subset I_2\textsc{-fib}$. Similarly $\cornerl{ I_3\textsc{-fib} / I_2 } \subset I_1\textsc{-fib}$.

Now this in turn implies that $I_1\textsc{-cof} \pitchfork \cornerl{ I_3\textsc{-fib} / I_2}$, which is equivalent to $I_2 \pitchfork \cornerl{ I_1\textsc{-cof} \backslash I_3\textsc{-fib} }$ which exactly means that $  \cornerl{ I_1\textsc{-cof} \backslash I_3\textsc{-fib}} \subset I_2\textsc{-fib}$, i.e. $\ref{Lem_mainJTcalculus:I1backI3toI2}$. Point $\ref{Lem_mainJTcalculus:I3/I2toI1}$ follows symmetrically.

Finally, as $\cornerl{ I_1\textsc{-cof} \backslash I_3\textsc{-fib}} \subset I_2\textsc{-fib}$ we have $I_2\textsc{-cof} \pitchfork \cornerl{ I_1\textsc{-cof} \backslash I_3\textsc{-fib} } $, hence $I_1\textsc{-cof} \corner{\odot} I_2\textsc{-cof} \pitchfork I_3\textsc{-fib}$, which gives $\ref{Lem_mainJTcalculus:I1I2toI3}$. \end{proof}

\begin{remark}
\label{Rk::JTcalculus_transnat}A special case of this observation that will be useful later is when $\Ecal_1$ is the category of presheaves over the category $(a \overset{f}{\rightarrow} b)$, with $I_1 = \{f\}$.

This means that there are two left adjoint functors $\lambda_a,\lambda_b \colon  \Ecal_2 \rightrightarrows \Ecal_3$ and a natural transformation $f\colon \lambda_a \rightarrow \lambda_b$. Given an arrow $g\colon X \rightarrow Y \in \Ecal_2$,  $f \corner{\odot} g$ is the arrow:

\[ \lambda_a(Y) \coprod_{\lambda_a(X)} \lambda_a(X) \rightarrow \lambda_b(Y)\]

\noindent and Lemma~\ref{Lem_mainJTcalculus} above says that if the map $f \corner{\odot} i \in I_3\textsc{-cof}$ for all $i \in I_2$ then it also holds for any $i \in I_2\textsc{-cof}$. Applied to $X=0$ this shows in particular that in this case $f_Y\colon \lambda_a(Y) \rightarrow \lambda_b(Y)$ is an $I_3$-cofibration for any $I_2$-cofibrant object $Y$.
\end{remark}

Finally, as our framework of weak model categories suggests to look at lifting properties against only cofibrations between cofibrant objects it is important to know that those are also stable under corner-product:

\begin{lemma}\label{Lem_JTcalculusCofDom}
Let $\odot\colon \Ecal_1 \times \Ecal_2 \rightarrow \Ecal_3$ be a functor divisible on both sides and let $I_1,I_2$ and $I_3$ be classes of maps such that for any $I_1$-cofibration between $I_1$-cofibrant objects $i_1$ and any $I_2$-cofibration between $I_2$-cofibrant objects $i_2$, the arrow $i_1 \corner{\odot} i_2$ is an $I_3$-cofibration. Then for any two such maps $i_1$ and $i_2$, the map $i_1 \corner{\odot} i_2$ also has an $I_3$-cofibrant domain.

\end{lemma}

\begin{proof}
Let $i_1\colon X_1 \rightarrow Y_1$ and $i_2\colon X_2 \rightarrow Y_2$ be as in the lemma, the domain of $i_1 \corner{\odot} i_2$ is:

\[(Y_1 \odot X_2) \coprod_{X_1 \odot X_2} (X_1 \odot Y_2) \]

The map $0 \rightarrow Y_1 \odot X_2$ is the same as $(0 \hookrightarrow Y_1) \corner{\odot} (0 \hookrightarrow X_2)$ (see \ref{exemplesOdotprime}), so it is an $I_3$-cofibration, the map $X_1 \odot X_2 \rightarrow X_1 \odot Y_2$ is $(0 \hookrightarrow X_1) \corner{\odot} (X_2 \hookrightarrow Y_2)$ (see also \ref{exemplesOdotprime}) so it is also an $I_3$-cofibration and the map from the initial object to the pushout above is just the composite of the first map with a pushout of the second, so it is indeed an $I_3$-cofibration, as $I_3$-cofibrations are stable under composition and pushout. \end{proof}

\section{The small object arguments in constructive mathematics}
\label{section_the_small_object_arguments}

The small object argument is the main technique to produce weak factorization systems, and the main reason why we always assumed we had weak factorization systems at our disposal.

\bigskip

It generally starts from a \emph{set} (and not a class) of maps $I$ in a cocomplete category $C$ and, under some conditions that are only there to ensure that some transfinite construction terminate, it shows that any map in $C$ can be factored into a ``$I$-cofibration'' followed by a ``$I$-fibration'' as in Definition~\ref{def_IfibCof}, hence producing a weak factorization system. It also tends to more precisely factor any map as a ``transfinite iterated pushout'' of maps in $I$ followed by an $I$-fibration, hence, using Lemma~\ref{retract_lemma}, it shows that any $I$-cofibration is a retract of such a transfinite composition of pushout of maps in $I$. Although this second aspect is less often true constructively than classically as we will see.

The status of the small object argument regarding constructivity is essentially the same as the special adjoint functor theorem: its not really possible to make it constructive in full generality, but it is for example always true in the internal logic of a Grothendieck topos, or if the category $C$ is a finitely presentable category and the set of maps $I$ are maps between finitely presentable objects then it can be made constructive under mild assumption on the natural number object. In fact it is equivalent to the special adjoint functor theorem, in the sense that any instance of each can be translated into an instance of the other.

The general idea is that we start with a map $f\colon X \rightarrow Y$ and we would like to factorize it as an $I$-cofibration followed by an $I$-fibration. In order do that we consider the set of all possible squares:

\begin{equation}\label{eq:LiftingSquareSOA}\begin{tikzcd}
A_i \ar[r,"a_i"] \arrow{d}{i \in I} & X \arrow{d}{f} \\
B_i  \ar[r,"b_i"] & Y \\
\end{tikzcd} \end{equation}

And we force all these lifting problems to have solutions by defining a new object $X_1$ in which the solutions exist: More precisely, we define $X_1$ to be the object obtained by gluing on $X$ all these maps $A \rightarrow B$, which is achieved by taking a pushout:

\[
\begin{tikzcd}
  \displaystyle \left(  \coprod A_i  \right) \ar[r,"(a_i)"] \ar[d,"\coprod i"]  \ar[dr,phantom,"\ulcorner"{very near end}] &  X \ar[d] \ar[dr,"f"] \\
\displaystyle \left( \coprod B_i\right) \ar[r] \ar[rr,bend right=30,"(b_i)"{description}] & X_1 \ar[r,dotted] & Y  \\
\end{tikzcd}
\]


 of the coproduct of the maps $A_i \rightarrow B_i$ indexed by the set of all square as in (\ref{eq:LiftingSquareSOA}).


This construction gives us a first factorization of $X \rightarrow X_1 \rightarrow Y$. The map $X \rightarrow X_1$ is an $I$-cofibration: In order to construct a diagonal filler in a square:

\[\begin{tikzcd}[ampersand replacement=\&]
X \arrow{d} \arrow{r} \& U \arrow{d}{p \in I\textsc{-Fib}} \\
X_1 \arrow[dotted]{ur} \arrow{r} \& V
\end{tikzcd}\]

\noindent we exactly need to chose a solution to all the lifting problems of $A_i \overset{i}{\rightarrow} B_i$ against $p$ for all the $i$ appearing in the definition of $X_1$. As $p$ is assumed to have chosen lift against all maps in $I$ this is automatic. Moreover the map $X_1 \rightarrow Y$ is ``closer'' to be a fibration in the sense that, by construction, each diagram of the form:

\[\begin{tikzcd}[ampersand replacement=\&]
 \& X \arrow{d} \\
A \arrow{ur} \arrow{r} \arrow{d}{\in I} \& X_1 \arrow{d}{f} \\
B \arrow[dotted]{ur}  \arrow{r} \& Y \\
\end{tikzcd}\]

\noindent has a canonical filling, given by canonical maps $B \rightarrow X_1$ corresponding to the outer square. The idea is then to iterate this construction (possibly through a transfinite construction), if we do this a sufficient ordinal number $\lambda$ of time, and if $Hom(A, \_)$ commutes\footnote{This is why this is called the small object argument. The key assumption is that the object $A$ have to be ``small'' in some sense, like $\lambda$-presentable or $\lambda$-compact.} to co-limits of $\lambda$-chain then any maps from $A$ to $X_{\lambda}$ will factors through one of the $X_{\lambda'}$ for $\lambda' < \lambda$ and hence we will be able to construct diagonal filler of any square, this should make the map $X \rightarrow X_{\lambda} \rightarrow Y$ into a factorization as an $I$-cofibration followed by an $I$-fibration. There are however some issue with constructivity, and some details to be careful of. We Distinguish essentially two, maybe three, version of this construction:

\begin{QuillenSOA} This corresponds to the version described above: we just iterate the construction described above and we stop at some large enough limits ordinal which we will call $\infty$. If the domain of all the arrows in $I$ are finitely presentable object, then  $\infty = \omega$ is a good place to stop. In classical mathematics this works fine, but constructively this is often insufficient: we always get that $X_{\infty} \rightarrow Y$ has the ``existential'' lifting property with respect to all maps in $I$, but not always a chosen lift: the choice of a diagonal filling is completely determined by the choice of a lifting of the map $A \rightarrow X_{\infty}$ to one of the $X_{\alpha}$ but such lifting are not always unique, or canonical:

\begin{itemize}

\item It might not be possible to decide for which level there is lifting $A \rightarrow X_{\alpha}$, so it is not always possible\footnote{Constructively, the fact that every inhabited subset of $\mathbb{N}$ has a smallest elements only holds for complemented (decidable) subsets.} to find a smallest level such that the lifting exists, nor to say that at each stage we only want to take pushout for maps that do not already have a lifting.

\item If the maps $X_i \rightarrow X_{i+1}$ are not monomorphism there might be several lifting $A \rightarrow X_{\alpha}$ at a given level.

\end{itemize}

\noindent and constructively there is in general no way to make the choice of a lift for each map $A \rightarrow X_{\infty}$.

But on the other hand this construction has a big advantage: the map $X \rightarrow X_{\infty}$ is explicitly constructed as a transfinite composition of pushouts of coproducts of maps in $I$. By Lemma~\ref{retract_lemma}, this implies in particular that any $I$-cofibration is a retract of a transfinite composition of pushouts of coproducts of arrows in $I$.

Note that assuming the axiom choice any pushout of a coproduct of maps in $I$ can be seen as a transfinite composition of pushout of maps in $I$ by choosing a well ordering on the indexing set of the coproduct and doing each pushout one after the other. Constructively this is of course not always possible which is why at many places in the paper we talk about ``transfinite composition of pushouts of coproducts of maps in $I$'' where classical references only talk about ``transfinite composition of pushouts of maps in $I$''.

\end{QuillenSOA}

There are essentially two ways to fix this problem in the constructive theory:

\begin{GarnerSOA} This was introduced \cite{garner2009understanding}. This construction differs from the one above in the fact that at each stage we additionally collapse together the maps $B \rightarrow X_{\alpha}$ that comes from squares:

\[\begin{tikzcd}[ampersand replacement=\&]
A \arrow{d}{\in I} \arrow{r} \& X_{\alpha'} \arrow{d} \\
B \arrow{r} \&  Y
\end{tikzcd}\]

\noindent at an earlier stage $\alpha' < \alpha$ and for which the maps $A \rightarrow X_{\alpha}$ coincide. We refer to \cite{garner2009understanding} for the technical details of the construction, but a short way to explain it is that it corresponds to the special adjoint functor theorem applied to construct a left adjoint functor to the forgetful functor from the category of arrows in $\Ccal$ equipped with chosen diagonal filling for each lifting problem against a map in $I$ (with morphisms the square preserving those chosen diagonal filling) to the category of arrows in $\Ccal$.

\label{requierement_GarnerSOA}This version of the construction works constructively as soon as we are able to talk about ordinal large enough so that the process stabilizes (and that it is possible to construct sets by induction on these ordinals). For the case of interest to us, we need:

\begin{enumerate}[label=(\roman*)]
\item There is a natural number object $\mathbb{N}$.
\item In $\Ccal$, pushout of coproduct of maps in $I$ exists\footnote{We can make sense of the ``pushout of a coproduct'' even if the coproduct itself do not exists if needed.}, equalizer exists, and colimit of $\mathbb{N}$-chain exists.
\item For any domain $A$ of an arrow in $I$, the functor $Hom(A, \_)$ commutes to colimits of $\mathbb{N}$-chain. 
\item The induction principle for the natural number object can be used to construct $\mathbb{N}$-chain of objects of $\Ccal$, using a colimit at each step. This is for example the case if the category $\Ccal$ has chosen colimits and we can use the induction principle of the natural number object with value in the set of objects $C$ (which is nontrivial if $C$ is not small). Or $C$ does not have chosen colimits, but we either have the axiom of dependent choice, or the ability to use the induction principle for $\mathbb{N}$ in an ``up to isomorphisms'' version.
\end{enumerate}

This applies to absolutely all the examples mentioned in the paper, as soon as we add the existence and requirement on the natural number object mentioned above, and sometimes the existence of quotient sets (in order to construct pushout) to our framework.

This version of the small object argument has lots of good categorical properties that Quillen's version does not have, but it has one big drawback: it no longer exhibit the map $X \rightarrow X_{\infty}$ as an iterated pushout, as there is also the need to collapse some maps at each stage, and it no longer proves that any $I$-cofibration is a retract of an iterated pushout of coproducts of maps in $I$.

\end{GarnerSOA}

\begin{EasySOA}\label{EasySOA} This corresponds essentially to the situation where the two versions of the small object argument become equivalent.  We add the requirement that for any pushout of coproduct of maps in $I$  (as in the construction of $X \rightarrow X_1$) and any object $A$ the source of one of the maps in $I$ the map of sets:

\[ Hom(A,X) \rightarrow Hom(A,X_1) \]

\noindent is a complemented monomorphism, i.e. it exhibits $Hom(A,X)$ as a complemented (decidable) sub-object of $Hom(A,X_1)$.

This is the case in all the examples treated in the paper. In each case, the reason for this is that these pushout are complemented monomorphisms on the underlying sets, and the objects $A$ are always ``finitely generated'' (in an appropriate sense dependings on the case under consideration), so that the question of whether a map from $A$ to $X_1$ factor in $X$ can be decided\footnote{a finite conjunction of decidable propositions is decidable.} by testing separately for each generators of $A$ if its image in $X_1$ is in $X$ or not.

Under this condition, a map $A \rightarrow X_{\infty}$ admits a unique lift to one of the $X_n$ with $n$ minimal for this property. And so the problem we had with Quillen small objects argument disappears and it can be applied constructively without problems. In this case this gives a constructive proof that cofibrations are retract of (transfinitely) iterated pushout of coproducts of maps in $I$.

Also in this case we can modify Quillen small object argument by saying that at each (finite) step $X_n$, we take the co-product only for the squares for which the map $A \rightarrow X_n$ does not factors into $A_{n-1}$. If we do that, then this version of the small object argument becomes equivalent to Garner's small object argument.

When this special case applies, very similarly to \ref{requierement_GarnerSOA}, the only additional requirement on $\Ccal$ are that:

\begin{enumerate}[label=(\roman*)]

\item $\Ccal$ has a natural number object.
\item Pushout of coproducts of maps in $I$ exists, and colimits of $\mathbb{N}$-indexed chains, whose transition maps are pushouts of coproducts of maps in $I$, exists.
\item If $A$ is the domain of one of the map in $I$ then $Hom(A,\_)$ sends pushout of coproducts of maps in $I$ to decidable inclusion and commutes to colimits of $\mathbb{N}$-chains of the form of the previous point.
\item We can construct objects of $\Ccal$ by induction on the natural number object, with a pushout of a coproduct of maps in $I$ at each step. (see the discussion of condition (iv) in \ref{requierement_GarnerSOA}).
\end{enumerate}

\end{EasySOA}

\begin{remark} In fact, we expect that most instances of the small object argument we use in this paper (in fact, all of them except maybe the one of Section~\ref{subsec:ChainCplex}, which might also require quotient), can be formalized in (the internal logic of) just a cartesian category with parametrized list objects. This is based on the fact that in this case the element of the object obtained by forming the factorization have a (unique) ``syntactic'' description, and it should be possible to formalize such a description using only list objects. But proving this directly requires a lot of work, outside the scope of the present paper. I am hoping to find a more conceptual way to prove such claims in a future work. In the meantime, a more precise account of the formalization of the small object argument internally to a category with enough structure can be found in \cite{swan2018w}.
\end{remark}

\bibliography{/home/henry/Dropbox/Latex/Biblio}{}
\bibliographystyle{plain}

\end{document}